\documentclass[12pt]{article}
\usepackage{amsfonts, amsmath, amssymb,latexsym,amsthm,float,graphicx}
\usepackage{mathrsfs}
\usepackage[usenames]{color}
\usepackage{xcolor}
\usepackage{diagbox,booktabs,makecell}

\usepackage{hyperref}
\hypersetup{
    colorlinks=true,
    linktoc=all,
    linkcolor=black,     
    citecolor=blue,     
}

\newtheorem{theorem}{Theorem}
\newtheorem{conjecture}{Conjecture}
\newtheorem{problem}{Problem}
\newtheorem{corollary}{Corollary}

\newtheorem{proposition}{Proposition}

\newtheorem{remark}{Remark}
\newcommand{\F}{\mathbb F}
\newcommand{\R}{\mathbb R}

\newcommand{\N}{\mathbb N}

\newcommand{\Q}{\mathbb Q}

\newcommand{\ex}{{\rm ex}}

\def\to{\rightarrow}

\def\cH{{\mathcal H}}

\def\cK{{\mathcal K}}
\def\cT{{\mathcal T}}
\def\girth{{\rm girth}}
\def\d{{\rm diam}}
\begin{document}

\title{\bf Some families of graphs, hypergraphs and digraphs
defined by systems of equations}
\author{
Felix Lazebnik \footnote{Department of Mathematical Sciences, University of Delaware, Newark, DE 19716, USA (\tt fellaz@udel.edu)} \;\;\;\;
Ye Wang \footnote{College of Mathematical Sciences, Harbin Engineering University, Harbin 150001, China (\tt ywang@hrbeu.edu.cn)}
}
\date{}
\maketitle
\begin{center}
{\it Dedicated to the memory of Pamela G. Irwin (1962 - 2023)}
\end{center}
\bigskip

\begin{abstract}
The families of graphs defined by a certain type of system of equations over commutative rings have been studied and used since the 1990s.
This survey presents these families and their applications related to graphs, digraphs, and hypergraphs. Some open problems and conjectures are mentioned.
\bigskip

\medskip

\end{abstract}

\newpage
\tableofcontents
\newpage
\section{Introduction}\label{S:intro}
\bigskip

One goal of this survey is to summarize results  concerning families of graphs, hypergraphs and digraphs defined by certain systems of equations.
Another goal is to provide  a comprehensive treatment of, probably,  the best known family of such graphs, denoted by $D(k,q)$. The original results on  these graphs  were scattered among many  papers,  with  the notations not necessarily consistent and reflecting  the origins of these graphs in Lie algebras.  It is our hope that this new exposition will make it easier for those who wish to understand the methods to continue research in the area or find new applications.

For a summary of related results which appeared before 2001,  see Lazebnik and Woldar \cite{LW01}.  One important feature of that article was an attempt of setting simpler notation and presenting results in greater generality.   Let us begin with a quote from   \cite{LW01} (with updated reference labels):

\begin{quotation}
In the last several years some
algebraic constructions of graphs
have appeared in the literature.
 Many of these
constructions  were motivated
by problems from extremal graph theory,
and, as a consequence, the graphs
obtained were primarily of interest
in the context of a
particular extremal problem.
In the case of the graphs
appearing in \cite{Wen91,LU93,LUW94,LUW94a,FLSUW95,LU95,LUW95,LUW96,LUW97,LUW99},
the authors discovered that
they exhibit
many interesting properties {\em beyond}
those which motivated their construction.
Moreover, these
properties tend to
remain present
even when the constructions
are made  far more general.
This latter observation forms the
motivation for our survey.
\end{quotation}

The research conducted since the appearance of  \cite{LW01}  was done in two directions:  attempting  to apply specializations of general constructions to new problems,  and trying to strengthen some old results. Those were collected in a 2017 survey by Lazebnik, Sun and Wang
\cite{LSW17}. Here we make use of both \cite{LW01} and \cite{LSW17}, often following them closely, and include results that were not covered in those surveys.  We also  thought that it is better to make the survey dynamic, and update it regularly.

Before proceeding,
we establish some notations; more will be introduced later.  The missing graph-theoretic definitions can be found in Bollob\'as \cite{Bol98}.
Most  graphs we consider in this survey are undirected,  and without loops or multiple edges.  Sometimes loops will be allowed,  in which case we will state it.  Given a
graph $\Gamma$, we denote the vertex set of $\Gamma$ by $V(\Gamma)$ and the edge set by $E(\Gamma)$. Elements of $E(\Gamma)$ will be written as $xy$, where $x,y\in V(\Gamma)$ are the corresponding adjacent vertices. For a vertex $v$ of $\Gamma$, let $N(v) = N_{\Gamma}(v)$ denote its neighborhood in $\Gamma$.

Though most of the graphs we plan to discuss are defined over  finite fields, many of their  properties hold over commutative rings,  and this is how we proceed.   Let  $R$ be  an arbitrary
commutative ring, different from the zero ring, and with multiplicative identity.
We write $R^n$ to denote the
Cartesian product of $n$ copies of $R$,
and we refer to its elements as {\em vectors}.  For $q=p^e$ with prime $p\ge 2$, let $\F_{q}$  denote the field of $q$ elements.

The survey is organized as follows. In Section \ref{S:mc} we go over the main constructions, and their general properties are discussed in Section \ref{S:prop}. In Section \ref{S:exandappl} we discuss various applications of the specialization of constructions from Section \ref{S:mc}, including recent results on similarly constructed digraphs.
Section 5 deals with constructions for hypergraphs.  In Section \ref{SS:cdkq} we present   a comprehensive treatment of the graphs  $D(k,q)$ and $C\!D(k,q)$,
and survey  new results. In Section \ref{appldkq}   we mention  some applications of the graphs $D(k,q)$ and $C\!D(k,q)$,  and we conclude with a brief discussion on the related work in coding theory and cryptography in Section~\ref{S:codecryp}.

\section{Main constructions}\label{S:mc}

\subsection{Bipartite version }\label{SS:bip}
\bigskip

Let $f_i: R^{2i-2}\to R$, $2\le i\le n$,  be
arbitrary functions on $R$ of 2,4, $\ldots$, $2n-2$ variables.
We define the bipartite graph
$B\Gamma_n =
B\Gamma (R;f_2,\ldots, f_n)$, $n\ge 2$,
as follows.
The set of vertices $V(B\Gamma_n)$ is the
disjoint union of two copies of
$R^n$, one denoted by $P_n$ and the other by $L_n$.
Elements of $P_n$ will be called {\em points}
and those of $L_n$
{\em lines}.
In order to distinguish
points from lines
we introduce the use of parentheses and brackets:
if $a\in R^n$, then $(a)\in P_n$ and
$[a]\in L_n$.
We define edges of $B\Gamma_n$ by declaring a point
$(p)=(p_1,p_2,\ldots,p_n)$ and a line
$[l]=[l_1,l_2,\ldots,l_n]$   adjacent
if and only if the following
$n-1$ relations on their coordinates hold:

\begin{equation}\label{Bipmain0}
\begin{aligned}
p_2 + l_2 &= f_2(p_1,l_1),\\
p_3 + l_3 &= f_3(p_1,l_1,p_2,l_2),\\
\quad \ldots & \quad\ldots \\
p_n + l_n &= f_n(p_1,l_1,p_2,l_2, \ldots, p_{n-1},l_{n-1}).
\end{aligned}
\end{equation}

\smallskip
\noindent For a function $f_i: R^{2i-2}\to R$,
we define $\overline{f_i}: R^{2i-2}\to R$
by the rule
\[
\overline{f_i}(x_1,y_1, \ldots, x_{i-1}, y_{i-1})=
f_i(y_1,x_1, \ldots, y_{i-1}, x_{i-1}).
\]
We call $f_i$ {\em symmetric}
if the functions $f_i$ and $\overline{f_i}$ coincide.
The following is trivial to prove.

\begin{proposition}\label{P:dualiso}
The graphs
$B\Gamma(R;f_2,\ldots, f_n)$
and
$B\Gamma(R;\overline{f_2},\ldots, \overline{f_n})$
are isomorphic, an explicit isomorphism being
given by
$\varphi: (a)\leftrightarrow [a]$.
\end{proposition}

We now define our second fundamental family of graphs
for which we require
that all functions are symmetric.

\subsection{Ordinary version }\label{SS:ord}
\bigskip

Let $f_i:R^{2i-2}\to R$ be symmetric for all $2\le i\le n$.
We define $\Gamma_n = \Gamma(R;f_2,\ldots, f_n)$ to be the graph with vertex set
$V(\Gamma_n)=R^n$, where distinct vertices (vectors) $a =\langle a_1,a_2,\ldots,a_n\rangle$ and $b =\langle b_1,b_2,\ldots,b_n\rangle$ are adjacent if and only
if the following $n-1$ relations on their coordinates hold:
\begin{equation}\label{E:eqn12}
\begin{aligned}
a_2 + b_2 &= f_2(a_1,b_1),\\
a_3 + b_3 &= f_3(a_1,b_1,a_2,b_2),\\
\quad\ldots & \quad\ldots  \\
a_n + b_n &= f_n(a_1,b_1,a_2,b_2,\ldots, a_{n-1},b_{n-1}).
\end{aligned}
\end{equation}
For the graphs $\Gamma_n$ our requirement that all functions $f_i$ are symmetric is necessary to ensure that adjacency is  symmetric. Without this condition one obtains not graphs, but digraphs.
It is sometimes beneficial to allow loops in~$\Gamma_n$, i.e., considering $a_i=b_i$ for all $i$ and satisfying (\ref{E:eqn12}).
\medskip

\subsection{{Uniform $k$-partite hypergraphs}}  \label{hyper}
\bigskip

In this survey, a {\it hypergraph} ${\mathcal H}$ is a family of
distinct subsets of a finite set. The members of ${\mathcal H}$ are
called {\it edges}, and the elements of $V({\mathcal H}) = \bigcup_{E
\in {\mathcal H}} E$ are called {\it vertices}. If all edges in ${\mathcal H}$ have size $r$, then ${\mathcal H}$ is called an {\it $r$-uniform}
hypergraph or, simply, {\it $r$-graph}. For example, a 2-graph is
a graph in the usual sense. A vertex $v$ and an edge $E$ are
called {\it incident} if $v\in E$. The {\it degree} of a vertex
$v$ of ${\mathcal H}$ is the number of edges of
${\mathcal H}$ incident with $v$. An $r$-graph ${\mathcal H}$ is {\it
$r$-partite} if its vertex set $V({\mathcal H})$ can be colored in $r$
colors in such a way that no edge of ${\mathcal H}$ contains two
vertices of the same color. In such a coloring, the color classes
of $V({\mathcal H})$ -- the sets of all vertices of the same color --
are called {\it parts} of ${\mathcal H}$. We refer the reader to
Berge \cite{Ber89,Ber98} for additional background on
hypergraphs.

In \cite{LM02}, Lazebnik and Mubayi generalized Theorem \ref{T:decomp}
to edge-decompositions of complete uniform $r$-partite
hypergraphs and complete uniform hypergraphs, respectively. The following comment is from \cite{LM02}.
\begin{quotation}
Looking back, it
is fair to say that most of these generalizations turned out to
be rather straightforward and natural. Nevertheless it took us
much longer to see this than we originally expected: some ``clear"
paths led eventually to nowhere, and  several technical steps
presented considerable challenge even after the ``right"
definitions had been found.
\end{quotation}

Though the following definitions can be made over an arbitrary commutative ring $R$, different from the zero, and with multiplicative identity, we will restrict ourselves to the case where $R$ is a finite field.

As before, let ${\F_q}$ be the
field of $q$ elements. For integers $d, i,r\ge 2$,  let $f_i: {\F_q^{(i-1)r}}\to {\F_q}$ be a function. For $x^i = (x_1^i,
\ldots, x_d^i)\in \F_q^d$,  let $(x^1,\ldots ,x^i)$ stand
for $(x_1^1, \ldots , x_d^1, x_1^2, \ldots , x_d^2, \ldots ,
x_1^i, \ldots , x_d^i)$.

Suppose $d,k,r$ are integers and $2\le r\le k$, $d\ge 2$. First we
define a $k$-partite $r$-graph
 $\cT = \cT (q,d,k,r, f_2,f_3,\ldots , f_d)$.
Let the vertex set $V(\cT)$ be a disjoint union of sets, or color
classes,  $V^1, \ldots, V^k$, where each $V^j$ is a copy of
$\F_q^d$. By $a^j = (a_1^j,a_2^j, \ldots , a_d^j)$ we
denote an arbitrary vertex from $V^j$. The edge set $E(\cT)$ is
defined as follows: for every $r$-subset $\{i_1,\ldots , i_r\}$ of
$\{1,\ldots, k\}$ (the set of colors), we consider the family of
all $r$-sets of vertices $\{a^{i_1}, \ldots ,a^{i_r} \}$, where
each $a^j\in V^j$, and such that the following system of $d-1$
equalities hold:

\begin{equation}\label{Bipmain}
\begin{aligned}
\sum_{j=1}^r a_2^{i_j}  &= f_2(a_1^{i_1},\ldots ,a_1^{i_r}),\\
\sum_{j=1}^r a_3^{i_j}  &= f_3(a_1^{i_1},\ldots ,a_1^{i_r}, a_2^{i_1},\ldots ,a_2^{i_r}),\\
\quad \ldots & \quad\ldots \quad\ldots \quad\ldots \quad\ldots \\
\sum_{j=1}^r a_d^{i_j}  &= f_d(a_1^{i_1},\ldots ,a_1^{i_r},
a_2^{i_1},\ldots ,a_2^{i_r},
\dots ,a_{d-1}^{i_1},\ldots ,a_{d-1}^{i_r}).\\
\end{aligned}
\end{equation}

The  system (\ref{Bipmain}) can also be used to define another
class of $r$-graphs, $\cK = \cK (q,d,r,f_2,f_3,\ldots , f_d)$, but
in order to do this, we have to restrict the definition to only
those functions $f_i$ which satisfy the following symmetry
property: for every permutation $\pi $ of $\{1,2,\ldots, i-1\}$,
$$f_i(x^{\pi(1)},\ldots, x^{\pi(i-1)}) = f_i(x^1,\ldots ,x^{i-1}).$$

Then let the vertex set $V(\cK)= \F_q^d$, and let the edge
set $E(\cK)$ be the family of all $r$-subsets $\{a^{i_1},\ldots ,
a^{i_r} \}$ of vertices which satisfy system (\ref{Bipmain}). We
impose the symmetry condition on the $f_i$ to make the definition
of an edge independent of the order in which its vertices are listed.
Note that $\cK$ can be also viewed as a $q^d$-partite $r$-graph, each
partition having one vertex only. If $d=r$, then $\{i_1,\ldots ,
i_r\}= \{1,\ldots , d\}$.
\medskip

\section{General properties of graphs $B\Gamma_n$, $\Gamma_n$ and hypergraphs $\cT$, $\cK$ } \label{S:prop}
\bigskip

The goal of this section is to state the properties of  $B\Gamma_n= B\Gamma(R; f_2, \ldots, f_n)$ and
$\Gamma_n = \Gamma(R; f_2, \ldots, f_n)$,  which are {\it independent} of the choice of $n$, $R$, and the functions $f_2, \ldots, f_n$. Specializing these parameters, one can obtain some additional properties of the graphs. All proofs can be found in \cite{LW01} or references therein, and we omit them, with the exception of Theorem \ref{T:ncc} below. Though trivial, it is of  utmost  importance for understanding the graphs.

\subsection{Degrees and neighbor-complete colorings }\label{SS:ncc}
\bigskip

One of the most important properties of the graphs
$B\Gamma_n$ and $\Gamma_n$
defined in the previous section is
the following.  In the case of $\Gamma_n$ we do allow loops, and assume that a loop on a vertex  adds 1 to the degree of the vertex.

\begin{theorem}\label{T:ncc}
For every vertex $v$ of $B\Gamma_n$ or of $\Gamma_n$,  and every
$\alpha \in R$,  there exists a unique
neighbor of $v$  whose first coordinate is $\alpha$.

If $|R|=r$,  all graphs $B\Gamma_n$ or $\Gamma_n$ are $r$-regular.
If 2 is a unit in  $R$, then $\Gamma_n$ contains exactly $r$ loops.
\end{theorem}

\begin{proof}
Fix a vertex $v\in V(B\Gamma_n)$, which
we may assume is a point $v=(a)\in P_n$ (if $v\in L_n$, the argument is similar).
Then for any $\alpha\in R$, there is a unique line
$[b]\in L_n$ which is adjacent to $(a)$ and for which $b_1=\alpha$.
Indeed, with respect to the unknowns $b_i$ the system (\ref{Bipmain0}) is
triangular, and each $b_i$ is uniquely determined from the values of
$a_1,\dots,a_i,b_1,\dots b_{i-1}$, $2\le i\le n$.

This implies that if $|R|=r$, then both $B\Gamma_n$ and $\Gamma_n$ are $r$-regular.
A vertex
$a\in V(\Gamma_n)$ has a loop on it if and only if
it is  of the form
$\langle a_1,a_2,\dots,a_n\rangle$,
where
\[
a_i= \frac{1}{2}
f_i(a_1,a_1,\ldots,a_{i-1},a_{i-1}),
\;\; 2\le i\le n.
\]
Hence,  there are exactly $r$ loops. Erasing them we obtain a simple graph with $r$ vertices of degree $r-1$ and $r^n-r$ vertices of degree $r$.
\end{proof}

Based on this theorem,  it is clear that each of the graphs $B\Gamma_n$ and $\Gamma_n$ allows a vertex coloring by all elements of $R$ such that
the neighbors of every vertex are colored in all possible colors: just color every vertex by its first coordinate.  These colorings are never proper,  as the color of a vertex is the same as the color of exactly one of its neighbors.
Such  colorings  were introduced by
Ustimenko in \cite{Ust98} under the name of ``parallelotopic" and
further explored by Woldar \cite{Wol02} under the name of
``rainbow",  and in \cite{LW01} under the name of ``neighbor-complete colorings", which we adopt  here.
In \cite{Ust98},  some group theoretic constructions
of graphs possessing neighbor-complete colorings are given; in
\cite{Wol02} purely combinatorial aspects of such colorings
are considered.  Non-trivial examples of graphs possessing
neighbor-complete colorings  are not easy to construct.
Remarkably, $B\Gamma_n$ and $\Gamma_n$ always admit them.

Similar statements,
with obvious modifications, hold for
$B\Gamma_n[A,B]$ and $\Gamma_n[A]$, and we
leave such verification to the reader.
\medskip

\subsection{Special induced subgraphs}\label{SS:indsub}
\bigskip

Let $B\Gamma_n$  be the bipartite graph defined in Section
\ref{SS:bip}, and let
$A$ and $B$ be arbitrary subsets of $R$.
We set
\begin{align}
P_{n,A} &=\{(p) =(p_1,p_2,\ldots,p_n) \in P_n\mid p_1\in A\}\notag, \\
L_{n,B} &=\{[l] = [l_1,l_2,\ldots,l_n]\in L_n\mid l_1\in B\}\notag,
\end{align}
and define $B\Gamma_n[A,B]$
to be the subgraph of $B\Gamma_n$
induced on the set of vertices
$P_{n,A}\cup L_{n,B}$.
Since we restrict the range of only
the first coordinates of vertices of
$B\Gamma_n$, $B\Gamma_n[A,B]$
can alternately be described as the
bipartite graph with bipartition
$P_{n,A}\cup L_{n,B}$
and adjacency relation
as given in (\ref{Bipmain0}).
This is a valuable observation
as it enables one
to ``grow'' the graph
$B\Gamma_n[A,B]$ directly, without
ever having to construct
$B\Gamma_n$.
In the case where $A=B$,
we shall abbreviate
$B\Gamma_n[A,A]$
 by
$B\Gamma_n[A]$.

Similarly, for arbitrary $A\subseteq R$ we
define $\Gamma_n[A]$ to be the subgraph of $\Gamma_n$
induced on the set $V_{n,A}$ of all vertices having
respective first coordinate
from $A$.
Again, explicit construction
of $\Gamma_n$ is not essential
in constructing $\Gamma_n[A]$;
the latter graph is obtained
by applying the adjacency relations
in (\ref{E:eqn12}) directly
to $V_{n,A}$.
(Note that when $A=R$ one has
$B\Gamma_n[R] = B\Gamma_n$ and
$\Gamma_n[R] = \Gamma_n$.)
\medskip

\subsection{Covers and lifts}\label{SS:cover}
\bigskip

The notion of a covering for graphs is
analogous to the one in topology.
We call $\overline{\Gamma}$ a
{\em cover}
of $\Gamma$
(and we write $\overline{\Gamma}\to \Gamma$)
if there exists a surjective mapping
$\theta:V(\overline{\Gamma}) \to V(\Gamma),
\; \overline{v}\mapsto v$,
which satisfies the two
conditions:

\begin{enumerate}
\item[(i)]
$\theta$ preserves adjacencies,
i.e., $uv\in E(\Gamma)$ whenever $\overline{u}\,
\overline{v}\in E(\overline{\Gamma})$;

\item[(ii)]
For any vertex $\overline{v}\in V(\overline{\Gamma})$,
the restriction of
$\theta$ to $\overline{N}(\overline{v})$
is a bijection between $\overline{N}(\overline{v})$
and $N(v)$.
\end{enumerate}

Note that our condition (ii)
ensures that $\theta$ is degree-preserving;
in particular,
any cover of an $r$-regular graph  is again
$r$-regular.  If $\overline{\Gamma}$ is a cover of $\Gamma$, we also say that $\overline{\Gamma}$ is a {\em lift} of $\Gamma$.

For $k< n$,
denote by $\eta=\eta(n,k)$
the mapping
$R^n\to R^k$
 which projects
$v\in R^n$ onto its
$k$ initial coordinates,
viz.
\[
v=\langle v_1, v_2, \dots, v_k,
\dots v_n\rangle
\mapsto v=
\langle v_1, v_2, \dots, v_k\rangle,
\]

Clearly, $\eta$ provides a mapping
$V(\Gamma_n)\to V(\Gamma_k)$,
and its restriction to
$V_{n,A}=A\times R^{n-1}$
gives mappings
$V(\Gamma_n [A]) \to V(\Gamma_k[A])$.
In the bipartite case, we further
impose that $\eta$ preserves
vertex type, i.e.~that
\[
\begin{gathered}
(p)= (p_1, p_2, \dots, p_k,
\dots p_n)
\mapsto (p)=
(p_1, p_2, \dots, p_k),\\
[l]= [l_1, l_2, \dots, l_k,
\dots l_n]
\mapsto [l]=
[l_1, l_2, \dots, l_k].
\end{gathered}
\]

Here, $\eta$ induces, in obvious fashion,
the mappings
$V(B\Gamma_n [A])\to V(B\Gamma_k[A])$.

In what follows,
the functions
$f_i$ ($2\le i\le n$)
for the graphs
$B\Gamma_n[A]$
are assumed to
be arbitrary, while those
for
$\Gamma_n[A]$,
continue, out of necessity,
to be assumed symmetric. The proof of the following theorem is easy and can be found in \cite{LW01}.

\begin{theorem}  \label{T:cover}
For every $A\subseteq R$, and every $k,n$,
$2\le k < n$,
$$B\Gamma_n[A] \rightarrow B\Gamma_k[A].$$
No edge of
$\Gamma_n[A]$
projects to a loop of
$\Gamma_k[A]$ if and only if
$$\Gamma_n[A]\rightarrow \Gamma_k[A].$$
\end{theorem}

\begin{remark}\label{rem10.1}
{\rm If a  graph $\Gamma$ contains cycles, its girth, denoted by $girth (\Gamma)$, is  the length of its shortest cycle.
One important consequence of Theorem~\ref{T:cover},
particularly amenable to
girth related Tur\'an-type problems in extremal graph theory,
is that the girth of a graph is not greater than the girth of its cover.  In particular, the girth of
$B\Gamma_n$ or $\Gamma_n$ is a non-decreasing
function of $n$. More precisely, for $2\le k\le n$,
$$girth (B\Gamma(R; f_2, \ldots, f_k)) \le girth (B\Gamma(R; f_2, \ldots, f_k\ldots, f_n)),$$
and similarly for $B\Gamma_n[A]$ or $\Gamma_n[A]$.
}
\end{remark}

A relation of $B\Gamma_n$ to voltage lifts will be discussed in Section \ref{voltagelift}.
\medskip

\subsection{Embedded spectra}\label{SS:embed}
\bigskip

The spectrum
$spec(\Gamma)$ of a graph $\Gamma$
is defined to be
the multiset of eigenvalues of
its adjacency matrix.
One important property of covers
discussed in Section~\ref{SS:cover} is that
the spectrum of any graph embeds
(as a multiset, i.e., taking into
account also the multiplicities of
the eigenvalues) in the
spectrum of its cover.
This result can be proven in
many ways, for example as a consequence
of either Theorem 0.12 or Theorem 4.7, both of Cvetkovi\'c, Doob and Sachs \cite{CDS80}.
As an immediate consequence of this fact
and Theorem~\ref{T:cover}, we obtain the following result.

\begin{theorem}\label{T:spec}
Assume $R$ is finite and let
$A\subseteq R$. Then for each $k,n$, $2\le k< n$,
\[
spec(B\Gamma_k[A])\subseteq
spec(B\Gamma_n[A]).
\]
For the graphs $\Gamma_n[A]$, one has $spec(\Gamma_k[A])\subseteq
spec(\Gamma_n[A])$ provided
no edge of
$\Gamma_n[A]$
projects to a loop of
$\Gamma_k[A]$.
\end{theorem}

\subsection{Edge-decomposition of $K_n$  and $K_{m,m}$}
\label{SS:edgedec}
\bigskip

The results of this subsection  were motivated by a question of Thomason \cite{Tho02},  who asked whether
$D(k,q)$
(which will be defined later in this survey) edge-decomposes $K_{q^k,q^k}$.

Let $\Gamma$ and $\Gamma'$ be graphs.
An {\em edge-decomposition of
$\Gamma$ by $\Gamma'$}
is a collection ${\mathcal C}$ of subgraphs
of $\Gamma$, each isomorphic to $\Gamma'$,
such that
$\{E(\Lambda)\mid \Lambda\in
{\mathcal C}\}$ is a partition of
$E(\Gamma)$.

We also say in this case that
{\em $\Gamma'$ decomposes $\Gamma$}.
It is customary to refer to the
subgraphs $\Lambda$ in ${\mathcal C}$ as
{\em copies} of $\Gamma'$, in which case
one may envision an
edge-decomposition of
$\Gamma$ by $\Gamma'$
as a decomposition of
$\Gamma$ into
edge-disjoint copies of
$\Gamma'$.

As usual, let $K_n$  denote
the complete graph
on $n$ vertices, and
$K_{m,n}$ the complete bipartite
graph with partitions of sizes $m$ and $n$.
The questions of decomposition of $K_n$ or  $K_{m,n}$ into copies of a graph $\Gamma'$ are classical in graph theory and have been of interest for many years. In many studied cases $\Gamma'$ is a matching, or a cycle, or a complete graph,  or a complete bipartite graph, i.e.,  a graph with a rather simple structure.
In contrast, the structure of $B\Gamma_n$ or $\Gamma_n$, is usually far from simple.
In this light the  following theorem from \cite{LW01} is a bit surprising.

\begin{theorem}\label{T:decomp}{\rm(\cite{LW01})}
Let $|R|=r$.  Then
\begin{enumerate}
\item
$B\Gamma_n$ decomposes
$K_{r^n,r^n}$.

\item
$\Gamma_n$ (with no loops) decomposes
$K_{r^n}$ if $2$  is a unit in $R$.
\end{enumerate}
\end{theorem}
\medskip

\subsection{{Edge-decomposition of complete $k$-partite $r$-graphs
and complete $r$-graphs}}
\bigskip

We begin by  presenting several parameters of hypergraphs
\[
\cT = \cT (q,d,k,r, f_2,\ldots , f_d)\;\;\text{and}\;\;
\cK = \cK (q,d,r,f_2,\ldots , f_d).
\]
The proof of the theorem below is similar to the one for $B\Gamma$ and $\Gamma$.

\begin{theorem}\label{D:degree} {\rm (\cite{LM02})}
Let $q,d,r,k $ be integers, $2\le r\le k$, $d\ge 2$, and  $q$ be a
prime power. Then
\begin{enumerate}
\item
$\cT$ is a regular $r$-graph of order $kq^d$ and size ${{k}\choose
{r}}q^{dr-d+1}$. The degree of each  vertex is ${{k-1}\choose
{r-1}}q^{dr - 2d + 1}$.

\item
For odd $q$, $\cK$ is an $r$-graph of order  $q^d$ and size
${1\over {q^{d-1}}}{q^d\choose r}$.
\end{enumerate}
\end{theorem}

For $r=2$ and $q$ odd, the number of loops in $\Gamma_n$ could be easily counted. Removing them leads to a bi-regular graph, with some vertices having degree $q$ and others having degree $q-1$.  In general this is not
true for $r\ge 3$. Nevertheless, it is true when $q = p$ is an odd
prime, and the precise statement follows. In this case, the
condition $(r,p) = 1$ implies $(r-1, p)=1$, which allows one to  prove
the following theorem by induction on $r$.

\begin{theorem}\label{D:degrees} {\rm (\cite{LM02})}
Let $p,d,r $ be integers, $2\le r < p$, $d\ge 2$, and $p$ be a
prime. Then $\cK$ is a bi-regular $r$-graph of order $p^d$ and
size ${1\over {p^{d-1}}}{p^d\choose r}$. It contains $p^d - p$
vertices of degree $\Delta $ and $p$ vertices of degree $\Delta +
(-1)^{r+1}$, where $\Delta  = {1\over
{p^{d-1}}}\bigl({{p^d-1}\choose {r-1}} + (-1)^r
\bigr)$.
\end{theorem}

We now turn to  edge-decompositions of hypergraphs.
 Let $\cH$ and $\cH '$ be
hypergraphs. An {\em edge-decomposition of $\cH $ by $\cH '$} is
a collection ${\mathcal P}$ of subhypergraphs of $\cH$, each
isomorphic to $\cH '$, such that $\{E({\mathcal X} )\mid
{\mathcal X}\in {\mathcal P}\}$ is a partition of $E(\cH )$.
We also say in this case that {\em $\cH '$ decomposes $\cH$}, and
 refer to the hypergraphs from ${\mathcal
P}$ as {\em copies} of $\cH '$.

Let $T^{(r)}_{kq^d}$, $2\le r\le k$, $d\ge 1$, denote the complete
$k$-partite $r$-graph with each partition class containing $q^d$
vertices. This is a regular $r$-graph of order $kq^d$ and size
${k\choose r}q^{dr}$, and the degree of each  vertex is
${{k-1}\choose {d-1}}q^{dr -d}$.

As before, let  $K^{(r)}_{q^d}$ denote the complete $r$-graph on
$q^d$ vertices.
The following theorem
 holds for {\it arbitrary}  functions $f_2,\ldots , f_r$.
The proof of the theorem  below is similar to the one for 2-graphs from
\cite{LW01}.

\begin{theorem}\label{T:decomp1} {\rm (\cite{LM02})}
Let $q,d,r,k $ be integers, $2\le r\le k$, $d\ge 2$, and  $q$ be a
prime power. Then
\begin{enumerate}
\item
$\cT = \cT (q,d,r,k, f_2,\ldots , f_d)$ decomposes
$T^{(r)}_{kq^d}$.

\item
$\cK = \cK (q,d,r, f_2,\ldots , f_d)$ decomposes
$K^{(r)}_{q^d}$  provided that $q$ is odd and $(r,q) = 1$.
\end{enumerate}
\end{theorem}
\medskip

\section{Specializations of general constructions and their applications}\label{S:exandappl}
\bigskip

In this section we  survey some applications of the graphs $B\Gamma_n$ and $\Gamma_n$, and of similarly constructed hypergraphs and digraphs.  In most instances, the graphs considered are
specializations of $B\Gamma_n$ and  $\Gamma_n$, with $R$ taken to be the finite
field $\mathbb {F}_q$  and the functions $f_i$
chosen in such  a way as to ensure the resulting graphs have additional desired  properties. In particular, many applications  deal with the existence of graphs (hypergraphs) of a fixed order, having many edges and not containing certain subgraphs (subhypergraphs). Our next section provides related terminology and basic references.
\medskip

\subsection{Generalized polygons}\label{SS:poly}
\bigskip

A {\it generalized} $k$-{\it gon} of order $(q,q)$, for $k\ge 3$ and
$q\ge 2$, denoted $\Pi ^k_q$,  is a $(q+1)$-regular bipartite graph of girth $2k$ and diameter $k$. It is
easy to argue that in such a graph each partition contains $n^k_q =
q^{k-1} + q^{k-2} + \ldots + q + 1$ vertices (for information on
generalized polygons, see,  e.g., Van Maldeghem \cite{VMal98}, Thas \cite {Tha95} or Brouwer, Cohen and Neumaier \cite{BCN89}).  It follows from a theorem
by Feit and Higman~\cite{FH64} that if $\Pi ^k_q$ exists, then
$k \in \{3,4,6\}$.  For each of these $k$, $\Pi_q^k$ is known to
exist only for arbitrary prime power~$q$. In the case $k = 3$, the
graph  is better known as the point-line incidence graph of a projective plane of order $q$; for $k=4$ -- as the  generalized quadrangle of order $q$,  and for $k=6$, as the generalized hexagon of order $q$.
Fixing an edge in the graph $\Pi_q^k$,  one can consider a subgraph in $\Pi_q^k$ induced by all vertices at distance at least $k-1$ from the edge. It is easy to argue that the resulting graph is $q$-regular, has girth $2k$ (for $q\ge 4$) and diameter $2(k-1)$ (for $q\ge 4$). We refer to this graph as a {\it biaffine part of $\Pi_q^k$}  (also known as an affine part). Hence, a biaffine part is a $q$-regular induced subgraph of $\Pi_q^k$  having  $q^{k-1}$ vertices in each partition.  Deleting all vertices of  a biaffine part  results in a spanning tree of $\Pi_q^k$ with each inner vertex of degree $q+1$.

  If $\Pi_q^k$ is edge-transitive,  then all its  biaffine parts are isomorphic and we can speak about {\it the}  biaffine part,  and denote it by  $\Lambda_q^k$.     Some classical generalized polygons are known to be edge-transitive.  It turns out that their biaffine parts can be represented as the graphs $B\Gamma_n$:
  \medskip
\begin{eqnarray}\label{biaffineparts}
  \Lambda_q^3 & \text {as} &B\Gamma_2(\F_q; p_1l_1) \label{biaffineparts3},\\
  \Lambda_q^4 & \text {as} &B\Gamma_3(\F_q; p_1l_1, p_1l_2) \cong B\Gamma_3(\F_q; p_1l_1, p_1l_1^2) \label{biaffineparts4},\\
   \Lambda_q^6 & \text {as} &B\Gamma_5(\F_q; p_1l_1, p_2l_1, p_3l_1, p_2l_3 - p_3l_2).\label{biaffineparts6}
\end{eqnarray}

We wish to mention that many other representations of these graphs are possible, and some are more convenient than others when we study particular properties of the graphs. The description of $\Lambda_q^6$ above is due to Williford \cite{Wil12}.

Presentations of $\Lambda_q^k$  in terms of systems of equations  appeared in the literature in different ways, firstly  as an attempt to coordinatize incidence  geometries $\Pi_q^k$, see Payne \cite{Pay70},  \cite{VMal98} and references therein.

Another approach, independent of the previous, is based on the work of Ustimenko, see \cite{Ust90,Ust90a,Ust91},  where incidence structures in group geometries, which were initially used to present generalized polygons,  were described as relations in the corresponding affine Lie algebras.   Some details and examples of related computations can be found in Lazebnik and Ustimenko \cite{LU93},  in Ustimenko and Woldar \cite{UW03}, in Woldar \cite{Wol10}, in Terlep and Williford \cite{TW12} and in more recent work by Yang, Sun and Zhang \cite{YSZ22}.

The descriptions of biaffine parts $\Lambda_q^k$ of the classical $k$-gons $\Pi_q^k$ via the graphs $B\Gamma_{k-1}$ above, suggested to generalize the latter to the values of $k$ for which no generalized $k$-gons exist. The property of nondecreasing  girth of the graphs $B\Gamma_n$ that we mentioned in Remark \ref{rem10.1} of Section \ref{SS:cover} turned out to be  fundamental  for  constructing families of graphs with many edges and without cycles of certain lengths, and in particular, of large girth.  We describe these applications in Section \ref{SS:largeg}.

The graphs  $B\Gamma_n$ can also be used to attempt to construct new generalized $k$-gons ($k\in \{3,4,6\}$) via the following logic: first construct a graph $B\Gamma_{k-1}$ of girth $2k$ and diameter $2(k-1)$,  and then try to ``attach a tree" to it.
In other words, construct a $\Lambda_{k-1}$-like graph, preferably  not isomorphic to one coming from $\Pi_q^k$,  and then, if possible,  extend it to a generalized $k$-gon.
For $k=3$, the extension  will always work.  Of course,  this approach has an inherited restriction on the obtained $k$-gon,  as the automorphism group of any graph $B\Gamma_{k-1}$ contains a subgroup isomorphic to the additive group of the field $\F_q$ (or a ring $R$).
This subgroup is  formed by the following $q$ maps $\phi_a$, $a\in \F_q$:
\begin{align}
\phi_a:  (p_1,p_2,  \ldots, p_{k-1}) &\mapsto (p_1,p_2, , \ldots, p_{k-1} +a),\label{autolastcoordp}\\
       [l_1, l_2, \ldots, l_{k-1}] &\mapsto [l_1, l_2, \ldots, l_{k-1} - a].\label{autolastcoordl}
\end{align}

Lazebnik and Thomason used this approach in \cite{LT04} to construct planes of order $9$ and, possibly new planes of order $16$.   The planes they  constructed all possessed a special group of automorphisms  isomorphic to the additive group of the field, but they were not always translation planes. Of the four planes of order 9, three admit the
additive group of the field
$\F_9$  as a group of translations, and the
construction yielded all three. The known planes of order 16 comprise four self-dual planes and eighteen other planes (nine dual pairs); of these, the method
gave three of the four self-dual planes and six of the nine dual pairs, including
the sporadic (not translation) plane of Mathon. Some attempts to construct new  generalized quadrangles are  discussed in Section \ref{monographs}.
\medskip

\subsection{Tur\'an-type extremal problems }\label{SS:largeg}
\bigskip

For more on this  subject,  see the book by Bollob\'as \cite{Bol78}  and the survey by F\"{u}redi and Simonovits  \cite{FS13}.

Let $\mathcal F$ be a family of graphs.
By $ex(\nu,\mathcal F)$ we denote the
greatest number of edges  in a graph  on $\nu$ vertices
which contains no
subgraph isomorphic to a graph  from $\mathcal F$, and $ex(\nu,\mathcal F)$ is referred to as the {\it Tur\'an number} of $\mathcal F$.  Determining the Tur\'an number $ex(\nu,\mathcal F)$ for a fixed  $\mathcal F$,  is called a {\it Tur\'an-type extremal graph problem}.  Graphs from $\mathcal F$ are called {\it forbidden} graphs,  and if a graph $G$ does not contain any graph from $\mathcal F$ as a subgraph, we say that $G$ is {\it $\mathcal F$-free}.

If $\mathcal F$ contains no bipartite graphs, the leading term in the asymptotic of $ex(\nu,\mathcal F)$ is given by the celebrated Erd\H os--Stone--Simonovits Theorem, see \cite{Bol78}.  When  $\mathcal F$ contains a bipartite graph, then determining $ex(\nu,\mathcal F)$ is called the {\it degenerate Tur\'an-type extremal graph problem}, and  often only bounds on $ex(\nu,\mathcal F)$ are known in this case.

Replacing ``graph" by ``hypergraph" in the definitions  above,   we obtain the corresponding ones for hypergraphs.

Let $C_n$ denote the cycle of length $n$, $n \ge 3$. If $n$ is even, we refer to $C_n$ as an {\it even} cycle.  Note that  $C_{2k}$ is bipartite,  and many of the applications mentioned later in this survey will be related to forbidden even cycles.

We will also use the following standard notation for the comparison of functions.  Let $f$ and $g$ be two real positive functions defined on positive integers.  We write
\medskip

$f = o(g)$  if $f(n)/g(n) \to 0$ as $n\to \infty$;
\medskip

$f= O(g)$ if there exists a positive constant $c$ such that $f(n) \le c g(n)$ for all sufficiently large $n$;
\medskip

$f= \Theta(g)$ if $f= O(g)$ and $g= O(f)$;
\medskip

$f=\Omega (g)$ if $g= O(f)$;
\medskip

$f \sim g $  if $f(n)/g(n) \to 1$ as $n\to \infty$.

\medskip

\subsection{Graphs without cycles of certain length and with many edges}\label{SS:largeg}
\bigskip

For more on this  subject,  see \cite{Bol78,FS13}.
Our goal here is to mention some results,  not mentioned in \cite{FS13},  and related constructions obtained by the algebraically defined graphs.

The best bounds on
$ex(\nu, \{ C_3,C_4, \cdots , C_{2k}\} )$ for
fixed $k$, $2\le k\not=5$, are presented as follows.
Let $\epsilon = 0$ if $k$ is odd, and
$\epsilon = 1$ if $k$ is even.   Then
\begin{equation}\label{UU:upperb1}
 \frac{1}{2^{1+1/k}} \nu^{1+{\frac{2}{3k - 3+\epsilon}}}\le
ex(\nu, \{ C_3,C_4, \cdots , C_{2k}\} )
\le \frac{1}{2}\,\nu^{1+{\frac{1}{k}}} + \frac{1}{2}\,\nu,
\end{equation}
and
\begin{equation}\label{UU:upperb2}
 \frac{1}{2^{1+1/k}} \nu^{1+{\frac{2}{3k - 3+\epsilon}}}\le
ex(\nu, \{ C_3,C_4, \cdots , C_{2k}, C_{2k+1}\} )
\le \frac{1}{2^{1+1/k}}\,\nu^{1+{\frac{1}{k}}}  + \frac{1}{2}\,\nu .
\end{equation}
The upper bounds in both (\ref{UU:upperb1}) and (\ref{UU:upperb2}) are immediate  corollaries  of the result by Alon, Hoory and Linial \cite{AHL02}.
The lower bound holds for
an infinite sequence of values of $\nu$.  It was established by  Lazebnik, Ustimenko and Woldar in~\cite{LUW95} using some  graphs $B\Gamma_n$,  and those will be discussed in detail in Section \ref{SS:cdkq}.

For $k=2,3,5$,   there exist more precise  results by  Neuwirth \cite{Neu01},  Hoory \cite{Hoo02} and Abajo and Di\'anez \cite{AD12}.
\begin{theorem}  For $k=2,3,5$ and
$\nu= 2(q^{k} + q^{k-1} + \cdots + q + 1)$, $q$ is a prime power,
$$ex(\nu, \{ C_3,C_4, \cdots , C_{2k}, C_{2k+1}\} )
= (q+1)(q^{k} + q^{k-1} + \cdots + q + 1),$$
and every extremal graph is a generalized $(k+1)$-gon $\Pi_q^{k+1}$.
\end{theorem}

Suppose $\mathcal F = \{C_{2k}\}$.  Erd\H{o}s Even Circuit Theorem (see Erd\H{o}s \cite{Erd65}) asserts that  $$ex(\nu, \{ C_{2k}\} ) = O(\nu^{1+1/k}),$$
and the upper bound is probably sharp, but, as far as we know,   Erd\H{o}s never published a proof of it. The first proof followed from a stronger  result by Bondy and Simonovits \cite{BS74}, which implied that $$ex(\nu, \{ C_{2k}\} ) \le 100kv^{1+1/k}.$$  The upper bound was improved by
Verstra\"ete \cite{Ver00} to $8(k-1)\nu^{1+1/k},$ by Pikhurko \cite{Pik12} to $(k-1)\nu^{1+1/k} + O(\nu)$ and by Bukh and Jiang \cite{BJ17} to $80\sqrt{k\log k}\,\nu^{1+1/k} +O(\nu).$

The only values of $k$ for which
$ex(\nu, \{ C_{2k}\} ) = \Theta(v^{1+1/k})$  are $k=2, 3$, and $5$,   with the strongest results  appearing in
\cite{Fur91,Fur96} by F\"uredi (for $k=2$), in \cite{FNV06} by F\"uredi, Naor and Verstra\"ete (for $k=3$), and in \cite{LUW99} by Lazebnik, Ustimenko and Woldar (for $k=5$).

It is a long standing question to determine the magnitude of $ex(\nu, \{ C_{8}\} )$.  The best lower bound is $\Omega (\nu^{6/5})$ and it comes from the generalized hexagon, which has girth 12.  The best upper bound is $O(\nu^{5/4})$  and it comes  from the general bound $O(\nu^{1+1/k})$ on $2k$-cycle-free graphs.

\begin{problem} Is there a graph $B\Gamma_3(\F_q; f_2,f_3,f_4)$ that  contains no 8-cycles for infinitely many $q$?
\end{problem}
A positive answer to this question would imply that $ex(\nu, \{ C_{8}\} )= \Theta(\nu^{5/4})$.

\medskip

\subsection{Wenger graphs}\label{SS:wenger}
\bigskip

A large part of this subsection is based on Cioab\u{a}, Lazebnik  and Li  \cite{CLL14}.

\subsubsection{Defining equations for Wenger graphs}
\bigskip

Let $q=p^e$, where $p$ is a prime and $e\geq 1$ is an integer.
For $m\geq 1$, $1\le k\le m$, let $f_{k+1} = l_kp_1$.  Consider the graph $W_m(q) = B\Gamma_{m+1}(\F_q; f_2,f_3,\ldots, f_{m+1})$:

\begin{align*}
l_2+p_2&=l_1p_1,\\
l_3+p_3&=l_2p_1,\\
&\vdots\\
l_{m+1}+p_{m+1}&=l_m p_1.
\end{align*}
The graph $W_m(q)$ has $2q^{m+1}$ vertices, is $q$-regular and has $q^{m+2}$ edges.

In \cite{Wen91}, Wenger  introduced a family of $p$-regular bipartite graphs $H_k(p)$ as follows. For every  $k\ge 2$, and every prime $p$, the partite sets of $H_k(p)$ are two copies of integer sequences $\{0,1,\ldots, p-1\}^k$,  with vertices  $a=(a_0,a_1,\ldots,a_{k-1})$ and $b=(b_0,b_1,\ldots,b_{k-1})$  forming an edge if
$$b_j \equiv a_j + a_{j+1} b_{k-1} \pmod{p}\;\; \text{for all}\;\; j=0, \ldots, k-2.$$
The introduction and study of these graphs were motivated by the degenerate Tur\'an-type  extremal graph theory problem of determining $\ex(\nu,\{C_{2k}\})$. It is shown in \cite{BS74} that $\ex(\nu,\{C_{2k}\}) = O(\nu^{1 + 1/k})$, $\nu \to \infty$. Lower bounds of magnitude $\nu^{1 + 1/k}$ were known (and still are) for  $k=2,3,5$ only, and the graphs $H_k(p)$, $k=2,3,5$, provided new and simpler examples of such magnitude extremal graphs.

In \cite{LU93}, using a construction based on a certain Lie algebra, the authors arrived at a family of bipartite graphs $H'_n(q)$,  $n\ge 3$, $q$ a prime power, whose partite sets were two copies of $\F_q^{n-1}$, with vertices $(p)=(p_2, p_3, \ldots, p_{n})$ and $[l]=[l_1,l_3,\ldots,l_{n}]$  forming an edge if
 $$l_k - p_k = l_1p_{k-1}\;\; \text{for all}\;\; k=3, \ldots, n.$$
 It is easy to see that for all $k\ge 2$ and prime $p$, the graphs $H_k(p)$ and $H'_{k+1}(p)$ are isomorphic, and the map
  \begin{align*}
\phi:  (a_0,a_1,\ldots,a_{k-1}) &\mapsto (a_{k-1}, a_{k-2}, \ldots, a_0),\\
       (b_0,b_1,\ldots,b_{k-1}) &\mapsto [b_{k-1}, b_{k-2}, \ldots, b_0],
\end{align*}
provides an isomorphism from $H_k(p)$ to $H'_{k+1}(p)$.
Hence, $H'_n(q)$ can be viewed as generalizations of $H_k(p)$.  It is also easy to show that the graphs $H'_{m+2}(q)$ and $W_m(q)$ are isomorphic: the function
      \begin{align*}
\psi:  (p_2, p_3, \ldots, p_{m+2}) &\mapsto [p_2, p_3, \ldots, p_{m+2}],\\
       [l_1, l_3, \ldots, l_{m+2}] &\mapsto (-l_1, -l_3, \ldots, -l_{m+1}),
\end{align*}
mapping points to lines and lines to points, is an isomorphism of $H'_{m+2}(q)$ to $W_m(q)$.

We call the graphs $W_m(q)$  {\it Wenger graphs}.

Another useful  presentation of Wenger graphs  appeared in Lazebnik and Viglione~\cite{LV02}. In this presentation the functions in the  right-hand sides of the
equations are represented as monomials of $p_1$ and $l_1$ only, see Viglione~\cite{Vig02}. For this, define a bipartite graph $W'_m(q)$ with the same partite sets as $W_m(q)$,  where $(p)=(p_1, p_2, \ldots, p_{m+1})$ and $[l]=[l_1,l_2,\ldots,l_{m+1}]$ are adjacent if
\begin{equation} \label{wenger11} l_k + p_k = l_{1}p_1^{k-1}\;\; \text{for all}\;\; k=2, \ldots, m+1.
\end{equation}
The map
  \begin{align*}
\omega:  (p) &\mapsto (p_1, p_2, p_3', \ldots, p_{m+1}'),\;\text{where}\; p_k'=  p_k+ \sum_{i=2}^{k-1} p_{i}p_1^{k-i},\;k=3,\ldots, m+1,\\
       [l] &\mapsto [l_1, l_2, \ldots, l_{m+1}],
\end{align*}
defines an isomorphism from $W_m(q)$  and $W'_m(q)$.
\medskip

\subsubsection{Automorphisms of Wenger graphs}
\bigskip

It was shown in \cite{LU93} that the automorphism group of $W_m(q)$ acts transitively on each of the partitions,  and on the set of edges of $W_m(q)$.   In other words,  the graphs $W_m(q)$ are point-, line-, and edge-transitive.  A more detailed study,  see Lazebnik and Viglione \cite{LV04},  also showed that $W_1(q)$ is vertex-transitive for all $q$, and that $W_2(q)$ is vertex-transitive for even $q$. For all $m\ge 3$ and $q\ge 3$,  and for $m=2$ and all odd $q$,  the graphs $W_m(q)$ are not vertex-transitive.

The full automorphism group of $W_1(q)$ and $W_2(q)$ were completely described in \cite{Vig02}. Some sets of automorphisms of $W_m(q)$ can be found in \cite{LU93}. A more general result, for all graphs $W_m(q)$, is contained in a paper by  Cara, Rottey and Van de Voorde in \cite{CRV14}.  They showed that the full automorphism group of $W_m(q)$ is isomorphic to a subgroup of the group $P\Gamma L(m+2,q)$ that stabilizes a $q$-arc contained  in a normal rational curve of $PG(m+1,q)$, provided  $q \ge  m + 3$ or
$q = p = m + 2$.
\medskip

\subsubsection{Connectivity of  Wenger graphs}
\bigskip

Another result of \cite{LV04} is that $W_m(q)$ is connected when $1\le m\le q-1$,  and disconnected when $m\ge q$, in which case it has $q^{m-q+1}$ components, each isomorphic to $W_{q-1}(q)$. The statement about the number of components of $W_{m}(q)$ becomes apparent from the representation (\ref{wenger11}).   Indeed, as $l_1p_1^i = l_1p_1^{i+q-1}$,  all points and lines in a component have the property that
 their coordinates $i$ and $j$, where $i \equiv j \mod(q-1)$, are equal.  Hence, points $(p)$,  having $p_1=\ldots = p_{q} = 0$,  and at least one distinct coordinate $p_i$, $q+1\le i\le m+1$, belong to different components. This shows that the number of components is at least $q^{m-q+1}$. As $W_{q-1}(q)$ is connected and $W_m(q)$ is edge-transitive, all components are isomorphic to $W_{q-1}(q)$. Hence, there are exactly $q^{m-q+1}$ of them.
 \medskip

It was pointed out in \cite{CLL14} that a result of Watkins \cite{Wat70}, and the edge-transitivity of $W_m(q)$ imply that the vertex connectivity (and consequently the edge connectivity) of $W_m(q)$ equals the degree of regularity $q$, for any $1\leq m\leq q-1$.
In \cite{Vig08}, Viglione  proved that when $1\leq m\leq q-1$, the diameter of $W_m(q)$ is $2m+2$. It will follow from Theorem \ref{CLL14} (see ahead) that for every fixed $m$ and sufficiently large $q$, $W_m(q)$ are expanders, which are defined below.

Let $G=(V,E)$ be a finite graph.
For a subset of vertices $S\subseteq V$,
let $\partial S$ denote the set of edges with one endpoint in $S$ and one endpoint in $V\setminus S$, i.e.
\[
\partial S:=\{xy\in E:x\in S, y\in V\setminus S\}.
\]
The Cheeger constant (also known as isoperimetric number or expansion ratio) of $G$ is defined by
\[
\min\Big\{\frac{|\partial S|}{|S|}:S\subseteq V,0<|S|\le \frac{1}{2}|V|\Big\}.
\]
An infinite family of {\it expanders} is a family of regular graphs whose Cheeger constants are uniformly bounded away from 0.

Let $G$ be a connected $d$-regular graph with $n$ vertices, and let
$\lambda_1\ge \lambda_2\ge\ldots\ge \lambda_n$ be the eigenvalues of the adjacency matrix of $G$.
Suppose $\lambda_2$ is the second largest eigenvalue in absolute value.
Then we call $G$ a {\it Ramanujan} graph if $\lambda_2\le 2\sqrt{d-1}$.
\medskip

 \subsubsection{Cycles in  Wenger graphs}
\bigskip

It is easy to check that all graphs $W_m(q)$ contain an 8-cycle.
Shao, He and Shan \cite{SHS08} proved that in  $W_m(q)$, $q=p^e$, $p$ prime,  for  $m\geq 2$, for any integer $l\neq 5, 4\leq l\leq 2p$ and any vertex $v$, there is a cycle of length $2l$ passing through the vertex $v$. We wish to remark that the edge-transitivity of $W_m(q)$ implies the existence of a $2l$-cycle through any edge, a stronger statement.  The results of \cite{SHS08} concerning cycle lengths in $W_m(q)$ were extended by Wang, Lazebnik and  Thomason  in \cite{WLT20} as follows.

\begin{enumerate}
\item

For $m\ge 2$ and $p \ge 3$, $W_m(q)$ contains cycles of length $2l$, where $4 \le l \le 4p +1$ and $l \neq 5$.

\item
For $q\ge 5$, $0<c<1$, and every integer $l$ with $3 \le  l \le q^c$, if $1 \le m \le (1 - c -
\frac{7}{3} \log_q 2)l - 1$,  $W_m(q)$ contains a $2l$-cycle.  In particular, $W_m(q)$ contains cycles of length $2l$, where $m + 2 \le l \le q^c$, provided $q$ is sufficiently large.
\end{enumerate}

Alexander, Lazebnik and Thomason,  see \cite{Ale16},  showed that for fixed $m$ and large $q$, Wenger graphs are hamiltonian.

\begin{conjecture} {\rm (\cite{WLT20})} For every $m\ge 2$,  and every prime power $q$, $q\ge 3$, $W_m(q)$ contains cycles of length $2k$, where $4\le k\le q^{m+1}$ and $k\neq 5$.
\end{conjecture}

 Representation (\ref{wenger11}) points to a relation of Wenger graphs with the moment curve $t\mapsto (1,t,t^2,t^3,..., t^m)$, and, hence,  with the Vandermonde's determinant, which was explicitly used in \cite{Wen91}. This is also in the background of some geometric constructions by Mellinger and Mubayi \cite{MM05} of magnitude extremal graphs without short even cycles, and in the previously mentioned article \cite{CRV14}.

\medskip

\subsubsection{Spectrum of Wenger graphs}
\bigskip

Futorny and Ustimenko \cite{FU07} considered applications of Wenger graphs in cryptography and coding theory,  as well as some generalizations.  They also conjectured that the second largest eigenvalue $\lambda _2$ of the adjacency matrix of Wenger graphs $W_m(q)$ is bounded from above by $2\sqrt{q}$.
The results of this paper confirm the conjecture for $m=1$ and $2$, or $m=3$ and $q\ge 4$,  and refute it in  other cases.  We wish to point out that for $m=1$ and $2$, or $m=3$ and $q\ge 4$,  the upper bound $2\sqrt{q}$ also follows from the known values of $\lambda _2$ for the point-line $(q+1)$-regular incidence graphs of the generalized polygons $PG(2,q)$, $Q(4,q)$ and $H(q)$ and eigenvalue interlacing, see \cite{BCN89}. In \cite{LLW09},  Li, Lu and Wang showed that the graphs $W_m(q)$,  $m=1,2$, are Ramanujan, by computing the eigenvalues of another family of graphs described by systems of equations in \cite{LU95}, namely  $D(k,q)$, for $k=2,3$.  Their result follows from the fact that  $W_1(q)\simeq D(2,q)$,  and $W_2(q)\simeq D(3,q)$.

 Extending the cases of  $m=2,3$ from \cite{LLW09},   the spectra of Wenger graphs were completely determined in \cite{CLL14}.
\begin{theorem}\label{CLL14}{\rm (\cite{CLL14})}
For all prime power $q$ and $1\leq m \leq q-1$, the distinct eigenvalues of $W_m(q)$ are
\begin{equation*}
\pm q, \; \pm\sqrt{mq}, \;\pm\sqrt{(m-1)q}, \;\cdots ,  \;\pm\sqrt{2q}, \;\pm\sqrt{q}, \; 0 .
\end{equation*}
The multiplicity of the eigenvalue $\pm\sqrt{iq}$ of $W_m(q)$,  $0\leq i \leq m$,  is
\begin{equation*}
(q-1){q \choose i}\sum_{d=i}^m\sum_{k=0}^{d-i} (-1)^k{q-i \choose k}q^{d-i-k}.
\end{equation*}
\end{theorem}

The idea of the proof of this theorem in \cite{CLL14} is the following.  Let $L$ and $P$ denote the set of lines and points of the bipartite graph $W_m(q)$, respectively. Let  $H$ denote the distance-two graph of $W_m(q)$ on $L$. This means that the vertex set of $H$ is $L$, and two distinct lines $[l]$ and $[l']$ of $W_m(q)$ are adjacent in $H$ if there exists a point $(p)\in  P$, such that $[l] \sim  (p)\sim  [l']$ in $W_m(q)$.  It is easy to see that the eigenvalues of $W_m(q)$ can be expressed through those  of $H$ as $\pm \sqrt{\lambda + q}$,  $\lambda \in spec(H)$.     It turns out that $H$  is actually the Cayley graph of the additive group of the vector space $\F^{m+1}$ with a generating set $S =\{(t,tu,...,tu^m) | t \in \F_q^* ,u \in \F_q \}$.  It  allowed the computation of the eigenvalues of $H$ using the techniques in Lov\'{a}sz  \cite{Lov75} and Babai \cite{Bab79},  and, hence, of $W_m(q)$.

As we already mentioned, this theorem implies that the graphs $W_m(q)$ are expanders for every fixed $m$ and large $q$.
\medskip

\subsection{Some Wenger-like graphs}
\bigskip

Extended from Wenger graphs, there are some Wenger-like graphs. In \cite{Por18}, Porter introduced some Wenger-like graphs and studied their properties.

\subsubsection{Generalized Wenger graphs}
\bigskip

Cao, Lu, Wan,  L.-P. Wang and Q. Wang \cite{CLWWW15} considered the {\it generalized Wenger graphs}  $G_m(q)=B\Gamma_m(\F_q; f_2,\ldots, f_{m+1})$, with  $f_k= g_k(p_1) l_1$, $2\le k\le m+1$,  where $g_k\in \F_q[X]$ and the mapping $\F_q\to \F_q^{m+1}$, $u\mapsto (1,g_2(u), \ldots, g_{m+1}(u))$ is injective.
A more general result in \cite{LW01} implies that the generalized Wenger graph $G_m(q)$ is $q$-regular. The authors in \cite{CLWWW15} determined the spectrum of the generalized Wenger graphs.
\begin{theorem}{\rm (\cite{CLWWW15})}\label{spectrumgeneral}
For all prime power $q$ and positive integer $m$, the eigenvalues of the generalized Wenger graph $G_m(q)$ counted with multiplicities, are
\begin{equation*}
\pm\sqrt{qN_{F_\omega}}, \;\omega=(\omega_1,\omega_2,\ldots,\omega_{m+1})\in \F_q^{m+1},
\end{equation*}
where $F_\omega(u)=\omega_1+\omega_2g_2(u)+\ldots+\omega_{m+1}g_{m+1}(u)$ and $N_{F_\omega}=|\{u\in \F_q:F_{\omega}(u)=0\}|$.
For $0\le i\le q$, the multiplicity of $\pm\sqrt{qi}$ is
\begin{equation*}
n_i=|\{\omega\in \F_q^{m+1}:N_{F_\omega}=i\}|.
\end{equation*}
Moreover, the number of components of $G_m(q)$ is
\begin{equation*}
q^{m+1-rank_{\F_q}(1,g_2,\ldots,g_{m+1})}.
\end{equation*}
Therefore $G_m(q)$ is connected if and only if $1,g_2,\ldots,g_{m+1}$ are $\F_q$-linearly independent.
\end{theorem}
The idea of the  proof is similar to the one in \cite{CLL14} with the distance-two graph $H$ of $G_m(q)$ on $L$ being the Cayley graph of the additive group of the vector space $\F^{m+1}$ with a generating set $S =\{(t,tg_2(u),...,tg_{m+1}(u)) | t \in \F_q^* ,u \in \F_q \}$.
\medskip

\subsubsection{Linearized  Wenger graphs}
\bigskip

An important particular case of the generalized Wenger graphs is obtained when $g_k(X)= X^{p^{k-2}}$,  $2\le k\le m+1$.  The authors of \cite{CLWWW15} called these graphs the {\it linearized Wenger graphs $L_m(q)$},  and they determined their girth, diameter and the spectrum.
If $m>e$, the linearized Wenger graph is not connected.
It has $q^{m-e}$ connected components, each of which is isomorphic to the graph with $m=e$.

\begin{theorem} {\rm (\cite{CLWWW15})}\label{diameterlinear}
If $m\le e$, the diameter of the linearized Wenger graph $L_m(q)$ is $2(m+1)$.
\end{theorem}

The outline of the proof is as follows.
Consider the distance between any two lines $L$ and $L'$ in $L_m(q)$.
If there is a path of length $2(m+1)$ with endpoints $L$ and $L'$,
then the coordinates of vertices on the path must satisfy a system of linear equations.
By the fact that if the coefficient matrix of the linear equations is nonsingular,
then the system of linear equations has a unique solution,
there are solutions for the coordinates of vertices on the path.
Thus the diameter of any two vertices in the same partite set is at most $2(m+1)$.
Modifying the construction so that the path goes through the point $P$, similarly, the distance between $P$ and $L$ is no more than $2(m+1)$, so that the diameter of $L_m(q)$ is at most $2(m+1)$.
On the other hand, the distance $2(m+1)$ can be attained. Choose two lines $L$ and $L'$ such that $L'-L=[0,\ldots,0,1]$, if the distance between them is $2s$ with $s\le m$.
Then the system of linear equations has no solution. So the distance between them is at least $2(m+1)$.

\begin{theorem}{\rm (\cite{CLWWW15})}\label{girthlinear}
Let $q=p^e$. If $m\ge 1,e\ge 1$ and $p$ is an odd prime, or $m=1,e\ge 2$ and $p=2$, then the girth of the linearized Wenger graph $L_m(q)$ is 6;
if $p=2$ and either $e=m=1$ or $e\ge 1, m\ge 2$, then the girth of the linearized Wenger graph $L_m(q)$ is 8.
\end{theorem}

For all $q\ge 2$, $L_1(q)$ is isomorphic to $W_1(q)$, which is 4-cycle-free.
Hence, $L_m(q)$ has girth at least 6 for $m\ge 1$.
The authors also determined the coordinates of points and lines of a shortest cycle in $L_m(q)$ in the proof of Theorem \ref{girthlinear}.

The spectrum of the linearized Wenger graphs is computed in \cite{CLWWW15} for the case $m\ge e$, and by Yan and Liu \cite{YL17} for the case $m<e$. The proofs of the following two theorems rely on the spectrum of the generalized Wenger graphs as shown in Theorem \ref{spectrumgeneral}.
\begin{theorem}{\rm (\cite{CLWWW15})}
Let $m\ge e$. The linearized Wenger graph $L_m(q)$ has $q^{m-e}$ components. The distinct eigenvalues are
\begin{equation*}
0, \;\pm\sqrt{qp^i},\;0\le i\le e.
\end{equation*}
The multiplicities of the eigenvalues $\pm\sqrt{qp^i}$ are $q^{m-e}p^{e-i}\frac{\prod_{j=0}^{e-i-1}(p^e-p^j)^2}{\prod_{j=0}^{e-i-1}(p^{e-i}-p^j)}$.
The multiplicity of the eigenvalue $0$ is $q^{m-e}\sum_{i=1}^e (p^e-p^{e-i})\frac{\prod_{j=0}^{e-i-1}(p^e-p^j)^2}{\prod_{j=0}^{e-i-1}(p^{e-i}-p^j)}$.
\end{theorem}
For the following theorem, the Gaussian binomial coefficients are required:
\begin{align*}
\begin{split}
\binom{n}{k}_p= \left \{
\begin{array}{ll}
    \prod_{t=0}^{k-1}\frac{p^n-p^t}{p^k-p^t},                    &\text{if } 1\le k\le n,\\
    1,     &\text{if } k=0,\\
    0,                                 &\text{if } k>n,
\end{array}
\right.
\end{split}
\end{align*}
where $n$ and $k$ are nonnegative integers.

\begin{theorem}{\rm (\cite{YL17})}
Let $m< e$. The eigenvalues of the linearized Wenger graph $L_m(q)$ are
\begin{equation*}
0, \;\pm q,\;\pm\sqrt{qp^i},\;0\le i\le m-1.
\end{equation*}
The multiplicities of the eigenvalues $\pm q$ are 1,
the multiplicities of the eigenvalues $\pm \sqrt{qp^i}$ are $p^{e-i}n_i$,
and the multiplicity of the eigenvalue $0$ is $2(q^{m+1}-1-\sum_{i=0}^{m-1}p^{e-i}n_i)$,
where
\begin{equation*}
n_i=\binom{e}{i}_p\sum_{j=0}^{m-i-1}(-1)^jp^{\frac{j(j-1)}{2}}\binom{e-i}{j}_p(q^{m-i-j}-1),
\end{equation*}
where $\binom{e}{i}_p$ and $\binom{e-i}{j}_p$ are the Gaussian binomial coefficients.
\end{theorem}
For $q=p^e$, the results imply  that the graphs $L_e(q)$ are expanders.  It follows from \cite{Ale16} that for a fixed $e$ and large $p$, $L_e(p^e)$  are hamiltonian.
The lengths of some cycles in $L_m(q)$ were found by Wang \cite{Wan17} using a similar method as in \cite{WLT20} and the adjacency of some points and lines in cycles.

\begin{theorem}{\rm (\cite{Wan17})}
Let $q$ be a power of the prime $p$ with $p\ge 3$. For any integer $k$ with $3\le k\le p^2$, the linearized Wenger graph $L_m(q)$ contains cycles of length $2k$.
\end{theorem}

\begin{problem} Determine the lengths of all cycles in the linearized Wenger graph $L_m(q)$.
\end{problem}
\medskip

\subsubsection{Jumped Wenger graphs}
\bigskip

Another particular case of the generalized Wenger graphs is obtained when $(g_2(X),\ldots,g_{m+1}(X))=(X,X^2,\ldots,X^{i-1},X^{i+1},\ldots,X^{j-1},X^{j+1},\ldots,X^{m+2})$.
L.-P. Wang, Wan, W. Wang and Zhou of \cite{WWWZ22} called these graphs the {\it jumped Wenger graphs $J_m(q,i,j)$}, where $1\le i<j\le m+1$.
It is easy to obtain that $J_m(q,i,j)$ is $q$-regular, and if $m+2<q$, $J_m(q,i,j)$ is connected. The authors also determined their girth and diameter.
\begin{theorem} {\rm(\cite{WWWZ22})}
If $1\le m<q-2$, the diameter of the jumped Wenger graph $J_m(q,i,j)$ is at most $2(m+1)$.
In particular, the diameters of $J_m(q,m,m+2)$, $J_m(q,m+1,m+2)$ and $J_m(q,m,m+1)$ are all $2(m+1)$.
\end{theorem}
The proof of this theorem used an idea similar to that in the proof of Theorem \ref{diameterlinear}.

\begin{theorem}{\rm (\cite{WWWZ22})}
The girth of the jumped Wenger graph $J_m(q,i,j)$ is at most $8$.
\end{theorem}
In the proof, the coordinates of each point and line in the cycle of length 8 are shown.
The exact girth of $J_m(q,i,j)$ is determined via case by case analysis, see \cite{WWWZ22}.
\medskip

\subsection{Sun graphs}
\bigskip

Here we discuss a large family of regular Cayley graphs which, in general,  are not bipartite and are defined by a system of equations described in Section~\ref{SS:ord}.  In particular, the generalized distance-two graphs of Wenger graphs and of the linearized Wenger graphs belong to this family.

Let $k\ge 3$ be an integer,  and $\F_q$ denote the field of $q$ elements. Let $f_i,g_i \in \F_q[X]$ with $g_i(-X) = - g_i(X)$, $3\le i \le k$. Consider  a graph $S(k,q) = S(k,q;f_3,g_3,\cdots ,f_k,g_k)$ with the vertex set $\F_q^k$  and edges defined as follows: vertices $a = (a_1,a_2,...,a_k)$ and $b = (b_1,b_2,...,b_k)$ are adjacent if $a_1= b_1$ and the following $k -2$ relations on their components hold: $$b_i -a_i = g_i(b_1 -a_1)f_i \Big(\frac{b_2 -a_2}{b_1 -a_1}\Big), \;\;\;
3 \le i \le k. $$

We call $S(k,q)$ {\it Sun graphs}. The requirement $g_i(-X)= -g_i(X)$  ensures that adjacency in $S(k,q)$  is  symmetric. It was introduced and studied by Sun \cite{Sun17} and Cioab\u{a}, Lazebnik and Sun \cite{CLS18}.
\medskip

Note that $S(k,q)$ are not bipartite.  It is easy to see that they are Cayley graphs with underlying group being the additive group of the vector space $\F_q^k$ with generating set
  $$ \Big\{\Big(a,au,g_3(a)f_3(u),\cdots, g_k(a)f_k(u)\Big)\,|\, a\in \F_q^*, u\in \F_q \Big\}. $$
   This implies that $S(k,q)$ is a vertex-transitive $q(q-1)$-regular graph,  and that their spectra can be studied using \cite{Bab79,Lov75}.   Note that for $f_i = X^{i-1}$ and $g_i = X$, $3\le  i\le  k +1$, $S(k +1,q)$ coincides with the distance-two graph of the Wenger graphs $W_k(q)$ on lines,  and for $f_i = X^{p^{i-2}}$ and $g_i = X$, $3 \le i\le  k+1$, $S(k+1,q)$ coincides with the distance-two graph of the linearized Wenger graphs $L_k(q)$ on lines.

The connectivity and expansion properties of the graphs $S(k,q)$ were studied in \cite{Sun17,CLS18}.  The spectral properties of $S(3; q; x^2; x^3)$ for prime
$q$  between 5 and 19, and of $S(4; q; x^2; x^3; x^3; x^3)$ for prime $q$ between 5 and
13, show that these graphs are {\it not} distance-two graphs of any $q$-regular bipartite
graphs.
\medskip

\subsection{Terlep--Williford graphs}
\bigskip

In \cite{TW12}, Terlep and Williford considered the graphs $TW(q)= B\Gamma_8(\F_q;f_2,\ldots ,f_8)$, where $$f_2 = p_1l_1, \;\; f_3=p_1l_2, \;\; f_4=p_1l_3, \;\; f_5 = p_1l_4, \;\; f_6 = p_2l_3 - 2p_3l_2 + p_4l_1, $$ $$f_7 = p_1l_6 + p_2l_4 - 3p_4l_2 +2p_5l_1, \;\; \text{and} \;\;  f_8 =2p_2l_6 - 3p_6l_2 + p_7l_1.$$
These graphs provide the best asymptotic lower bound on the Tur\'an number of the 14-cycle $ex(\nu,\{C_{14}\})$, i.e.  the greatest number of edges in a graph of order $v$ containing no 14-cycle. The approach to their construction is similar to the one in \cite{LU95}, and it is obtained from a Lie algebra related to a generalized Kac--Moody algebra of rank 2.
\begin{theorem}  {\rm(\cite{TW12})}
For infinitely many values of $\nu$, $$ex(\nu,\{C_{14}\}) \ge (1/2^{9/8}) \nu^{9/8}. $$
\end{theorem}
The best known upper bound is still of magnitude $\nu^{8/7}$, as follows from the Erd\H os Even Circuit Theorem and \cite{BS74}.  We wish to note that $TW(q)$ also have no cycles of length less than 12. For $q =5,7$, they do contain 12-cycles, and likely have girth 12 in general. The proof in \cite{TW12} performs a Gr\"{o}bner basis computation using the computer algebra system Magma, which established the absence of 14-cycles over the field of algebraic numbers. The transition to finite fields was made using the Lefschetz principle, see, e.g. Marker \cite{Mar02}. We wish to end this section with a problem.

\begin{problem}
Provide a computer-free proof of the fact that the graphs $TW(q)$ contain no 14-cycles for infinitely many $q$.
\end{problem}
\medskip

\subsection{Verstra\"ete--Williford graphs}
\bigskip

Let $q$ be an odd prime power.  In \cite{VW19},  Verstra\"ete  and Williford  consider the graphs $G_q$ with vertex set $V= \F_q^4$, where distinct vertices $a=\langle a_1,a_2,a_3,a_4\rangle$ and $b=\langle b_1,b_2,b_3,b_4\rangle$ are adjacent if $$ a_2+b_2= a_1b_1, \;\; a_3+b_4=a_1^2b_1, \;\; a_4+b_3=a_1b_1^2.$$
One can easily check that every $G_q$ has $q^2$ vertices of degree $q-1$ and all other vertices are of degree $q$.  Hence,  $G_q$ has $\frac{1}{2}\,(q^5-q^2)$ edges. It is also easy to check that $G_q$ contains neither  4-cycles nor 6-cycles.   The last statement also follows from the fact that $G_q$ is a polarity graph of $D(4,q)$  and Theorem 1 in \cite{LUW99}. For  details,  see Section \ref{polarauto}.

A theta-graph, denoted $\theta_{k,l}$,   consists of $k\ge 2$ internally disjoint paths of length $l$ with the same endpoints.  It is demonstrated in \cite{VW19} that for odd prime power $q$,  $G_q$ contains no subgraph isomorphic to $\theta_{k,l}$.  Together with the upper bound of Faudree and Simonovits from \cite{FS83}, it implies the asymptotic order of magnitude of the Tur\'an number for $\theta_{3,4}$,  namely that   $$ex(\nu, \theta_{3,4}) = \Theta(\nu^{5/4}).$$  More on what is known about $\theta_{k,l}$ can be found in \cite{VW19}.
\medskip

\subsection{Monomial graphs and generalized quadrangles}\label{monographs}
\bigskip

When all functions $f_i$ are monomials of two variables,  we call the graphs $B\Gamma_n(\F_q; f_2, \ldots, f_n)$ {\it monomial graphs}. These graphs were first studied in \cite{Vig02}  and in Dmytrenko \cite{Dmy04}.  Let $B(q;m,n) = B\Gamma_2(\F_q; p_1^m l_1^n)$.  For fixed $m,n$ and sufficiently large $q$, the isomorphism problem for $B(q;m,n)$  was solved in
\cite{Vig02},  and for all $m,n,q$,  in Dmytrenko, Lazebnik and Viglione \cite{DLV05}.

\begin{theorem}\label{isomono}{\rm (\cite{DLV05})}
Let $m,n,m',n'$ be positive integers and let $q, q'$ be prime powers. The graphs $B(q;m,n)$ and
$B(q';m',n')$ are isomorphic if and only if $q=q'$ and $\{\gcd(m,q-1), \gcd(n,q-1)\}= \{\gcd(m',q-1), \gcd(n',q-1)\}$ as multisets.
\end{theorem}

It is easy to argue, see \cite{DLV05},   that every 4-cycle-free graph of the form $B(q;m,n)$,   is isomorphic to $B(q;1,1)$, and so is isomorphic to the biaffine part of the point-line incidence graph of $PG(2,q)$.
This result extends simply to the graphs $B\Gamma_2(\F; p_1^ml_1^n)$, where $\F$ is an algebraically closed field.  On the other hand,  it does not hold when $\F=\R$ -- the field of real numbers. Kronenthal, Miller, Nash, Roeder, Samamah and Wong \cite{KMNRSW24} showed that there exists $f\in \R[p_1,l_1]$ such that $\Gamma(\R;f)$ has girth 6 and is nonisomorphic to $\Gamma(\R;p_1l_1)$ by providing
$f=f(p_1,l_1)=\sum_{i,j\in \N}\alpha_{i,j}p_1^il_1^j\in \Q[p_1,l_1]\setminus \{0\}$,
where both $i$ and $j$ are odd for all nonzero $\alpha_{i,j}$ and either all $\alpha_{i,j}\ge 0$ or all $\alpha_{i,j}\le 0$.

An analogous statement in dimension three is less clear.
For each odd prime power $q$, only two non-isomorphic generalized quadrangles  of order $q$, viewed as finite  geometries, are known.
They are usually denoted by $W(q)$ and $Q(4,q)$, and it is known that one is the dual of the other, see Benson \cite{Ben65}.  This means that viewed as bipartite graphs they are isomorphic. For odd prime power $q$, the graph $B\Gamma_3(\F_q; p_1l_1, p_1l_1^2)$, which has girth 8,  is the biaffine part of $W(q)$.
Just as a 4-cycle-free graph $B\Gamma_2(\F_q; f_2)$ gives rise to a projective plane, a three-dimensional 4- and 6-cycle-free graph $B\Gamma(\F_q; f_2,f_3)$  may give rise to a generalized quadrangle. This suggests to study  the existence of such graphs, and it is reasonable to begin to search for them in the `vicinity' of the graphs $B\Gamma_3(\F_q; p_1l_1, p_1l_1^2)$, by which we mean among monomial graphs.

Another motivation to study monomial graphs in this context is the following.
For  $q=2^e$, contrary to the two-dimensional case, the monomial graphs can lead to a variety of non-isomorphic generalized quadrangles, see Payne \cite{Pay90}, \cite{VMal98}, Glynn \cite{Gly83,Gly89}, Cherowitzo \cite{Che00}. It is conjectured  in \cite{Gly83} that known examples of such quadrangles represent all possible ones. The conjecture was checked by computer for all $e \le 28$ in \cite{Gly89}, and for all $e \le 40$,  by Chandler \cite{Cha05}.

The study of monomial graphs of girth eight for odd $q$ began in \cite{Dmy04}, and continued in
Dmytrenko, Lazebnik and Williford \cite{DLW07}, and in Kronenthal \cite{Kro12}.
All results in these papers suggested that for $q$ odd,  every monomial graph $B\Gamma_3(\F_q; f_2,f_3)$ of girth at least eight is isomorphic to the graph $B\Gamma_3(\F_q; p_1l_1,p_1l_1^2)$, as was conjectured in \cite{DLW07}.

 We wish to note that investigation of cycles in the monomial graphs leads to several interesting questions about bijective functions on $\F_q$, also known as {\it permutation polynomials}  (every function on $\F_q$ can be represented as a polynomial). It was shown in \cite{Dmy04} that if $q$ is odd and the girth of $B\Gamma_3(\F_q; p_1l_1,p_1^m l_1^n)$ is eight,  then  $m=k$ and $n=2k$ for  an integer $k$ satisfying certain conditions. This led to the following conjecture.
\begin{conjecture} {\rm (\cite{DLW07})}\label{DLW} Let $q=p^e$ be an odd prime power.  For an integer  $k$, $1\le k\le q-1$, let
$A_k = X^k[(X+1)^k-X^k]$  and $B_k = [(X+1)^{2k} -1]X^{q-1-k} -2X^{q-1}$ be polynomials in $\F_q[X]$. Then each of them  is a permutation polynomial of $\F_q$  if and only if $k$ is a power of $p$.
\end{conjecture}
It was shown in \cite{Dmy04,DLW07} that the validity of the conjecture  for either $A_k$ or $B_k$,  would imply the following theorem.

 \begin{theorem}\label{HLL17} {\rm(\cite{HLL17})} Let $q$ be an odd prime power.  Then
 every monomial graph $B\Gamma_3(\F_q; f_2,f_3)$ of girth at least eight is isomorphic to
 the graph $B\Gamma_3(\F_q; p_1l_1,p_1l_1^2)$.
 \end{theorem}

Hou, Lappano and Lazebnik in \cite{HLL17} proved Theorem \ref{HLL17} by making sufficient progress on Conjecture \ref{DLW},  though falling short of proving it either for $A_k$ or  for $B_k$. Finally,  Conjecture \ref{DLW} for $A_k$  was confirmed by Hou \cite{Hou18}, and for $B_k$, by  Bartoli and Bonini \cite{BB24}.

Hence,  no new generalized 4-gon can be constructed in this way.  What if not both polynomials $f_2$ and $f_3$ are monomials?  In \cite{KL16},  Kronenthal and Lazebnik showed that over every algebraically closed field $\F$ of characteristic zero,  every  graph $B\Gamma_3(\F; p_1l_1,f_3(p_1,l_1))$  of girth at least eight is isomorphic to the graph $B\Gamma_3(\F; p_1l_1,p_1l_1^2)$.  Their methods imply that the same result holds over infinitely many finite fields. In particular, the following theorem holds.

\begin{theorem}\label{KL16} {\rm(\cite{KL16})} Let $q$ be a power of a prime $p$, $p\ge 5$, and   let $M=M(p)$ be the least common multiple of
integers $2,3,\ldots, p-2$. Suppose $f_3\in \F_q[p_1,l_1]$ has degree at most $p-2$ with respect to each of $p_1$ and $l_1$.  Then over every finite field extension $\F$ of $\F_{q^M}$,
every graph $B\Gamma_3(\F; p_1l_1,f_3(p_1,l_1))$ of girth at least eight is isomorphic to the
graph $B\Gamma_3(\F; p_1l_1,p_1l_1^2)$.
\end{theorem}
Kronenthal, Lazebnik and Williford \cite{KLW19} extended these ``uniqueness" results to the family of graphs $B\Gamma_3(\F; p_1^ml_1^n,f_3(p_1,l_1))$ (with $p_1l_1$ replaced by an arbitrary monomial $p_1^ml_1^n$).
Then Xu, Cheng and Tang \cite{XCT21} extended these results to the family of graphs $\Gamma_3(\F;f_2,f_3)$,  where $f_2=g(p_1)h(l_1)$, which is the product of two univariate polynomials, and $f_3=f_3(p_1,l_1)$.

\begin{problem}\label{polyquad}

(i) Let $q$ be an odd prime power, and let $f_2,f_3\in \F_q[p_1,l_1]$.  Is it true that every graph $B\Gamma_3(\F_q; f_2,f_3)$ with girth at least eight is isomorphic to the graph $B\Gamma_3(\F_q; p_1l_1,p_1l_1^2)$?
\smallskip

\noindent (ii) Let $q$ be an odd prime power, and let $f_2\in \F_q[p_1,l_1]$ and $f_3\in \F_q[p_1,l_1,p_2,l_2]$. Is it true that every graph $B\Gamma_2(\F_q; f_2,f_3)$ with girth at least eight is isomorphic to the graph $B\Gamma_2(\F_q; p_1l_1,p_1l_1^2)$?

\end{problem}
It is clear that a negative  answer to each of the two parts of Problem \ref{polyquad}  may lead to a new generalized quadrangle. It will lead to one, if such a graph exists and it is possible to ``attach" it to a $(q+1)$-regular tree on  $2(q^2+q+1)$ vertices.  Though we still cannot conjecture the uniqueness result for  odd $q$,  we  believe that it holds over algebraically closed fields.

\begin{conjecture} {\rm (\cite{KL16})} Let $\F$ be an algebraically closed field of characteristic zero, and let $f_2,f_3\in \F [p_1,l_1]$.  Then every graph $B\Gamma_3(\F; f_2,f_3)$ with girth at least eight is isomorphic to the graph $B\Gamma_3(\F; p_1l_1,p_1l_1^2)$.
\end{conjecture}

The investigations in \cite{KL16} were continued in  Nassau \cite{Nas20}. Let $\F$  be an algebraically closed field of characteristic zero.
Consider the graph $B\Gamma_3(\F; f_2,f_3)$.  In their  definition (\ref{Bipmain0}), the first defining polynomial functions are  $f_2= f_2(p_1,l_1)$  and $f_3= f(p_1,l_1,p_2,l_2)$.  Note that $f_3$ can always be thought of as a function of three variables rather than four, namely  $f_3=  f_3(p_1,l_1, p_2) = f(p_1,l_1, p_2, f_2(p_1,l_1) - p_2)$ or $f_3=  f_3(p_1,l_1, l_2) = f(p_1,l_1, f_2(p_1,l_1) - l_2, l_2)$.
Therefore, we may assume that $f_3=  f_3(p_1,l_1, p_2)$.  In particular, if $f_3$ is actually a function of two variables, say $l_1$ and $p_2$, we write it as $f_3(l_1,p_2)$.
Among several results in \cite{Nas20},  we mention just a few.
\begin{theorem} \label{NasRes} {\rm(\cite{Nas20})} Let $\F$ be an algebraically closed field,  and $m,n$ be positive integers.
\begin{enumerate}
\item If $B\Gamma_3(\F; p_1l_1,f_3(l_1,p_2))$ has girth at least 8, then its girth is 8 and it is isomorphic to the graph $B\Gamma_3(\F; p_1l_1, p_1l_1^2)$.
\item If $B\Gamma_3(\F; p_1l_1,p_1^m g(l_1,p_2))$, where $g$ is a polynomial over $\F$,  has girth at least 8, then its girth is 8 and it is isomorphic to the graph $B\Gamma_3(\F; p_1l_1, p_1l_1^2)$.
\item If $B\Gamma_3(\F; p_1^m l_1^n, f_3(p_1,l_1, p_2))$  has girth at least 8, then $m=1$ or $m=2$.
\end{enumerate}
\end{theorem}
Another question raised in \cite{Nas20} was the following: Suppose $q$ is an odd prime power. Given $B\Gamma_3 (\F_q; f_2(p_1,l_1), f_3(p_1,l_1,p_2))$,  is it always possible to find a polynomial $h=h(p_1,l_1)$ such that the graphs $B\Gamma _3(\F_q;  f_2 (p_1,l_1), h(p_1,l_1))$ and $B\Gamma_3 (\F_q; f_2(p_1,l_1), f_3(p_1,l_1,p_2))$ are isomorphic?
There are many examples where the answer to this question  is affirmative. Moreover, given $f_3$, the polynomial $h$ can be found in many ways. For example, it is easy to verify that the following graphs are isomorphic:
    $$B\Gamma _3(\F_q;p_1l_1, p_1^2l_1 - p_1p_2) \simeq B\Gamma_3(\F_q;p_1l_1, p_2l_1)\simeq $$
    $$B\Gamma _3(\F_q;p_1l_1, p_1l_1^2)\simeq B\Gamma _3(\F_q;p_1l_1, p_1^2l_1).$$
Does one really need three variables in $f_3$ to create a graph not isomorphic to those with two variables $p_1$ and $l_1$  for infinitely many $q$?
The  answer to this question is positive. The following graph appeared in \cite{Nas20}: $$B\Gamma_3 ({\F_p}; p_1l_1, p_1 1_1 p_2(p_1 + p_2 + p_1p_2)),$$
where $p$ is an odd prime,  and it was conjectured to have the property, i.e. it is not isomorphic to any graph $B\Gamma_3 (\F_p; f_2(p_1,l_1), h(p_1,l_1)) $.   The result was proven for primes  $p \equiv 1 \pmod{3}$  in Lazebnik and Taranchuk \cite{LT22}.  The authors  showed that the automorphism group of the graph $B\Gamma _3({\F_p}; p_1l_1, p_1l_1p_2(p_1 + p_2 + p_1p_2))$ is trivial. At the same time,  the   automorphism group of any  graph $B\Gamma_3 (\F_p; f_2(p_1,l_1), h(p_1,l_1)) $ contains a subgroup isomorphic to the additive group of the field $\F_p$, formed by the automorphisms $\phi_a$ presented in (\ref{autolastcoordp}), (\ref{autolastcoordl}). Hence, for every prime $p \equiv 1 \pmod{3}$, the graphs $B\Gamma _3({\F_p}; p_1\ell_1, p_1\ell_1p_2(p_1 + p_2 + p_1p_2))$ and $B\Gamma_3 (\F_p; f_2(p_1,l_1), h(p_1,l_1)) $
are not isomorphic.

Kodess, Kronenthal and Wong \cite{KKW22} studied $B\Gamma_3(\F;f(p_1)h(l_1),g(p_1)j(l_1))$, and classified all the graphs $B\Gamma_3(\F;f(p_1)h(l_1),g(p_1)h(l_1))$ by girth, where $\F$ is an algebraically closed field.
Ganger, Golden, Kronenthal and Lyons \cite{GGKL19} proved that every graph $B\Gamma_2(\R;f_2(p_1,l_1))$ has girth 4 or 6 and classified infinite families of such graphs by girth, where $\R$ is the field of real numbers.
Then Kodess, Kronenthal, Manzano-Ruiz and Noe \cite{KKMN21} gave a complete classification of monomial graphs $B\Gamma_3(\R;p_1^ml_1^n,p_1^sl_1^t)$.
\medskip

\subsection{$(n^{2/3},n)$-bipartite graphs of girth 8 with many edges}
\label{SS:szek}
\bigskip

  Let $f(n,m)$ denote the greatest number of edges
in a bipartite graph
whose bipartition sets have cardinalities
$n,m$ ($n\ge m$) and whose  girth is at least 8.
In \cite{Erd79}, Erd\H{o}s conjectured that $ f(n,m) = O(n)$
for $m= O(n^{2/3})$. For a motivation of this question, see de Caen and Sz\'ekely~\cite{CS91}.
Using some results from combinatorial number
theory and set systems, this conjecture was refuted in \cite{CS91},
by showing the existence of an infinite family
of $(m,n)$-bipartite graphs
with $m\sim n^{2/3}$,  girth
at least 8, and having
$n^{1+ 1/57 + o(1)}$ edges. As the authors
pointed out,
this disproved Erd\H{o}s' conjecture, but fell well short
of their
upper bound
$O(n^{1+1/9})$.

Using certain induced subgraphs of
algebraically defined graphs,
Lazebnik, Ustimenko and Woldar~\cite{LUW94}
explicitly constructed
an infinite family
of $(n^{2/3},n)$-bipartite graphs
of girth 8 with $n^{1+ 1/15}$ edges.
Here is the construction.

Let $q$ be an odd prime power, and
set
$P= \mathbb{F}_q\times
\mathbb{F}_{q^2}\times \mathbb{F}_q$,
$L=\mathbb{F}_{q^2}\times  \mathbb{F}_{q^2}\times
\mathbb{F}_q$.
We define the bipartite graph $\Gamma(q)$
with bipartition $P\cup L$ in which $(p)$ is adjacent to $[l]$ provided
\begin{align}
  l_2 + p_2 &= p_1l_1, \notag \\
  l_3+ p_3  &= - (p_2\overline{l_1} + \overline{p_2}l_1),\notag
\end{align}
and here $\overline{x}$ denotes the
image of $x$ under the involutory
automorphism of $\mathbb{F}_{q^2}$ with fixed field
$\mathbb{F}_q$.

In the context of the current survey,
$\Gamma(q)$ is closely related
to the induced subgraph
$B\Gamma_3[\mathbb{F}_q,\mathbb{F}_{q^2}]$
of $B\Gamma_3=B\Gamma_3(\mathbb{F}_{q^2};p_1l_1,- (p_2\overline{l_1} + \overline{p_2}l_1))$
(see Section~\ref{SS:indsub}).
Indeed, the only difference is that the third
coordinates of the vertices of
$\Gamma(q)$ are required to come from
$\mathbb{F}_q$.

Assuming now that
$q^{1/3}$ is an integer, we may further
choose $A\subset\mathbb{F}_q$
with $|A|=q^{1/3}$.
Set
$P_A=A\times
\mathbb{F}_{q^2}\times \mathbb{F}_q$,
and denote by $\Gamma'(q)$
the subgraph of
$\Gamma(q)$ induced by the
set $P_A\cup L$.
Then the family
$\{\Gamma'(q)\}$
 gives
the desired
$(n^{2/3},n)$-bipartite graphs
of girth 8 and $n^{1+ 1/15}$ edges,
where
$n = q^2$, see ~\cite{LUW94} for details.
\begin{problem}  Improve the magnitude (exponent of $n$) in either the upper  or the lower bound in the inequality
$$   c_1n^{1+1/15}\le  f(n^{2/3},n) \le c_2n^{1+1/9},$$
where $c_1,c_2$ are positive constants.
\end{problem}
\medskip

\subsection{Digraphs}
\bigskip

Consider a digraph with loops $ D_n=\Gamma_n(\F_q;f_2, \ldots, f_{n})$, defined as in Section \ref{SS:ord} by system (\ref{E:eqn12}), with $f_i$'s  not necessarily symmetric. The study of these digraphs was initiated in Kodess \cite{Kod14}. Some general properties of these digraphs are similar to those of the graphs $\Gamma_n$.  A digraph is called {\it strongly connected} if there exists a directed path between any two of its vertices,  and every digraph is a union of its strongly connected (or just strong)  components.

Suppose  each $f_i$ is a function of only two variables,  and
there is an arc from a vertex $ \langle a_1,\dotso,a_{n}\rangle$ to a vertex
$ \langle b_1,\dotso,b_{n}\rangle$ if
$$
a_i + b_i = f_{i}(a_1,b_1),\;\;\text{for all}\; i, \;2\le i \le n.
$$
The strong  connectivity of these digraphs was studied by Kodess and Lazebnik \cite{KL15}. Utilizing  some ideas from \cite{Vig02}, they obtained   necessary and sufficient conditions for strong connectivity of $D_n$ and completely described its strong components.  The results are expressed in terms of the properties of the span over $\F_p$ of the image of an explicitly constructed  vector function from $\F_q^2$ to $\F_q^{n-1}$, whose definition depends on the functions $f_i$. The details are a bit lengthy, and can be found in  \cite{KL15}.

Finding the diameter of strongly connected  digraphs $D_n$ seems to be a very hard problem,  even for $n=2$. Specializing $f_2$ to a monomial of two variables, i.e.,  $f_2=X^mY^n$,   makes it a bit easier,  though exact results are still
 hard to obtain.  In \cite{KLSS16},  Kodess, Lazebnik, Smith and Sporre studied the diameter of digraphs $D(q; m,n) = D_2(\F_q; X^mY^n)$.    They obtained precise values and good bounds on the diameter of these digraphs  for many instances of the parameters. For some of the results, the connection to Waring numbers over finite fields was utilized. The necessary and sufficient conditions for strong connectivity of $D(q;m,n)$ in terms of the arithmetic properties of $q,m,n$ appeared in \cite{KL15}.

Another interesting question about monomial digraphs is the isomorphism problem:  when is $D(q; m_1,n_1)$  isomorphic to $D(q;m_2,n_2)$?  A similar question for the bipartite monomial graphs $B(q;m,n)$ was answered in Theorem~\ref{isomono}.  For those graphs $B(q;m,n)$, just the count of 4-cycles resolves the isomorphism question for fixed $m,n$  and large $q$ (see \cite{Vig02}),  and the count of complete bipartite subgraphs gives the answer for all $q,m,n$ (see \cite{DLW07}).   In contrast, for the digraphs $D(q;m,n)$,  counting cycles of length from one (loops) to seven  is not sufficient: there exist digraphs for which these counts coincide, and which are not isomorphic (see  \cite{Kod14}).  In this regard, we would like to state the following problem and a conjecture. For any digraphs $A$ and $B$, let $|A(B)|$ denote the number of subdigraphs of $A$ isomorphic to $B$.

\begin{problem}
Are there  digraphs $D_1, \ldots, D_k$ such that any two monomial digraphs $D=D(q;m,n)$ and $D'=D(q';m',n')$  are isomorphic if and only if  $|D(D_i)|= |D'(D_i)|$ for each $i=1,\ldots, k$?
\end{problem}

Kodess and Lazebnik \cite{KL17} discussed several necessary conditions and several sufficient conditions for the isomorphism.
Though the sufficiency of the condition in the following conjecture is easy to verify, see  \cite{KL17}, its  necessity is still to be established.

\begin{conjecture}
\label{conj_isom} {\rm (\cite{Kod14})}
Let $q$ be a prime power.
The digraphs $D(q; m_1,n_1)$ and $D(q; m_2,n_2)$
are isomorphic
if and only if there exists $k$, coprime with $q-1$,
such that
\[
m_2 \equiv k m_1 \mod (q-1),
\]
\[
n_2 \equiv k n_1 \,\mod (q-1).
\]
\end{conjecture}
The conjecture is still not resolved even when $q$ is a prime.
As a bi-product of the work on isomorphism of  monomial digraphs,  Coulter,  De Winter, Kodess and Lazebnik \cite{CWKL19} established a peculiar result on the number of roots of certain polynomials over finite fields.  It can be considered as an application of the digraphs $D(q; m,n)$ to algebra.

\subsection{Multicolor Ramsey numbers}\label{SS:MRnum}
\bigskip

Let $k\ge 2$.
The {\it multicolor Ramsey number $r_n(G)$} is the minimum integer $N$ such that in any edge-coloring of the complete graph $K_N$ with $n$ colors, there is a monochromatic $G$.
Using a 4-cycle free graph $\Gamma_2= \Gamma_2({\mathbb F}_q; XY)$ with $q$ being an odd prime power,
Lazebnik and Woldar \cite{LW00} showed that $r_q(C_4)\ge q^2+2$.  It  compared well with
an upper bound by
Chung and Graham \cite{CG75},  which implied that  $r_q(C_4) \le q^2 + q + 1$.

Lazebnik and Mubayi \cite{LM02} generalized this result.
\begin{theorem} {\rm(\cite{LM02})}
 Let $p$ be a prime. Then
\begin{equation}  \label{rkK2tplus1}
tk^2 + 1 \le r_k(K_{2,t+1}) \le tk^2 + k + 2,
\end{equation}
where the lower bound holds whenever $t$ and $k$ are both prime powers of $p$. If $k$ is a prime power, then
$$r_k(C_4) \ge k^2 + 2.$$
\end{theorem}
The upper bound in this theorem follows from K\H{o}v\'ari,  S\'os and Tur\'an \cite{KST54}. For the lower bound, the construction was algebraic.  Recently,  Taranchuk \cite{Tar24private}  suggested that the construction of $K_{2,t+1}$-free graphs that are used to construct  the lower bound in (\ref{rkK2tplus1}) can be obtained as a particular case of a $\Gamma_n$-like graph.
Let $q=p^e$, where $p$ is prime and $e\ge 1$, and let $\F_q$ be a field of order $q$. The field $\F_q$ can be viewed as an $e$-dimensional vector space $\F_p^e$ over $\F_p$.
Given a positive integer $d$, $1\le d < e$, consider a linear map $f$ of $\F_p^e$ onto any of its $(e-d)$-dimensional  subspaces $W$.  As the kernel of $f$ has dimension $d$, it contains $p^d$ vectors. It is known that any function on $\F_q$ can be interpolated by a polynomial with one indeterminate of degree at most $q-1$. Without changing notation,  we denote the polynomial that interpolates the linear map $f$  by $f$ again. Consider the  graph $G$ with the vertex set $V(G) = \F_q \times W$ and two distinct vertices $(x_1,x_2)$ and $(y_1, y_2)$ being adjacent if
$$  x_2 + y_2 = f(x_1y_1).$$
It is clear that if $t=p^d$, then $n=|V(G)| = q\,|W|=  q^2/t$.  It is easy to verify that $G$ is $K_{2,t+1}$-free,  and, after the loops are deleted,  it is left with
$$\frac{\sqrt{t}}{2} n^{3/2} - \frac{\sqrt{tn}}{2}$$
edges.  Partitioning the edge set of the complete graph $K_n$ into copies of $G$ (Theorem \ref{T:decomp}), leads to the lower bound in (\ref{rkK2tplus1}).
\medskip

Li and Lih \cite{LL09} used Wenger graphs to determine the asymptotic behavior of the Ramsey number $r_n(C_{2k})=\Theta(n^{k/(k-1)})$ when $k\in \{2,3,5\}$.
For details,  and more on the multicolor Ramsey numbers, see
\cite{CG75,LW00,LW25},  and a survey by Radziszowski \cite{Rad24}.
\medskip

\subsection{Miscellaneous constructions and applications}
\bigskip

Two graphs on the same vertex set are {\it $G$-creating} if their union contains $G$ as a subgraph. Let $H(n,k)$ be the maximum number of pairwise $C_k$-creating hamiltonian paths of $K_n$. By the constructions of bipartite graphs with many edges and without even cycles in Reiman \cite{Rei58}, Benson \cite{Ben66} and Lazebnik, Ustimenko and Woldar \cite{LUW95}, Solt\'{e}sz \cite{Sol20} showed that
\[
n^{n/k-o(n)}\le H(n,2k)\le n^{(1-2/(3k^2-2k))n-o(n)}.
\]
Then Harcos and Solt\'esz \cite{HS20} used Ramanujan graphs to improve the upper bound to $n^{(1-\frac{1}{3k})n+o(n)}$,
and Byrne and Tait \cite{BT24} improved this upper bound to $n^{\frac{2}{3}n+o(n)}$ for $k=3$, to $n^{\frac{4}{5}n+o(n)}$ for $k=4,5$, and to $n^{(1-\frac{2}{3k})n+o(n)}$ for $k\ge 6$.
Mirzaei, Suk and Verstra\"ete \cite{MSV19} studied the effect of forbidding short even cycles in incidence graphs of point-line arrangements in the plane.
By modifying the construction in \cite{LU95},
the authors constructed an arrangement of $n$ points and $n$ lines in the plane, such
that their incidence graph has girth at least $k+5$, and determines at least $\Omega(n^{1+4/(k^2+6k-3)})$ incidences.
\medskip

\section{Some Tur\'an-type extremal problems for hypergraphs}\label{SS:hyper}
\bigskip

Let $ex_r(\nu,H)$ be the largest number of edges in an $H$-free $r$-uniform hypergraph on $\nu$ vertices. For $r=2$,  we use the notation for usual graphs:  $ex(\nu,H) =  ex_2(\nu,H)$.

For $k \geq 2$, a {\it cycle} (Berge cycle) in a hypergraph ${\mathcal H}$ is an
alternating sequence of vertices and edges of the form
$v_{1},E_{1},v_{2},E_{2},\dots ,v_{k},E_{k},v_{1}$, such that

\begin{tabular}{cl}
(i) & $v_{1},v_{2},\dots,v_{k}$ are distinct vertices of ${\mathcal H}$, \\
(ii) & $E_{1},E_{2},\dots,E_{k}$ are distinct edges of ${\mathcal H}$, \\
(iii) & $v_{i}, v_{i+1} \in E_{i}$ for each $i \in \{1,2,\ldots,
k-1\}$, and $v_{k}, v_{1} \in E_{k}$.
\end{tabular}

We refer to a cycle with $k$ edges as a {\it $k$-cycle},
and denote the family of all $k$-cycles by ${\mathcal C}_{k}$.
For example, a 2-cycle consists of a pair of vertices and a pair of edges such that
the pair of vertices is a subset of  each edge.
The above definition of a hypergraph cycle is the ``classical" definition
(see, for example, Duchet \cite{Duc85}).
For $r = 2$ and $k\ge 3$, it coincides with the definition of a cycle $C_{k}$ in
graphs and, in this case, ${\mathcal C}_{k}$ is a family consisting of precisely
one member.
The {\it girth} of a hypergraph ${\mathcal H}$, containing a cycle,
is the minimum length of a cycle in ${\mathcal H}$.

In \cite{LV03}, Lazebnik and Verstra\"ete considered the Tur\'an-type
extremal problem of
determining the maximum number of edges in an $r$-graph on $\nu$
vertices of girth five.  For graphs ($r = 2$), this is an old
problem of Erd\H os \cite{Erd75}.
The following inequalities are the best known:
 $$\frac{1}{2\sqrt{2}}\nu^{3/2} + \Omega(\nu^{5/4})\le ex(\nu, \{C_3, C_4\})\le \frac{1}{2} \nu \sqrt{\nu -1}, $$
 where the upper bound holds for all $\nu$ and can be derived easily,  and the lower bound is a recent result by Ma and Yang \cite{MY22} for  $\nu=2(q^2+q+1)$ and $q$ a prime power.

For bipartite graphs, on the other hand,
this maximum is $(1/2\sqrt{2})\nu^{3/2} + O(\nu)$ as $\nu \rightarrow
\infty$. Many attempts at reducing the gap between the constants
$1/2\sqrt{2}$ and $1/2$ in the main terms of the upper and lower bounds have not
succeeded so far. Tur\'{a}n-type
questions for hypergraphs  are generally harder than for graphs, and the following
result was  surprising,  as in this case the constants in the upper and lower bounds for the maximum turned out to be equal,  and the difference between the bounds was $O(\nu^{1/2})$.
\begin{theorem}\label{thLV} {\rm (\cite{LV03})}
Let ${\mathcal H}$ be a 3-graph on $\nu$ vertices and of girth at least
five. Then
\[ |{\mathcal H}| \le
\frac{1}{6}\nu\sqrt{\nu - \frac{3}{4}} + \frac{1}{12}\nu.\]
For any odd prime power $q\ge 27$, there exist 3-graphs ${\mathcal H}$
on $\nu = q^2$ vertices,
of girth five, with
\[
 |{\mathcal H}| = {{q+1}\choose 3} = \frac{1}{6}\nu^{3/2} -
 \frac{1}{6}\nu^{1/2}.
\]
\end{theorem}
In the context of this survey,  we wish to mention that the original construction for the lower bound
came from considering the following algebraically defined  3-graph ${\mathcal G}_q$  (Lazebnik--Verstra\"ete 3-graph),   of order $\nu =
q(q-1)$,  of girth five (for sufficiently large $\nu$) and number of edges $\sim \frac{1}{6}\nu^{3/2} -\frac{1}{4}\nu+o(\nu^{1/2})$, $\nu\to \infty$.  Let
$\F_{q}$ denote the finite field of odd characteristic,
and  let $S_{q}$ denote the set of points on the curve $2x_2 =
x_1^2$, where $(x_1,x_2)\in \F_{q} \times \F_{q}$. Define a hypergraph ${\mathcal G}_q$ as follows. The vertex
set of ${\mathcal G}_q$ is $\F_{q} \times \F_{q}
\setminus  S_{q}$. Three distinct vertices $a=(a_1,a_2)$,
$b=(b_1,b_2)$ and $c=(c_1,c_2)$ form an edge $\{a,b,c\}$ of ${\mathcal
G}_q$ if and only if the following three equations are satisfied:
\begin{eqnarray*}
a_{2} + b_{2} &=& a_{1}b_{1}, \\
b_{2} + c_{2} &=& b_{1}c_{1}, \\
c_{2} + a_{2} &=& c_{1}a_{1}.
\end{eqnarray*}
It is not difficult to check that ${\mathcal G}_q$ has girth at least
five for all odd $q$ and girth five for all sufficiently large
$q$. The number of edges in ${\mathcal G}_q$ is precisely ${q \choose
3}$, since there are ${q \choose 3}$ choices for distinct
$a_1,b_1$ and $c_1$, which uniquely specify $a_2,b_2$ and $c_2$
such that $a,b,c$ are not on the curve $2y = x^2$ and $\{a,b,c\}$
is an edge.

The idea to consider the hypergraph $\mathcal {H}_q$, whose edges are 3-sets of vertices of triangles in the polarity graph of $PG(2,q)$ with absolute points deleted, is due to Lov\'asz, see \cite{LV03} for details.   It raised the asymptotic lower bound to $\sim \frac{1}{6}\nu^{3/2} -\frac{1}{6}\nu+o(\nu^{1/2})$, $\nu\to \infty$, as stated in Theorem \ref{thLV}.  So, treating $\le$ in an asymptotic sense, we have
$$ \frac{1}{6}\nu^{3/2} -\frac{1}{6}\nu+o(\nu^{1/2}) \;\le \; ex_3(\nu, \{{\mathcal C}_{3}, {\mathcal C}_{4}\}) \; \le \;
\frac{1}{6}\nu\sqrt{\nu - \frac{3}{4}} + \frac{1}{12}\nu.$$

Mukherjee \cite{Muk24} used the construction of bipartite graphs without even cycles and with many edges in \cite{LUW95} to prove that
 $ex_3(\nu,\widetilde{C}_6)=\Theta(\nu^{7/3})$,
where $\widetilde{C}_6$ is the 3-uniform hypergraph obtained by adding a new vertex $x$ to $V(C_6)$ and $\widetilde{C}_6=\{e\cup{x} : e \in E(C_6)\}$.

A $k$-partite graph is said to be {\it balanced} $k$-partite if each partite set has the same number of vertices.
Let $ex(\nu,\nu,\nu, \mathcal F)$ be the maximum number of edges in balanced 3-partite graphs on partition classes of size $\nu$, which are $\mathcal F$-free.
Lv, Lu and Fang constructed balanced $3$-partite graphs with many edges, which are ${\mathcal C}_4$-free in \cite{LLF20} and $\{{\mathcal C}_3,{\mathcal C}_4\}$-free in \cite{LLF22}, and showed that
$$ex(\nu,\nu,\nu, {\mathcal C}_{4})=\Big(\frac{3}{\sqrt{2}}+o(1)\Big)\,\nu^{3/2}$$
and
$$ex(\nu,\nu,\nu, \{{\mathcal C}_3, {\mathcal C}_{4}\} )\ge \Big(\frac{6\sqrt{2}-8}{(\sqrt{2}-1)^{3/2}}+o(1)\Big)\,\nu^{3/2}.$$
For more on Tur\'an-type problems for graphs and hypergraphs,  see \cite{Bol78,Fur91,FS13}.
\medskip

\section{Graphs $D(k,q)$ and $C\!D(k,q)$}\label{SS:cdkq}
\bigskip

For any $k\ge 2$, and any prime power $q$, the bipartite graph
$D(k,q)$ is defined to be $B\Gamma _k(\mathbb {F}_q; f_2,\ldots, f_k)$,
where $f_2 = p_1l_1$,
$f_3= p_1l_2$, and for $4\le i\le k$,
\[
f_i=
\begin{cases}
-p_{i-2}l_1, & \text{for}\; i\equiv 0 \;\; {\rm or} \;\; 1
\pmod 4, \cr
p_1l_{i-2}, & \text{for}\; i\equiv 2 \;\; {\rm or} \;\; 3
\pmod 4.
\end{cases}
\]

It was shown that these graphs are
edge-transitive and, most importantly, the girth of $D(k,q)$ is at least $k+5$ for
odd $k$. It was shown in~\cite{LUW95} that
for $k\ge 6$ and $q$ odd, the graphs $D(k,q)$ are disconnected, and
the order of each component (any two being isomorphic) is at least
$2q^{k -\lfloor{\frac{k+2}{4}}\rfloor  +1}$. Let $C\!D(k,q)$
denote one of these components. It is the family of graphs
$C\!D(k,q)$ which provides the best lower bound mentioned before,
 being a slight improvement over the previous best lower
bound $\Omega({\nu}^{1 + \frac{2}{3k + 3}})$  given by the family
of Ramanujan graphs constructed by Margulis \cite{Mar88}, and
independently by Lubotzky, Phillips and Sarnak \cite{LPS88}.

The construction of the graphs
$D(k,q)$
was motivated by
attempts to
generalize the notion of
the biaffine part of
a generalized polygon,
and it was facilitated by
results in Ustimenko \cite{Ust91} on the
embedding of Chevalley group geometries
into their corresponding Lie algebras.  For a more recent exposition of these ideas,  see  \cite{UW03,Wol10,TW12}.

In fact, $D(2,q)$ and
$D(3,q)$ ($q$ odd) are exactly the biaffine parts of
a regular generalized
$3$-gon and $4$-gon, respectively
(see ~\cite{LU93}
for more details).
We wish to point out that $D(5,q)$ is not the biaffine part of the generalized hexagon.

As we mentioned before, the generalized $k$-gons exist only for $k=3,4,6$ (see \cite{FH64}),
therefore, $D(k,q)$ are not subgraphs of generalized $k$-gons for $k\ge 4$.

In this section we will discuss some basic  properties of these graphs.

\medskip

\subsection{Equivalent representation of $D(k,q)$}
\bigskip

The defining equations for the graph $D(k,q)$ have changed with time,  and the changes reflected better understanding of their automorphisms.  The following statement describes some transformations of the defining equations that lead to isomorphic graphs.

\begin{proposition} {\rm (\cite{Sun17,LSW17})} \label{Dkq:signs}
Let $k \ge 2$ and $a_1,\ldots,a_{k-1} \in \F_q^*$.  Let $H(k,q) = B\Gamma_k(q;f_2,\ldots,f_k) $
where $f_2 = a_1p_1l_1$, $f_3 = a_2p_1l_2$, and for $4\le i\le k$,
\[
f_i=
\begin{cases}
-a_{i-1}p_{i-2}l_1, &  {\rm for}\;\; i\equiv 0 \;\; {\rm or} \;\; 1
\pmod 4,\cr
a_{i-1}p_1l_{i-2}, & {\rm for }\;\; i\equiv 2 \;\; {\rm or} \;\; 3
\pmod 4.\cr
\end{cases}
\]
Then $H(k,q)$ is isomorphic to $D(k,q)$.
\end{proposition}
\begin{proof}
Let $\varphi : V(D(k,q)) \mapsto V(H(k,q))$ be defined via $(p) \rightarrow (x)$, and $[l] \rightarrow [y]$, where
\begin{align*}
x_1 = p_1 , \;\;\;& \;\;\; y_1 = l_1,\\
x_2 = a_1p_2, \;\;\;&\;\;\; y_2 = a_1l_2, \\
x_{2i+1} = a_{2i}a_{2i-2}\ldots a_2a_1p_{2i+1}, \;\;&\;\; y_{2i+1} = a_{2i}a_{2i-2}\ldots a_2a_1l_{2i+1},\\
x_{2i} = a_{2i-1}a_{2i-3}\ldots a_1p_{2i},\;\;&\;\; y_{2i} = a_{2i-1}a_{2i-3}\ldots a_1l_{2i}.\\
\end{align*}
Clearly, $\varphi$ is a bijection.  The verification that
 $\varphi$ preserves the adjacency is straightforward, and can be found in \cite{Sun17,LSW17}.
\end{proof}

Taking
\[
a_{i} = \begin{cases}
 &-1, \quad\quad {\rm for}\;\; i\equiv 0 \;\; {\rm or} \;\; 3
\pmod 4,\cr
 & 1, \ \ \quad\quad {\rm for}\;\; i\equiv 1 \;\; {\rm or} \;\; 2
\pmod 4,
 \end{cases}
 \]
 and using Proposition~\ref{Dkq:signs}, we see that $D(k,q)$ is isomorphic to $B\Gamma_k(q;f_2,\ldots,f_k)$
where $f_2 = p_1l_1$, $f_2=p_1l_2$, and for $4\le i\le k$,
\[
f_i=
\begin{cases}
p_{i-2}l_1, & \text{for }\; i\equiv 0 \;\; \text{ or} \;\; 1
\pmod 4,\cr
p_1l_{i-2}, & \text{for }\; i\equiv 2 \;\; \text{ or} \;\; 3
\pmod 4.
\end{cases}
\]
\underline{From now on, we will use this representation of  $D(k,q)$.}

\bigskip

Moreover, in the case of $q=2$,
\[
D(2,2)\cong C_8, D(3,2) \cong 2 C_8, D(4,2) \cong 4 C_8,
\]
and
\[
D(k,2) \cong 2^{k-3}C_{16},
\]
for $k\ge 5$. Here $nH$ denotes the union of $n$ disjoint copies of a graph $H$. Therefore we assume that $q\ge 3$ for the rest of this section.\\
\medskip

\subsection{Automorphisms of $D(k,q)$}
\bigskip

There are many automorphisms of $D(k,q)$, and below we  list the ones that are used to establish some properties of the graph.  It is a straightforward verification that the mappings we describe are
indeed automorphisms.
For more details,  see \cite{LU93,LU95}, F\" uredi, Lazebnik, Seress,
Ustimenko and Woldar \cite{FLSUW95}, Lazebnik, Ustimenko and Woldar \cite{LUW96} and Erskine \cite{Ers17}.
The automorphisms in these references
 may look different to the ones we list here
since we use another representation of the graph.

\subsubsection{Multiplicative automorphisms}
\bigskip

For any $a,b\in \mathbb{F}_q^*$, consider the map $m_{a,b}: \mathcal{P}_k\mapsto \mathcal{P}_k, \mathcal{L}_k\mapsto \mathcal{L}_k$
such that $(p) \xrightarrow{m_{a,b}} (p')$, and $[l]\xrightarrow{m_{a,b}} [l']$ where $p'_1 = ap_1$, $l'_1 = bl_1$, and for any $2\le i\le k$,
\[
p'_i =
\begin{cases}
a^{\lfloor \frac{i-1}{4} \rfloor +1}b^{\lfloor \frac{i}{4} \rfloor+1} p_i, & \rm{for}\;i \equiv  0{\rm ,}\;\; 1\;\;{\rm or}\;\;2
\pmod 4,\cr
a^{\lfloor \frac{i}{4} \rfloor +2}b^{\lfloor \frac{i}{4} \rfloor+1} p_i, &\rm{for}\; i \equiv 3
\pmod 4,
\end{cases}
\]
\[
l'_i =
\begin{cases}
a^{\lfloor \frac{i-1}{4} \rfloor +1}b^{\lfloor \frac{i}{4} \rfloor+1} l_i, & \rm{for}\;i \equiv  0{\rm ,}\;\; 1\;\;{\rm or}\;\;2
\pmod 4,\cr
a^{\lfloor \frac{i}{4} \rfloor +2}b^{\lfloor \frac{i}{4} \rfloor+1} l_i, &\rm{for}\; i \equiv 3
\pmod 4.
\end{cases}
\]
In Table~\ref{table:multi}, each entry illustrates how each coordinate is changed under the map $m_{a,b}$, i.e., the factor that the corresponding coordinate of a point
or a line is multiplied by. For example, $m_{a,b}$ changes $p_1$ to $ap_1$, $l_1$ to $bl_1$, both $p_{4t+3}$ and $l_{4t+3}$ to their product with $a^{t+2}b^{t+1}$.

\begin{table}[H]
\centering
\begin{tabular}{ |c|c|c|c|}
\toprule
    & $m_{a,b}$ & $  $ & $m_{a,b}$ \\
\hline
 $p_1$ & $* a$ & $l_1$ & $ * b$  \\
\hline
  $p_{4t}$ & $* a^tb^{t+1} $ & $l_{4t}$ &$ * a^tb^{t+1}$ \\
\hline
 $p_{4t+1}$& $* a^{t+1}b^{t+1}$ &$l_{4t+1}$ & $* a^{t+1}b^{t+1}$  \\
\hline
 $p_{4t+2}$& $* a^{t+1}b^{t+1}$ &$l_{4t+2}$ & $* a^{t+1}b^{t+1}$  \\
\hline
  $p_{4t+3}$& $* a^{t+2}b^{t+1}$ &$l_{4t+3}$ & $* a^{t+2}b^{t+1}$  \\
\bottomrule
\end{tabular}
\caption{Multiplicative automorphism}\label{table:multi}
\end{table}

\begin{proposition}\label{S:multiauto}
For any $a,b\in\mathbb{F}_q^*$, $m_{a,b}$ is an automorphism of $D(k,q)$.
\end{proposition}
\medskip

\subsubsection{Additive automorphisms}
\bigskip

For any $x \in \mathbb{F}_q$, and any $0\leq j\leq k$,
we define the map $t_{j,x} : \mathcal{P}_k\rightarrow \mathcal{P}_k, \mathcal{L}_k\rightarrow \mathcal{L}_k$ as follows.
\vskip 2mm
\begin{enumerate}
\item The map $t_{0,x}$ fixes the first coordinate of a line, whereas $t_{1,x}$ fixes the first coordinate of a point.
In Table~\ref{table:add1}, we illustrate how each coordinate is changed under the maps. If the entry is empty,
it means that this coordinate is fixed by the map. For example, the map $t_{1,x}$ changes the following coordinates
of a line according to the rule:  $l_1 \rightarrow l_1+x$, $l_4 \rightarrow l_4+l_2x$, $l_{2t} \rightarrow l_{2t}+l_{2t-3}x$ for
$t\ge 3$, and the following coordinates of a point according to the rule: $p_2 \rightarrow p_2+p_1x$, $p_4 \rightarrow p_4+2p_2x+p_1x^2,\cdots$.

\begin{table}[H]
\centering
\begin{tabular}{ |c|c|c|c|}
\toprule
&  $t_{0,x}$ & $t_{1,x}$ & $t_{2,x}$\\
 \hline
\;\;$p_1$\;\; & $+x$ & & \\
 \hline
\;\;$p_2$\;\; &  &$+p_1x$ & $+x$\\
 \hline
 \;\;$p_3$\;\; & $+p_2x$ &  & $-p_1x$ \\
  \hline
  \;\;$p_4$\;\; & & $+2p_2x+p_1x^2$ &  \\
  \hline
  \;\;$p_5$\;\; & $+p_4x$ & $+p_3x$ & $-p_2x$ \\
  \hline
 \;\;$p_{4t+1}$\;\; & $+p_{4t}x$  & $+p_{4t-1}x$ & $-p_{4t-3} x$  \\
 \hline
 \;\;$p_{4t+2}$\;\; &  & $+p_{4t-1}x$  & $+p_{4t-2}x$\\
 \hline
   \;\;$p_{4t+3}$\;\; & $+p_{4t+2}x$ &    & $-p_{4t-1}x$ \\
    \hline
   \;\;$p_{4t}$\;\; &  & \makecell{$+p_{4t-2}x+p_{4t-3}x$\\$+p_{4t-5}x^2$} & $+p_{4t-4}x$ \\
  \hline
  \;\;$l_1$\;\; & &$+x$ &  \\
  \hline
   \;\;$l_2$\;\; & $+l_1x$ &  & $-x$\\
   \hline
  \;\;$l_3$\;\; & $+2l_2x+l_1x^2$ & & \\
  \hline
   \;\;$l_4$\;\; & & $+l_2x$ & $+l_1x$\\
   \hline
   \;\;$l_5$\;\; & $+l_4x$ &  & $-l_2x$ \\
   \hline
   \;\;$l_{4t+1}$\;\; & $+l_{4t}x$ &  &$-l_{4t-3}x$\\
   \hline
   \;\;$l_{4t+2}$\;\; & $+l_{4t}x$ & $+l_{4t-1}x$ & $+l_{4t-2}x$\\
   \hline
   \;\;$l_{4t+3}$\;\; & \makecell{$+l_{4t+2}x+l_{4t+1}x$\\$+l_{4t}x^2$} & & $-l_{4t-1}x$ \\
   \hline
   \;\;$l_{4t}$\;\; &  &$+l_{4t-3}x$ & $+l_{4t-4}x$\\
\bottomrule
\end{tabular}
\caption{Additive automorphism}\label{table:add1}
\end{table}

\item For $2\le j\le k$, $t_{j,x}$ is a map which fixes the first $j-1$ coordinates of a point and a line.
In Table~\ref{tabel:add2}, we illustrate how each coordinate is changed under the corresponding map.

\begin{table}[ht]
\centering
\begin{tabular}{ |c|c|c|c|}
\toprule
    \multicolumn{4}{|c|}{ $j \equiv 0\;,\;1\pmod 4$} \\
\hline
   & $t_{j,x}$ &  & $t_{j,x}$\\
  \hline
     \makecell{$p_i$\\$i\le j-1$} & & \makecell{$l_i$\\$i\le j-1$} & \\
  \hline
  $p_j$ & $+x$ & $l_j$ & $-x$ \\
  \hline
  $p_{j+1+2t}$ & & $l_{j+1+2t}$&\\
  \hline
  $p_{j+2}$ & $-p_1x$  & $l_{j+2}$ &   \\
  \hline
  $p_{j+4}$ & $-p_2x$ & $l_{j+4}$ & $-l_2x$ \\
    \hline
    $p_{j+4+2t}$ & $-p_{2t+1}x$  &$l_{j+4+2t}$&  $-l_{2t+1}x$\\
\bottomrule
\end{tabular}
\vskip 5mm
\begin{tabular}{ |c|c|c|c|}
\toprule
    \multicolumn{4}{|c|}{$j \equiv 2\;,\;3\pmod 4$}\\
\hline
   & $t_{j,x}$ &  & $t_{j,x}$ \\
 \hline
   \makecell{$p_i$\\$i\le j-1$} & & \makecell{$l_i$\\$i\le j-1$} & \\
      \hline
  $p_j$ & $+x$ & $l_j$ & $-x$  \\
   \hline
$p_{j+1+2t}$ & & $l_{j+1+2t}$ & \\
   \hline
 $p_{j+2}$&  & $l_{j+2}$ &$+l_1x$    \\
    \hline
   $p_{j+4}$ & $+p_2x$ & $l_{j+4}$ & $+l_2x$ \\
    \hline
$p_{j+4+2t}$ & $+p_{2t+2}x$ &$l_{j+4+2t}$ & $+l_{2t+2}x$ \\
\bottomrule
\end{tabular}
\caption{Additive automorphism (continued)}\label{tabel:add2}
\end{table}

\end{enumerate}
\vskip 5mm
\begin{proposition}\label{S:addiauto}
For any $x\in\mathbb{F}_q$, and any $0\leq j\leq k$, $t_{j,x}$ is an automorphism of $D(k,q)$.
\end{proposition}
\medskip

\subsubsection{Polarity automorphism}\label{polarauto}
\bigskip

Consider the map $\phi:\mathcal{P}_k\rightarrow \mathcal{L}_k,\mathcal{L}_k \rightarrow \mathcal{P}_k$ such that
\[
(p_1,p_2,p_3,p_4,\ldots,p_{k-1},p_k)\xrightarrow{\phi}
\begin{cases}
 [p_1,p_2,p_4,p_3,\ldots,p_{k},p_{k-1}],  & \text{ if $k$ is even,} \cr
 [p_1,p_2,p_4,p_3,\ldots,p_{k-1},p_{k-2},p_k], &\text{\ if $k$ is odd,}
\end{cases}
\]
and
\[
[l_1,l_2,l_3,l_4,\ldots,l_{k-1},l_k]\xrightarrow{\phi}
\begin{cases}
 (l_1,l_2,l_4,l_3,\ldots,l_{k},l_{k-1}),  & \text{ if $k$ is even,} \cr
 (l_1,l_2,l_4,l_3,\ldots,l_{k-1},l_{k-2},l_k), &\text{\ if $k$ is odd.}
\end{cases}
\]
The proof of the following proposition is straightforward. \begin{proposition}\label{S:polarauto}
If $k$ is even, or $q$ is even, then $\phi$ is an automorphism of $D(k,q)$.
\end{proposition}
\begin{theorem}\label{S:transitive}{\rm(\cite{LU95})}
For any integer $k\geq 2$, and any prime power $q$, the automorphism group of  $D(k,q)$ is transitive on $\mathcal{P}_k$, transitive on $\mathcal{L}_k$, and the graph is edge-transitive.
If any one of $k$ and $q$ is even, then $D(k,q)$ is vertex-transitive.
\end{theorem}

Moreover, the automorphism group of $D(k,q)$ acts transitively on the set of  paths of length 3 (3-paths). This useful fact appeared implicitly in several papers, and it was rediscovered independently and stated explicitly in \cite{TW12,Sun17,Ers17}.  For $k\ge 4$ with $k\not\equiv 3 \pmod{4}$ and any prime power $q$, the automorphism group of $D(k,q)$ acts transitively on the set of all ordered 3-paths (see \cite{Ers17}).
It will be discussed in Section \ref{conDkq} that, except for finitely many $k$, the graph $D(k,q)$ is disconnected and all its components are isomorphic. The components are denoted by
$C\!D(k,q)$. The order of the automorphism group of $C\!D(k,q)$, for some small $k$ and $q$, was computed in \cite{Ers17}, and the following conjecture appeared there.
\begin{conjecture}\label{erskinethesis}{\rm (\cite{Ers17})}  Let $k\ge 3$, $m= k- \lfloor\frac{k-2}{4} \rfloor$ and let $q=p^e$ be an odd prime power larger than 3.  Then the automorphism group of the  graph  $C\!D(k,q)$ has order exactly
\[
\begin{cases}
 eq^{m+1}(q-1)^2,  & \text{ if }    q\equiv 3  \pmod{4},  \\
 2eq^{m+1}(q-1)^2,  & \text{ if} \;  q\equiv 0,1,2  \pmod{4}.
\end{cases}
\]
\end{conjecture}
This count suggests that the stabilizer of the 3-path  $$[1,0,\ldots,0]\sim (0,0,\dots, 0)\sim [0,0,\ldots, 0]\sim (1,0,\ldots, 0) $$  in the action of the automorphism group of $C\!D(k,q)$ on its 3-paths   is generated by the multiplicative automorphisms and the Frobenius automorphisms of the field. The results in  \cite{Vig02} on the structure of the automorphism group of  $D(2,q)$ and $D(3,q)$  support this.    For $q=3$,  the number of components of $D(k,3)$ is not well-understood.
\begin{problem}  Determine the automorphism group  of the graph $C\!D(k,q)$.
\end{problem}

\subsection{Girth of $D(k,q)$}
\bigskip

Lazebnik and Ustimenko in~\cite{LU95}
showed that $\girth(D(k,q)) \ge k+5$ for
odd~$k$, and  $\girth(D(k,q)) \ge k+4$ for even~$k$.

\begin{theorem}\label{S:girthLB}~{\rm(\cite{LU95})} Let $k\ge 2$ be an integer, and let $q$ be a prime power. Then $\girth(D(k,q)) \ge k+5$ if $k$ is odd, and
$\girth(D(k,q)) \ge k+4$ if $k$ is even.
\end{theorem}

Recently, Taranchuk \cite{Tar24} presented a simpler proof of this theorem. Here are the main ideas of his proof.
The algebraically defined graph $A(n,q)=B\Gamma_n(q;f_2,\ldots,f_n)$
was introduced by Ustimenko in \cite{Ust04,Ust07,Ust13},
where
$$
f_i= \begin{cases}p_{i-1} \ell_{1}, & \text { if } i\text { is even, } \\
p_{1} \ell_{i-1}, & \text { if } i \text { is odd, }\end{cases}
$$
for $2 \leq i \leq n$.  Let $(0)$ denote the point corresponding to the zero vector.
The point $(0)$ in $A(n, q)$ is not contained in any cycle of length less
than $2n + 2$. Then there is a covering map from $D(2k + 1, q)$ to
$A(k + 2, q)$ which maps the point $(0)$ in $D(2k + 1, q)$ to the point $(0)$ in $A(k + 2, q)$. By transitivity of $D(2k + 1, q)$
on the set of points and the properties of covering maps, no cycle of length less than $2k + 6$ can appear in $D(2k + 1, q)$.

\medskip

The following conjecture was stated
in~\cite{FLSUW95} for all $q\ge 5$, and here we extend it
to the case where $q=4$.
\begin{conjecture}\label{S:girthconj}
The graph $D(k,q)$ has girth $k+5$ for odd $k$ and girth $k+4$ for even $k$, and all prime powers $q\ge 4$.
\end{conjecture}
The conjecture is wide open, and it was confirmed only for a few infinite families of $k$ and $q$,
see Schliep \cite{Sch94}, Thomason \cite{Tho97,Tho02}, Xu, Cheng and Tang \cite{XCT23} and Xu \cite{Xu23}.

For $q=2$, the girth of $D(k,2)$ is 8 if $k=2,3,4$, and 16 if $k\ge 5$.
For $q=3$, the girth of $D(k,3)$ exhibits different behavior, and we do not understand it completely.
The known results on the girth for $2\le k\le 320$ are summarized by Xu, Cheng and Tang in  \cite{XCT23a}.
They are shown in the following table with either the exact values of the girth  or upper bounds for the girth, where $[m,n]$ denotes the set of all  integers $k$ such that $m\le k\le n$.
Note that the lower bound of the girth is $k+5$ for odd $k$, and $k+4$ for even $k$.

\begin{table}[htbp]
\centering
\begin{tabular}{|c|c|c|c|c|c|c|c|}
\toprule
$k$ &2   & 3 & $[4,8]$    &$[9,14] $    & $[15, 16]$ &$[17,19]$& $[20,24]$ \\
\hline
girth & 6 & 8 & 12  &18   & 20&24&28  \\
\hline
$k$ &$[25,26]$ & $[31,32]$&35 & 39&$[49,50]$    &51  &55 \\
\hline
girth &34&36& 40 &$\le 48$& 54  &56 &$\le 68$ \\
\hline
$k$  & 67 &71&79&[103,104]& 111 & 135&143     \\
\hline
girth & 72&$\le 80$ &$\le 96$&108 & $\le136$ &$\le 144$& $\le 160$ \\
\hline
$k$&[157,158]  &159  &211 &223&271&287& [319,320]  \\
\hline
girth  &162 &$\le 192$  & 216& $\le 272$&$\le 288$&$\le 320$&324\\
\bottomrule
\end{tabular}
\caption{Known girth of $D(k,3)$ for $2\le k\le 320$ }
\end{table}

\begin{problem}
Determine the girth of $D(k,3)$ for all $k\ge 2$.
\end{problem}

Conjecture~\ref{S:girthconj} was proved only for infinitely many pairs of $(k,q)$. The following results describe all of them.
\medskip

\begin{theorem}\label{F:girth}{\rm{(\cite{FLSUW95})}}
For any $k\ge 3$ odd, and $q$ being a member of the arithmetic progression $\{1+n(\frac{k+5}{2})\}_{n\ge 1}$,
\[
\girth(D(k,q)) = k+5.
\]
\end{theorem}
\begin{remark}
The theorem could be extended for even $k\ge 2$ and $q$ being a member of the arithmetic progression $\{1+n(\frac{k+4}{2})\}_{n\ge 1}$, and in this case $\girth\bigl(D(k,q)\bigr) = k+4$. The proof is essentially the same as the proof in \cite{FLSUW95}, and we omit it.
\end{remark}
By modifying an idea from  \cite{FLSUW95}, this result was strengthened in \cite{Sun17}.
\begin{theorem}\label{S:girthk3}{\rm{(\cite{Sun17})}}
For any $k\ge 3$ with $k\equiv 3\pmod 4$, and $q$ being a member of the arithmetic progression $\{1+n(\frac{k+5}{4})\}_{n\ge 1}$,
\[
\girth\bigl(D(k,q)\bigr) = k+5.
\]
\end{theorem}

Cheng, Chen and Tang found other sets of pairs $(k,q)$ for which the girth of $D(k,q)$ could be determined precisely, see \cite{CCT14,CCT16}. See also the aforementioned papers \cite{XCT23,XCT23a}. Their results are as follows.

\begin{theorem}\label{F:girth14}{\rm{(\cite{CCT14})}}
For any $q\ge 4$, and any odd $k$ such that $(k+5)/2$ is a power of the characteristic of $\mathbb{F}_q$,
\[
\girth(D(k,q)) = k+5.
\]
\end{theorem}

\begin{theorem}\label{F:girth16}{\rm{(\cite{CCT16})}}
For any prime $p$, and any positive integers $h,m,s$ with $h | (p^m~-~1)$ and $hp^s > 3$,
\[
\girth(D(2hp^s-4,p^m)) = \girth(D(2hp^s-5,p^m) )= 2hp^s.
\]
\end{theorem}

\begin{theorem}\label{F:girth23}{\rm{(\cite{XCT23})}} For any $q>3$,
\[
\girth(D(3,q))=\girth(D(4,q))=8 \mbox{  and  }\girth(D(5,q))=10.
\]
\end{theorem}

\begin{theorem}\label{F:girth23a}{\rm{(\cite{XCT23a})}}
\rm{(i)} $\girth(D(4t+2,q))=\girth(D(4t+1,q))$.\\
(ii) $\girth(D(4t+3,q))=4t+8$ if $\girth(D(2t,q))=2t+4$.\\
(iii) $\girth(D(8t,q))=8t+4$ if $\girth(D(4t-2,q))=4t+2$.\\
(iv) $\girth(D(2^{s+2}t-5,q)) =2^{s+2}t$ if $p\ge 3, 2^s\mid (q-1), 2^{s+1} \nmid(q-1),
2\nmid t$ and $t\mid_p(q-1)$,
where $t\mid_p(q-1)$ denotes $t\mid (q-1)p^r$ for some $r\ge 0$.
\end{theorem}

We wish to note that part (iv) of Theorem \ref{F:girth23a} can be easily obtained from part (ii) and Theorem~\ref{F:girth16}.
\bigskip

Suppose that the girth of $D(k,q)$ satisfies Conjecture~\ref{S:girthconj}. Then the following theorem allows us to determine the exact value of the girth of $D(k',q)$
for infinitely many values of $k'$,
namely, $k' = p^m\girth\bigl(D(k,q)\bigr)-5$  and  $k' = p^m\girth\bigl(D(k,q)\bigr)-4$ for an arbitrary positive integer $m$ if $k\not\equiv 3\pmod 4$, and $k' = p^m\girth\bigl(D(k,q)\bigr)-5$ for an arbitrary positive integer $m$ if $k\equiv 3\pmod 4$.

\begin{theorem}\label{S:girthlift}{\rm{(\cite{Sun17})}}
Let $p$ be the characteristic of $\mathbb{F}_q$ and $g_k = \girth(D(k,q))$, where $k\ge 3$.  Suppose that $g_k$ satisfies Conjecture~\ref{S:girthconj}. Then
$$\girth\bigl(D(pg_k-5,q)\bigr) = pg_k.$$
In addition, if $k\not\equiv 3\pmod 4$, then the following also holds:
$$\girth\bigl(D(pg_k-4,q)\bigr) = pg_k.$$
\end{theorem}

By Theorems~\ref{F:girth},~\ref{S:girthk3},~and \ref{S:girthlift}, Conjecture~\ref{S:girthconj} is true for $(k+5)/2$
 being the product of a factor of $q-1$ which is at least 4 and a power of the characteristic of $\mathbb{F}_q$,
 and for $(k+5)/4$ being the product of a factor of $q-1$ which is at least 2 and a power of the characteristic
  of $\mathbb{F}_q$.

The known values of girth of $D(k,q)$ for $2\le k\le 100$ and $4\le q\le 97$ are shown as follows,
which are proven either by computer or Theorems
\ref{F:girth}-\ref{S:girthlift}, and the empty cells correspond to the cases where the girth is unknown.

\begin{table}[htbp]
\centering
\scalebox{0.95}{\begin{tabular}{ |c|c|c|c|c|c|c|c|c|c|c|c|c|c|c|c|}
\toprule
\diagbox{$q$}{$k$} &2   & 3 & 4   & 5   & 6 & 7     &8    &9     & 10 & 11 & 12 &13  & 14&15&16\\
\midrule
4 &6    & 8   & 8 & 10  &10 &12  &12 & 14 & 14 & 16 & 16  & 18 &18&20&20\\
\hline
5 &6    & 8   & 8 & 10  &10 &12  &12 & 14 & 14 & 16 & 16  & 18 &18&20&20\\
\hline
7 & 6 & 8    &8  &10 &10  &12 & 12  & 14& 14 &16 & 16& 18 &18&20&20 \\
\hline
8 & 6 & 8    &8   &10 &10  &12 & 12  & 14&14 &16 & 16& 18 &18&20 &20\\
\hline
9 & 6 & 8    &8   &10 &10  &12 & 12  &14 & 14 & 16&16 &18  &18 &20&20\\
\hline
11 & 6 & 8    &8   &10 &10  &12 & 12  & &  & 16& 16& 18 & 18&20&20\\
\hline
13 & 6 & 8    &8   &10 &10  &12 & 12  & 14 & 14  & 16& 16& 18 &18 &20&20\\
\hline
16 & 6 & 8    &8   &10 &10 &12 & 12  & 14 & 14  &16 & 16&18  &18 &20&20\\
\hline
17 & 6 & 8    &8   & 10& 10 &12 &  12 & 14 & 14 & 16& 16& 18 &18 &20&20\\
\hline
19 & 6 & 8    &8   & 10&10  &12 &12   & 14 & 14  & 16& 16& 18 &18& 20&20\\
\hline
23& 6 &8     & 8 &10 & 10 & 12& 12  & &  & 16&16 & 18 &18  & 20&20\\
 \hline
25 & 6 & 8    &8  &10 & 10 &12 &12   & 14 & 14  &16 &16 & 18 & 18 & 20&20\\
 \hline
27& 6 & 8    & 8 & 10& 10 &12 & 12  & &  & 16& 16& 18 &18  &20 &20\\
 \hline
29 & 6 &  8   &8  &10 & 10 & 12& 12  &14 & 14 & 16&16 &18  &18  &20 &20\\
 \hline
31&  6&  8   &  8&10 & 10 & 12& 12  & 14 & 14  & 16& 16&18  &18  & 20&20\\
 \hline
 32& 6 & 8    & 8 &10 & 10 & 12& 12  & &  &16 &16 & 18 &  18&20 &20\\
 \hline
37& 6 &   8  & 8 &10 & 10 & 12& 12    & 14 & 14  &16 &16 & 18 & 18 & 20&20\\
 \hline
41 & 6 &  8   &  8&10 & 10 & 12& 12    & 14 & 14  &16 &16& 18 &18  &20&20 \\
 \hline
43 & 6 &  8   & 8 &10 & 10 & 12& 12    & 14&14  & 16&16 &18  & 18 &20 &20\\
 \hline
47 & 6 &  8   & 8 & 10&  10 & 12& 12  & &  & 16& 16&  18&18  &20 &20\\
 \hline
49 & 6 &   8  &8  & 10& 10 & 12& 12   &14 &14  & 16&16 &18  &18  &20&20 \\
 \hline
53 & 6 &  8   & 8 &10 & 10 & 12& 12  & 14 & 14  & 16& 16& 18 &18  &20 &20\\
 \hline
 59& 6 &  8   & 8 & 10& 10 & 12& 12    & &  &16 & 16&   18 &18    &20 &20\\
 \hline
 61& 6 &   8  &  8& 10& 10 & 12& 12    & 14 & 14  & 16& 16&  18 &18    &20 &20\\
 \hline
64& 6 &   8  & 8 &10 & 10 & 12& 12    &14 & 14 &16 & 16& 18 &18  &20 &20\\
 \hline
67& 6 &    8 & 8 &10 &  10 & 12& 12   & 14 & 14  &16 &16 &   18 &18    &20 &20\\
 \hline
71 & 6 &   8  & 8 & 10&  10 & 12& 12   &14 & 14 & 16& 16&   18 &18    & 20&20\\
 \hline
 73 & 6 &   8  & 8 &10 & 10 & 12& 12   & 14 & 14  & 16&16 & 18 & 18 &20 &20\\
 \hline
79 & 6 &   8  &8  &10 & 10 & 12& 12   & 14 & 14  &16 &16 &   18 &18   &20 &20\\
 \hline
81 & 6 &   8  &8  &10 &  10 & 12& 12    & 14 & 14 & 16&16 &18  & 18 &20 &20\\
 \hline
83 & 6 &   8  &  8&10 & 10 & 12& 12   & &  & 16&16 &   18 &18   & 20&20\\
 \hline
89 & 6 &    8 &  8& 10& 10 & 12& 12   & 14 & 14  & 16&16 &   18 &18    & 20&20\\
 \hline
97 & 6 &    8 &8  &10 &  10 & 12& 12  & 14 & 14  &16 &16 &   18 &18    &20 &20\\
\bottomrule
\end{tabular}}
\caption{Girth of $D(k,q)$ for $2\le k\le 16$ and $4\le q\le 97$ }
\end{table}

\begin{table}[htbp]
\centering
\scalebox{0.95}{\begin{tabular}{ |c|c|c|c|c|c|c|c|c|c|c|c|c|c|c|}
\toprule
\diagbox{$q$}{$k$} &17   & 18 & 19  & 20   & 21 & 22     &23   &24     & 25 & 26 & 27 &28  & 29&30\\
\midrule
4 &  22 & 22   & 24 & 24  & 26& 26 &28 & 28 & &  & 32  & 32 &  &    \\
\hline
5 &   22 & 22   & 24 &  24 & 26& 26  &28 & 28 & &  & 32  & 32 &  &    \\
\hline
7&   22 & 22   & 24 & 24  & 26& 26  & 28& 28& &  & 32  & 32 &  &     \\
\hline
8 &   22 & 22   & 24 &  24 & 26& 26  & 28& 28 & &  &  32 &  32&  &    \\
\hline
9 &   22 & 22    & 24 & 24  & 26& 26  & 28& 28 &30 & 30 & 32  &  32&  &   \\
\hline
11 & 22  &22    & 24 & 24 & 26& 26  & 28& 28  & &  & 32  &32  &  &   \\
\hline
13 &  22 & 22   & 24 &24   &26 &26  & 28& 28 & &  & 32  & 32 &  &   \\
\hline
16 &   22 & 22   & 24 &  24 &26& 26  & 28& 28 & 30&30  &  32 & 32 &  &    \\
\hline
17 &   22 & 22    & 24 &  24 & &  & 28& 28  & &  &  32 &32  &  34&  34  \\
\hline
19 &   22 & 22    & 24 & 24  & &  & 28& 28 & &  & 32  & 32 &  &    \\
\hline
23& 22&  22 & 24   & 24 &   & &  28& 28 &  & & 32 & 32  &  &   \\
 \hline
25 &   22 & 22   &  24&  24 & 26& 26  & 28& 28 &30 & 30 & 32  & 32 &  &   \\
 \hline
27&   22 & 22    & 24 & 24  &26 & 26 & 28& 28  & 30&30  &  32 & 32 &  &    \\
 \hline
29&  22 & 22    & 24 &24   & &  & 28& 28 & &  &  32 & 32 &  &    \\
 \hline
31&   &    & 24 & 24  & &  & 28& 28  & 30&  30&  32 &32  &  &    \\
 \hline
 32&   22 & 22   &24  &  24 & &  & 28& 28& &  &32   &32  &  &    \\
 \hline
37&   &    & 24 & 24  & &  &  28& 28 &  &  & 32  & 32 &  &   \\
 \hline
41&   &    & 24 & 24  & &  &  28& 28 & &  & 32  &  32&  &     \\
 \hline
43&   &    & 24 & 24  & &  & 28& 28 & &  &  32 &32  &  &    \\
 \hline
47 &   &    & 24 & 24  &  & &  28& 28 & & & 32  &32  &  &    \\
 \hline
49&   22 & 22   & 24 & 24  & 26& 26  & 28& 28 & &  & 32  &32  &  &     \\
 \hline
53&   &    & 24 & 24  &26 & 26 & 28& 28  & &  &  32 & 32 &  &    \\
 \hline
 59&   &    &24  & 24  & &  & 28& 28 & &  & 32  & 32 &  &   \\
 \hline
 61&   &    & 24 & 24  & &  & 28& 28  &30 &  30&  32 &32  &  &    \\
 \hline
64&   22 & 22   & 24 &  24 & 26& 26  & 28&  28& &  & 32  &32  &  &   \\
 \hline
67&  22 &   22 & 24 &  24 & &  & 28& 28  & &  & 32  &32  &  &    \\
 \hline
71 &   &    & 24 &  24 & &  & 28& 28 & &  &32   &32  &  &    \\
 \hline
 73&   &    & 24 & 24  & &  & 28& 28 & &  & 32  & 32 &  &    \\
 \hline
79&   &    & 24 & 24  & 26& 26 & 28& 28  & &  & 32  & 32 &  &   \\
 \hline
81&   22 & 22   & 24 & 24  & 26& 26  & 28& 28  &30 & 30 &  32 & 32 &  &   \\
 \hline
83&   &    & 24 &  24 & &   & 28& 28&  &  &  32& 32 &  &   \\
 \hline
89 &  22 & 22   &24  & 24  & &  &28& 28  & &  &  32 & 32 &  &   \\
 \hline
97 &   &    & 24 &24   & &   &  28& 28 & & & 32  &32  &  &  \\
 \bottomrule
\end{tabular}}
\caption{Girth of $D(k,q)$ for $17\le k\le 30$ and $4\le q\le 97$}
\end{table}

\begin{table}[H]
\scalebox{0.95}{\begin{tabular}{ |c|c|c|c|c|c|c|c|c|c|c|c|c|c|c|}
\toprule
\diagbox{$q$}{$k$} &31   & 32 & 33  & 34   &35 & 36    &37   &38     & 39 & 40 & 41 &42  & 43&44\\
\midrule
4 & 36  & 36   &  &   &40 & 40 & &  & 44&44  &   &  & 48 &   48 \\
\hline
5 &  36 & 36   &  &   &40 & 40 & &  &44 & 44 &   &  & 48 &   48    \\
\hline
7& 36  & 36   &  &   &40 & 40 &42 &42  & 44&44  &   &  & 48 &   48  \\
\hline
8&  36 & 36   &  &   &40 & 40 & &  & 44&44  &   &  & 48 &   48   \\
\hline
9& 36  & 36   &  &   &40 & 40 &42 & 42 & 44& 44 &   &  & 48 &   48  \\
\hline
11&  36 &  36  &  &   &40 & 40 & &  &44 & 44 &   &  & 48 &   48 \\
\hline
13& 36  &  36  &  &   &40 & 40 & &  &44 & 44 &   &  & 48 &   48 \\
\hline
16&  36 &  36  &  &   &40 & 40 & &  &44 &44  &   &  & 48 &   48   \\
\hline
17& 36  &  36  &  &   &40 & 40 & &  &44 & 44 &   &  & 48 &   48    \\
\hline
19 &36   &  36  &38  & 38  &40 & 40 & &  &44 &44  &   &  & 48 &   48\\
\hline
23&  36 & 36   &  &   &40 & 40 & &  & 44& 44 & 46  & 46 & 48 &   48   \\
 \hline
25 &  36 &  36  &  &   &40 & 40 & &  & 44& 44 &   &  & 48 &   48\\
 \hline
27& 36  & 36   &  &   &40 & 40 & &  & 44& 44 &   &  & 48 &   48    \\
 \hline
29& 36  &  36  &  &   &40 & 40 & &  &44 & 44 &   &  & 48 &   48  \\
 \hline
31&  36 &  36  &  &   &40 & 40 & &  & &  &   &  & 48 &   48   \\
 \hline
 32& 36  & 36   &  &   &40 & 40 & &  &44 & 44 &   &  & 48 &   48 \\
 \hline
37&  36 &  36  &  &   &40 & 40 & &  & &  &   &  & 48 &   48 \\
 \hline
41& 36  &  36  &  &   &40 & 40 & &  & &  &   &  & 48 &   48     \\
 \hline
43& 36  &  36  &  &   &40 & 40 &42 & 42 & &  &   &  & 48 &   48    \\
 \hline
47&  36 &  36  &  &   &40 & 40 & &  & &  &  46 & 46 & 48 &   48   \\
 \hline
49&  36 &  36  &  &   &40 & 40 &42 &42  & 44& 44 &   &  & 48 &   48  \\
 \hline
53& 36  &  36  &  &   &40 & 40 & &  & &  &   &  & 48 &   48 \\
 \hline
 59&  36 &  36  &  &   &40 & 40 & &  & &  &   &  & 48 &   48  \\
 \hline
 61& 36  &  36  &  &   &40 & 40 & &  & &  &   &  & 48 &   48   \\
 \hline
64&  36 &36    &  &   &40 & 40 &42 & 42 &44 & 44 &   &  & 48 &   48 \\
 \hline
67&  36 &  36  &  &   &40 & 40 & &  & 44&44  &   &  & 48 &   48   \\
 \hline
71 & 36  &  36  &  &   &40 & 40 & &  & &  &   &  & 48 &   48    \\
 \hline
 73&  36 &  36  &  &   &40 & 40 & &  & &  &   &  & 48 &   48   \\
 \hline
79&  36 &  36  &  &   &40 & 40 & &  & &  &   &  & 48 &   48   \\
 \hline
81& 36  & 36  &  &   &40 & 40 &42 & 42 &44 & 44 &   &  & 48 &   48   \\
 \hline
83&  36 &  36  &  &   &40 & 40 & &  & &  &   &  & 48 &   48  \\
 \hline
89 & 36  &  36  &  &   &40 & 40 & &  &44 &44  &   &  & 48 &   48  \\
 \hline
97&  36 &   36 &  &   &40 & 40 & &  & &  &   &  & 48 &   48\\
 \bottomrule
\end{tabular}}
\caption{Girth of $D(k,q)$ for $31\le k\le 44$ and $4\le q\le 97$}
\end{table}

\begin{table}[H]
\scalebox{0.95}{\begin{tabular}{ |c|c|c|c|c|c|c|c|c|c|c|c|c|c|c|}
\toprule
\diagbox{$q$}{$k$} &45   & 46 & 47  &48  &49 & 50   &51 &52   & 53& 54 & 55 &56  & 57&58\\
\midrule
4 &   &    &  52&  52 & &  & 56& 56 & &  &   &  &  &    \\
\hline
5 &  50 & 50   & 52&  52  & &  &56 &  & &  & 60  &60  &  &    \\
\hline
7&   &    &  52&  52  & &  & 56&56  & &  &   &  &  &     \\
\hline
8 &   &    &  52&  52   & &  &56 &56  & &  &   &  &  &    \\
\hline
9 &   &    &  52&  52  &54 & 54 &56 &  & & &60   &60  &  &   \\
\hline
11 &   &    &  52&  52   & &  &56 &  & &  &   &  &  &   \\
\hline
13 &   &    & 52 &52  & &  &56 &  & &  &   &  &  &   \\
\hline
16 &   &    &  52&  52   & &  & 56&  & &  &  60 &60 &  &    \\
\hline
17 &   &    &  &   & &  &56 &  & &  &   &  &  &    \\
\hline
19 &   &    &  &   & &  &56 &  & &  &   &  &  &    \\
\hline
23& &   &    &  &   & & 56 & &  & &  &   &  &   \\
 \hline
25 & 50  &  50  &  52&  52   & &  & 56&  & &  &  60 & 60 &  &   \\
 \hline
27&   &    &52  &52   & 54& 54 &56 &  & &  & 60 &60  &  &    \\
 \hline
29&   &    &  &   & &  &56 &  56& 58& 58 &   &  &  &    \\
 \hline
31&   &    &  &   & &  & 56&  & &  & 60  & 60 &62  &  62  \\
 \hline
 32&   &    &  &   & &  & 56&  & &  &   &  &  62&   62 \\
 \hline
37&   &    &  &   & &  & 56&  & &  &   &  &  &   \\
 \hline
41&   &    &  &   & &  &56 &  & &  &   &  &  &     \\
 \hline
43&   &    &  &   & &  & 56&  56& &  &   &  &  &    \\
 \hline
47 &   &    &  &   & &  & 56&  & &  &   &  &  &    \\
 \hline
49&   &    &  52&  52  & &  & 56& 56 & &  &   &  &  &     \\
 \hline
53&   &    & 52 & 52  & &  &56 &  & &  &   &  &  &    \\
 \hline
 59&   &    &  &   & &  &56 &  &58 & 58 &   &  &  &   \\
 \hline
 61&   &    &  &   & &  &56 &  & &  &  60 &60  &  &    \\
 \hline
64&   &    & 52&  52   & &  & 56& 56 & &  &   &  &  &   \\
 \hline
67&   &    &  &   & &  & 56&  & &  &   &  &  &    \\
 \hline
71 &   &    &  &   & &  &56 &56  & &  &   &  &  &    \\
 \hline
 73&   &    &  &   & &  &56 &  & &  &   &  &  &    \\
 \hline
79&   &    & 52 & 52  & &  &56 &  & &  &   &  &  &   \\
 \hline
81&   &    &  52&  52   & 54& 54 &56 &  & &  & 60  & 60 &  &   \\
 \hline
83&   &    &  &   & &  &56 &  & &  &   &  &  &   \\
 \hline
89 &   &    &  &   & &  &56 &  & &  &   &  &  &   \\
 \hline
97 &   &    &  &   & &  &56 &  & &  &   &  &  &  \\
 \bottomrule
\end{tabular}}
\caption{Girth of $D(k,q)$ for $45\le k\le 58$ and $4\le q\le 97$}
\end{table}

\begin{table}[H]
\scalebox{0.95}{\begin{tabular}{ |c|c|c|c|c|c|c|c|c|c|c|c|c|c|c|}
\toprule
\diagbox{$q$}{$k$} &59  & 60 & 61  & 62  &63 & 64    &65   &66     & 67 & 68 & 69 &70  & 71&72\\
\midrule
4 & 64  &  64  &  &   & &  & &  &72 & 72 &   &  &  &    \\
\hline
5& 64  &  64  &  &   & &  & 70& 70 &72 &  &   &  &  &    \\
\hline
7& 64  &  64  &  &   & &  & 70& 70 &72 &  &   &  &  &   \\
\hline
8 & 64  &  64  &  &   & &  & &  &72 &  &   &  &  &    \\
\hline
9 & 64  &  64  &  &   & &  & &  &72 & 72 &   &  &  &  \\
\hline
11 & 64  &  64  &  &   & &  & &  &72 &  &   &  &  &   \\
\hline
13& 64  &  64  &  &   & &  & &  &72 &  &   &  &  &  \\
\hline
16& 64  &  64  &  &   & &  & &  &72 &  &   &  &  &   \\
\hline
17& 64  &  64  &  &   & 68& 68 & &  &72 &  &   &  &  &    \\
\hline
19& 64  &  64  &  &   & &  & &  & 72&  &   &  &  76&76  \\
\hline
23& 64  &  64  &  &   & &  & &  &72 &  &   &  &  &   \\
 \hline
25 & 64  &  64  &  &   & &  &70 &70  &72 &  &   &  &  &   \\
 \hline
27& 64  &  64  &  &   & &  & &  &72 &72  &   &  &  &    \\
 \hline
29& 64  &  64  &  &   & &  & &  &72 &  &   &  &  &   \\
 \hline
31& 64  &  64  &  &   & &  & &  &72 &  &   &  &  &   \\
 \hline
 32& 64  &  64  &  &   & &  & &  &72 &  &   &  &  & \\
 \hline
37& 64  &  64  &  &   & &  & &  & 72&72  &74   & 74 &  &  \\
 \hline
41& 64  &  64  &  &   & &  & &  &72 &  &   &  &  &     \\
 \hline
43& 64  &  64  &  &   & &  & &  &72 &  &   &  &  & \\
 \hline
47 & 64  &  64  &  &   & &  & &  &72 &  &   &  &  &  \\
 \hline
49& 64  &  64  &  &   & &  &70 &70  &72 &  &   &  &  & \\
 \hline
53& 64  &  64  &  &   & &  & &  &72 &  &   &  &  &   \\
 \hline
 59& 64  &  64  &  &   & &  & &  & 72&  &   &  &  &  \\
 \hline
 61& 64  &  64  &  &   & &  & &  &72 &  &   &  &  &  \\
 \hline
64& 64  &  64  &  &   & &  & &  &72 & 72 &   &  &  &  \\
 \hline
67& 64  &  64  & 66 &66   & &  & &  &72 &  &   &  &  &   \\
 \hline
71& 64  &  64  &  &   & &  & 70&  70& 72&  &   &  &  &   \\
 \hline
 73& 64  &  64  &  &   & &  & &  &72 & 72 &   &  &  &    \\
 \hline
79& 64  &  64  &  &   & &  & &  &72 &  &   &  &  &   \\
 \hline
81& 64  &  64  &  &   & &  & &  &72 & 72 &   &  &  &  \\
 \hline
83& 64  &  64  &  &   & &  & &  &72 &  &   &  &  &   \\
 \hline
89& 64  &  64  &  &   & &  & &  &72 &  &   &  &  &  \\
 \hline
97 & 64  &  64  &  &   & &  & &  & 72&  &   &  &  & \\
 \bottomrule
\end{tabular}}
\caption{Girth of $D(k,q)$ for $59\le k\le 72$ and $4\le q\le 97$}
\end{table}

\begin{table}[H]
\scalebox{0.95}{\begin{tabular}{ |c|c|c|c|c|c|c|c|c|c|c|c|c|c|c|}
\toprule
\diagbox{$q$}{$k$} &73   & 74 & 75  & 76  &77 &78    &79   &80    & 81 & 82 & 83 &84  & 85&86\\
\midrule
4 &   &    &80  &80   & &  & &  & &  & 88  &  &  90& 90   \\
\hline
5 &   &    & 80  &  80 & &  & &  & &  & 88  &  &  90& 90    \\
\hline
7&   &    & 80  &   & &  &84 &84  & &  & 88  &  &  90& 90     \\
\hline
8 &   &    & 80  & 80  & &  & &  & &  & 88  &  &  90& 90    \\
\hline
9 &   &    &  80 &   & &  &84 & 84 & &  & 88  &  &  90& 90   \\
\hline
11 &   &    & 80  &   & &  & &  & &  & 88  &88  &  90& 90   \\
\hline
13 & 78  & 78   &  80 &   & &  & &  & &  & 88  &  &  90& 90   \\
\hline
16 &   &    & 80 & 80  & &  & &  & &  & 88  &  &  90& 90    \\
\hline
17 &   &    & 80  &   & &  & &  & &  &  88 &  &  90& 90    \\
\hline
19 &   &    & 80  &   & &  & &  & &  &  88 &  &  90& 90    \\
\hline
23& &   &    80&  &   & &  & &  & &  88&   &  90& 90   \\
 \hline
25 &   &    & 80  &  80 & &  & &  & &  & 88  &  &  90& 90   \\
 \hline
27&  78 &  78  & 80  &   & &  & &  & &  & 88  &  &  90& 90   \\
 \hline
29&   &    &80   &   & &  & &  & &  &88   &  &  90& 90    \\
 \hline
31&   &    &80   &   & &  & &  & &  &   &  &  90& 90    \\
 \hline
 32&   &    & 80  &  80 & &  & &  & &  &  88 &  &  90& 90    \\
 \hline
37&   &    & 80  &   & &  & &  & &  &   &  &  90& 90   \\
 \hline
41&   &    & 80 & 80  & 82& 82 & &  & &  &   &  &  90& 90    \\
 \hline
43&   &    & 80  &   & &  &84 & 84 &86 & 86 &   &  &  90& 90   \\
 \hline
47 &   &    & 80  &   & &  & &  & &  &   &  & 90& 90    \\
 \hline
49&   &    &80   &   & &  &84 &  84& &  & 88  &  &  90& 90    \\
 \hline
53&   &    &80   &   & &  & &  & &  &   &  &  90& 90  \\
 \hline
 59&   &    & 80  &   & &  & &  & &  &   &  &  90& 90  \\
 \hline
 61&   &    & 80 &   & &  & &  & &  &   &  & 90& 90   \\
 \hline
64&   &    & 80  &  80 & &  &84 &84& &  & 88  &  &  90& 90  \\
 \hline
67&   &    & 80  &   & &  & &  & &  & 88  &  &  90& 90   \\
 \hline
71 &   &    & 80  &   & &  & &  & &  &   &  &  90& 90    \\
 \hline
 73&   &    &80   &   & &  & &  & &  &   &  & 90& 90  \\
 \hline
79& 78  &  78  & 80  &   & &  & &  & &  &   &  &  90& 90   \\
 \hline
81&   &    & 80 &  80 & &  & 84& 84 & &  &88   &  &  90& 90  \\
 \hline
83&   &    & 80  &   &82 & 82 & &  & &  &   &  &  90& 90  \\
 \hline
89 &   &    & 80  &   & &  & &  & &  &  88 &88  & 90& 90  \\
 \hline
97 &   &    & 80  &   & &  & &  & &  &   &  &  90& 90  \\
 \bottomrule
\end{tabular}}
\caption{Girth of $D(k,q)$ for $73\le k\le 86$ and $4\le q\le 97$}
\end{table}

\begin{table}[H]
\scalebox{0.95}{\begin{tabular}{ |c|c|c|c|c|c|c|c|c|c|c|c|c|c|c|}
\toprule
\diagbox{$q$}{$k$} &87   & 88 & 89  & 90   &91 & 92    &93   &94     & 95 &96 &97 &98  & 99&100\\
\midrule
4 &   &    &  &   &96 & 96 & &  & &  &   &  &104  &    \\
\hline
5 &   &    &  &   &96 & 96 & &  &100 & 100 &   &  & 104 &    \\
\hline
7&   &    &  &   &96 & 96 &98 &  98& &  &   &  & 104 &    \\
\hline
8 &   &    &  &   &96 & 96 & &  & &  &   &  & 104 &     \\
\hline
9 &   &    &  &   &96 & 96 & &  & &  &   &  & 104 & \\
\hline
11 &   &    &  &   &96 & 96 & &  & &  &   &  &104  &    \\
\hline
13 &   &    &  &   &96 & 96 & &  & &  &   &  &  104& 104 \\
\hline
16 &   &    &  &   &96 & 96 & &  & &  &   &  & 104 &    \\
\hline
17&   &    &  &   &96 & 96 & &  & &  &   &  &  &    \\
\hline
19&   &    &  &   &96 & 96 & &  & &  &   &  &  &     \\
\hline
23& 92  &   92 &  &   &96 & 96 & &  & &  &   &  &  & \\
 \hline
25&   &    &  &   &96 & 96 & &  & 100&100  &   &  &104  &   \\
 \hline
27&   &    &  &   &96 & 96 & &  & &  &   &  & 104 &    \\
 \hline
29&   &    &  &   &96 & 96 & &  & &  &   &  &  &  \\
 \hline
31&   &    &  &   &96 & 96 & &  & &  &   &  &  &     \\
 \hline
 32&   &    &  &   &96 & 96 & &  & &  &   &  &  &    \\
 \hline
37&   &    &  &   &96 & 96 & &  & &  &   &  &  &   \\
 \hline
41&   &    &  &   &96 & 96 & &  & &  &   &  &  &    \\
 \hline
43&   &    &  &   &96 & 96 & &  & &  &   &  &  &   \\
 \hline
47& 92  &  92  & 94 & 94  &96 & 96 & &  & &  &   &  &  &    \\
 \hline
49&   &    &  &   &96 & 96 &98 & 98 & &  &   &  & 104 &     \\
 \hline
53&   &    &  &   &96 & 96 & &  & &  &   &  & 104 &104   \\
 \hline
 59&   &    &  &   &96 & 96 & &  & &  &   &  &  &    \\
 \hline
 61&   &    &  &   &96 & 96 & &  & &  &   &  &  &   \\
 \hline
64&   &    &  &   &96 & 96 & &  & &  &   &  &104  & \\
 \hline
67&   &    &  &   &96 & 96 & &  & &  &   &  &  & \\
 \hline
71 &   &    &  &   &96 & 96 & &  & &  &   &  &  &    \\
 \hline
 73&   &    &  &   &96 & 96 & &  & &  &   &  &  &    \\
 \hline
79&   &    &  &   &96 & 96 & &  & &  &   &  &104  & \\
 \hline
81&   &    &  &   &96 & 96 & &  & &  &   &  &104  &    \\
 \hline
83&   &    &  &   &96 & 96 & &  & &  &   &  &  &    \\
 \hline
89&   &    &  &   &96 & 96 & &  & &  &   &  &  &   \\
 \hline
97 &   &    &  &   &96 & 96 & &  & &  &   &  &  &  \\
 \bottomrule
\end{tabular}}
\caption{Girth of $D(k,q)$ for $87\le k\le 100$ and $4\le q\le 97$}
\end{table}

\medskip

\subsection{Connectivity of $D(k,q)$}\label{conDkq}
\bigskip

Let $c(G)$ be the number of components of a graph $G$.
In~\cite{LUW95}, Lazebnik,
Ustimenko and Woldar proved that
for $k\ge 6$ and $q$ odd, the graph $D(k,q)$ is disconnected. As the graph $D(k,q)$ is edge-transitive, all components are isomorphic. Let $C\!D(k,q)$ denote one of them.
It was shown in~\cite{LUW95} that $c(D(k,q))\ge q^{t-1}$, where $t=\lfloor \frac{k+2}{4}\rfloor$, and therefore the order of $C\!D(k,q)$ is at most $2q^{k -t +1}$. Moreover, in~\cite{LUW96}, the same authors proved that for all odd $q$, $c(D(k,q)) = q^{t-1}$. It is shown in \cite{LV04} that $c(D(k,q))=q^{t-1}$ for  even $q>4$, $c(D(k,4)) = q^t$ for $k\ge 4$, and $c(D(2,4))=c(D(3,4)) = 1$.\\

In order to characterize the components,  we begin with the notion of an invariant vector of the component (see~\cite{LUW95}).

\subsubsection{Invariant vector}
\bigskip

Let $k\geq 6$ and $t =\lfloor \frac{k+2}{4} \rfloor$. For every point
$(p) = (p_1,\ldots, p_k)$ and every line $[l] = [l_1,\ldots, l_k]$ in $D(k,q)$, and for any $2\le r\le t$,
let $a_r :\mathcal{P}_k\cup\mathcal{L}_k \rightarrow \mathbb{F}_q$ be given by:
\[
a_r((p)) = \begin{cases}
 -p_1p_4 + p_2^2 +p_5-p_6,  &\text{if } r=2, \\
  (-1)^{r-1}[p_1p_{4r-4}-p_2p_{4r-6}-p_2p_{4r-7}+p_3p_{4r-8}-p_{4r-3}+&\\
                                       p_{4r-2}+\sum_{i=2}^{r-2} (-p_{4i-3}p_{4(r-i)-2} + p_{4i-1}p_{4(r-i)-4})], &\text{if } r\geq 3,

                \end{cases}
 \]
and
\[
a_r([l]) =
\begin{cases} -l_1l_3+l_2^2-l_5+l_6, &\text{if }r=2, \\
(-1)^{r-1} [l_1l_{4r-5}-l_2l_{4r-6}-l_2l_{4r-7}+l_3l_{4r-8} +l_{4r-3}-&\\
l_{4r-2}+\sum_{i=2}^{r-2}(-l_{4i-3}l_{4(r-i)-2} +l_{4i-1}l_{4(r-i)-4})], &\text{if } r\geq 3.
\end{cases}
\]
For example,
\[
a_3((p)) = p_1p_8-p_2p_6-p_2p_5+p_3p_4-p_9+p_{10},
\]
and
\[
a_3([l]) = l_1l_7-l_2l_6-l_2l_5+l_3l_4+l_9-l_{10}.
\]

The {\it invariant vector} $\vec{a}(u)$ of  a vertex $u$ is defined to be
\[
\vec{a}= \vec{a}(u) = \langle a_2(u),a_3(u),\ldots,a_t(u)\rangle.
\]
The following proposition justifies the term. The proof is straightforward, and the details can be found in the original paper \cite{LUW95}, and with the representation of $D(k,q)$ adopted in this survey, in \cite{LSW17}.

\begin{theorem}\label{S:INVAR}{\rm{(\cite{LUW95})}}
 If $(p) \sim [l]$, then $\vec{a}((p)) = \vec{a}([l])$.
  \end{theorem}



\begin{corollary}\label{S:RINVA}
All the vertices in the same component of $D(k,q)$
have the same invariant vector.
\end{corollary}

Let $(0)$ denote the point corresponding to the zero vector. By Corollary~\ref{S:RINVA}, and the fact that $\vec{a}((0)) = \vec{0}$, we obtain the following theorem.
\begin{theorem} {\rm {(\cite{LUW95})}}\label{S:invariantForC}
Let $u$ be a vertex  in the component of $D(k,q)$ containing $(0)$. Then $$\vec{a}(u) = \vec{0}.$$
\end{theorem}
A natural question at this point is whether the equality of invariant vectors of two vertices of $D(k,q)$ implies that the vertices are in the same component.  The answer is affirmative for $k\ge 6$ and $q\neq 4$,  and we will discuss it in the following four subsections.  This discussion will also lead to determining the number of components $c(D(k,q))$.
\medskip

\subsubsection{Lower bound for $c(D(k,q))$}
\bigskip

\begin{theorem}\label{S:CLB}{\rm{(\cite{LUW95})}}
For any $k\ge 2$, let
$t=\lfloor \frac{k+2}{4} \rfloor$ and let $q$ be a prime power. Then
\[
c(D(k,q)) \ge q^{t-1}.
\]
\end{theorem}
\begin{proof} Let $x=(x_2,\ldots,x_t)$ and $y=
(y_2,\ldots,y_t)$ be  distinct vectors in $ \mathbb{F}_q^{t-1}$. Consider
points  $(p) = (p_1,\ldots,p_k)$ and
$(p') = (p'_1,\ldots,p'_k)$ defined by:
\[
p_j =
         \begin{cases}  x_{\frac{j-2}{4}},  &\text{if } j \equiv 2 \pmod 4,\cr
                                 0, &\text{otherwise,}
          \end{cases}
   \]
   and
   \[
p'_j =
         \begin{cases}  y_{\frac{j-2}{4}},  &\text{if }j \equiv 2 \pmod 4,\cr
                                 0, &\text{otherwise.}
          \end{cases}
   \]
It is easy to see that $\vec{a}((p)) =x  \neq y= \vec{a}((p'))$, and by Corollary~\ref{S:RINVA},
$(p)$ and $(p')$ are in different components.  So there are at least $q^{t-1}$ components.
\end{proof}
\medskip

\subsubsection{Projections and lifts}
  \bigskip

The results presented in  this section are rather technical. They will allow
us to show in the following  two sections
that equality of invariant vectors of two points of $D(k,q)$
implies that the points belong to the same component of $D(k,q)$. The proofs can be found in
the original papers \cite{LUW96} for $q$ odd, in \cite{LV04} for $q$ even,  and, using the
representation of $D(k,q)$ adopted in this survey, in \cite{LSW17}.

For $k\ge 3$, the \emph{projection}
\[
\pi: V(D(k,q)) \rightarrow V(D(k-1,q))
\]
 is defined via
\[
(p_1,\ldots,p_k) \mapsto (p_1,\ldots,p_{k-1}),\quad\quad [l_1,\ldots,l_k] \mapsto [l_1,\ldots,l_{k-1}].
\]

\noindent
As we mentioned in Section~\ref{SS:cover}, $\pi$ is a graph homomorphism of $D(k,q)$ to $D(k-1,q)$.
The vertex $w=v^\pi \in V(D(k-1,q))$
will often be denoted by $v'$; we say that $v$ is a \emph{lift} of $w$ and $w$ is a \emph{projection} of $v$.  If $B$
is a component of $D(k,q)$, we will often denote $B^\pi$ by $B'$, and $\pi_B$ will denote the restriction
of $\pi$ to $B$. We say that an automorphism $\tau$ $stabilizes$ $B$ if $B^\tau =B$; the group of all such
automorphisms is denoted by $Stab(B)$. A component of $D(k,q)$ containing a vertex $v$ will be denoted
by $C(v)$. The point and line corresponding to the zero vector $\vec{0}$ will be denoted by $(0)$ and
$[0]$, respectively. We will always denote the component $C((0))$ of $D(k,q)$ by just $C$. Then $C'$ will
be the corresponding component in $D(k-1,q)$.

\begin{theorem}\label{S:stab}{\rm{(\cite{LUW96,LV04})}} Let $\tau$ be an automorphism of $D(k,q)$,
and let $B$ be a component of $D(k,q)$ with $v\in V(B)$. Then $\tau$ stabilizes $B$ if and only if $v^\tau\in B$.
In particular, $t_{0,x}$, $t_{1,x}$, and $m_{a,b}$ are in $Stab(C)$ for all $x,a,b\in \mathbb{F}_q$, $a,b\neq 0$.
\end{theorem}

\begin{theorem}\label{S:pro}{\rm{(\cite{LUW96,LV04})}}  Let $B$ be a component of $D(k,q)$. Then $\pi_B$ is a $t$-to-1
graph homomorphism for some $t$, $1\le t\le q$. In particular, let $k\equiv 0,\;3\pmod 4$, and suppose $\pi_C$
is a $t$-to-1 mapping for some $t>1$. Then $t=q$.
\end{theorem}

\begin{theorem}\label{S:surj}{\rm{(\cite{LUW96,LV04})}} The map $\pi_C:V(C)\rightarrow V(C')$ is surjective.
\end{theorem}
\medskip

\subsubsection{Exact number of components of  $D(k,q)$, $q\neq 4$}
\bigskip

The proofs of the results of this section can be found in
the original papers \cite{LUW96} for $q$ odd, in \cite{LV04} for $q$ even,  and, using the
representation of $D(k,q)$ adopted in this survey, in \cite{LSW17}.

\begin{theorem}\label{S:invariantInverseNot4}{\rm{(\cite{LUW96,LV04})}}
Let $q$ be a prime power, $q\neq 4$, and $k\ge 6$. If $v\in V(D(k,q))$ satisfies $\vec{a}(v) = \vec{0}$,
then $v\in V(C)$.
\end{theorem}
\begin{proof}
The proof proceeds by induction on $k$. It is known from \cite{LUW96} that for $q\neq 4$, the graphs $D(k,q)$
are connected for $k=2,3,4,5$.

We begin with the base case $k=6$. Let $v\in V(D(6,q))$ with $\vec{a}(v) = \vec{0}$,
and let $v' = v^\pi \in V(D(5,q))$. Since $D(5,q)$ is connected, then $v'\in C' = D(5,q)$. Since $\pi_C$ is
surjective by Theorem~\ref{S:surj}, there is a $w\in V(C)$ such that $w^\pi = v' = v^\pi$. Since the sixth
coordinate of any vertex $u$ is uniquely determined by its initial five coordinates and $\vec{a}(u)$,
we have $v=w\in V(C)$.\\
Suppose that the theorem is true for $k' < k$, with $k\ge 7$.

If $k\equiv 2\pmod 4$, choose $v\in V(D(k,q))$ with $\vec{a}(v) = \vec{0}$, and let $v'
= v^\pi \in V(D(k-1,q))$. Then $\vec{a}(v') =\vec{0}$. Let $w$ be any lift of $v'$ to $C$. Then $\vec{a}(w)
= \vec{0} =\vec{a}(v)$ and $w^\pi = v' = v^\pi$. This implies that $v=w$, as in the base case $k=6$. Thus
$v\in V(C)$.

If $k \equiv 0,1,3\pmod 4$, we  show that $\pi_C$ is a $q$-to-1 map. (In the case of
$k\equiv 0,3\pmod 4$, it suffices to show that there is a point $(p')\in V(C')$ which has two lifts to $D(k,q)$
in $V(C)$ by Theorem~\ref{S:pro}).  These are exactly the values
of $k$ for which the invariant vectors of $C$ and $C'$ are the same. Choose $v\in V(D(k,q))$ such that $\vec{a}(v)
=\vec{0}$. Let $v' = v^\pi \in V(D(k-1,q))$. Since $\vec{a}(v) = \vec{a}(v') = \vec{0}$, then $v'\in C'$ by the
induction hypothesis. But then since $\pi_C$ is a $q$-to-1 map, all of the lifts of $v'$, including $v$ itself,
lie in $C$, and we are done. So one can  proceed with these cases.
Since the proof in the  case $k\equiv 0 \pmod 4$ is more involved,
we illustrate our approach with this case only.\\

\noindent \textbf{Case} $k\equiv 0 \pmod 4$. Write $k=4j$, $j\ge 2$. Let
\[
(p') = (0,\ldots,0,1,1,0)\in V(D(k-1,q)).
\]
Clearly $\vec{a}(p') =\vec{0}$, so $(p') \in V(C')$ by the induction hypothesis. Since $\pi_{C}$ is surjective, there
is $(p) \in V(C)$ with $(p)^\pi = (p')$, i.e., for some $y\in\mathbb{F}_q$,
\[
(p) = (0,\ldots,0,1,1,0,y).
\]
First suppose that $y\neq 0$. Then
\[
(p)^{m_{a,b}} = (0,\ldots,0,a^jb^j,a^jb^j,0,a^jb^{j+1}y).
\]
One can always choose $a,b\in \mathbb{F}_q^*$ such that $ab = 1$ but $b\neq 1$. With this choice of $a$
and $b$, we have
\[
(p)^{m_{a,b}} = (0,\ldots,0,1,1,0,by) \in V(C)
\]
by Theorem~\ref{S:stab}.
 Since $y\neq 0$, and $b\neq 1$, $(p')$ has two lifts to $D(k,q)$ in $C$. \\
Now suppose that $y=0$, then
\[
(0) \xrightarrow{t_{4j-3,1}} (0,\ldots,0,1,0,0,0) \xrightarrow{t_{4j-2,1}} (p).
\]
Therefore, $t_{4j-3,1}t_{4j-2,1} \in Stab(C)$ by Theorem~\ref{S:stab}. Now let
\[
(p') = (0,\ldots,0,1,1,0,0,0).
\]
Clearly $\vec{a}(p') = \vec{0}$, so $(p') \in V(C')$ by the induction hypothesis. Since $\pi_C$ is surjective, there
is $(p) \in V(C)$ with $(p)^\pi = (p')$, i.e., for some $y\in \mathbb{F}_q$,
\[
(p) = (0,\ldots,0,1,1,0,0,0,y).
\]
Note that
\[
(p) \xrightarrow{t_{1,-1}} (0,\ldots,0,1,1,-1,-1,0,y+1) \xrightarrow{t_{4j-3,1}t_{4j-2,1}} (0,\ldots,0,1,1,0,0,0,y+1).
\]
Since $t_{1,-1}$, $t_{4j-3,1}t_{4j-2,1} \in Stab(C)$ by Theorem~\ref{S:stab}, all the vertices above are in $V(C)$.
 Hence $(p')$ has two lifts to $D(k,q)$ in $C$.
 \medskip

Cases $k\equiv 1,3 \pmod 4$ can be dealt with similarly,  and the details can
be found in the references.
\end{proof}

\begin{theorem}\label{S:numCompNot4}{\rm{(\cite{LV04})}} Let $q$ be a prime power with $q\neq 4$,
$k\ge 2$ be an integer, and $t = \lfloor \frac{k+2}{4} \rfloor$. Then $c(D(k,q)) = q^{t-1}$.
\end{theorem}
\begin{proof} We have already mentioned (see the beginning of the proof of Theorem~\ref{S:invariantInverseNot4})
that for $2\le k \le 5$, and $q\ne 4$, $D(k,q)$ is connected. Hence the statement is correct in these cases. We
also remind the reader that for all $k\ge 2$ and prime powers $q$, $D(k,q)$ is edge-transitive, hence all its components
are isomorphic.

Let $k\ge 6$. Combining Theorem~\ref{S:invariantForC} and Theorem~\ref{S:invariantInverseNot4}, we have that
$v\in V(C)$ if and only if $\vec{a}(v) = \vec{0}$. To determine the number of points in $C$, we need only determine
how many solutions there are to the equation $\vec{a}((p)) = \vec{0}$, or equivalently to the system of equations
$a_r = 0$ for every $r\ge 2$. For $3\le r\le t$, and arbitrary $p_1,\ldots,p_5$, $p_{4r-3}, p_{4r-4},p_{4r-5}$ and $p_{4t-1},
\ldots, p_k$, we can uniquely solve for $p_{4r-2}$ for $2\le r\leq t$. Therefore, there are $q^{5+3(t-2)+k-(4t-2)} = q^{k-t+1}$ points in $C$.

Since the total number of points in $D(k,q)$ is $q^k$, and all its components are isomorphic, we have
\[
c(D(k,q)) = \frac{q^k}{q^{k-t+1}} = q^{t-1}.
\]
\end{proof}
We will show that the invariant vector of a vertex characterizes the component containing the vertex. Let $C(x)$ be the component of $D(k,q)$ containing the vertex $x$.
\begin{corollary}\label{S:character}{\rm{(\cite{LV04})}}
Let $k\ge 6$, and $q\neq 4$. Then $\vec{a}(u) = \vec{a}(v)$ if and only if $C(u) = C(v)$.
\end{corollary}
\begin{proof}
Let $t = \lfloor \frac{k+2}{4} \rfloor$.
 Let $X$  be the set of all components of $D(k,q)$
 and we define the mapping
$f: X\to  \mathbb{F}_q^{t-1}$ via $f(C(v)) = \vec{a}(v)$.
From Theorem~\ref{S:INVAR}, we know that $f$ is well defined, i.e., $C(u) = C(v)$ implies $\vec{a}(u) = \vec{a}(v)$.
By Theorem~\ref{S:numCompNot4}, $|X| = q^{t-1}$, so that $f$ is bijective. Thus $C(u) =C(v)$ whenever $\vec{a}(u) = \vec{a}(v)$.
\end{proof}
\medskip

\subsubsection{Exact number of components of $D(k,4)$}
\bigskip

In order to deal with the case $q=4$, we will need an analog of
Theorem~\ref{S:invariantInverseNot4}. We begin by
defining an invariant vector for $D(k,4)$. Its definition is very close to $\vec{a}$
defined before, the only difference being
the presence of an extra coordinate.
For $u\in V(k,4)$, and $t =\lfloor \frac{k+2}{4}\rfloor$, the invariant is given by
\[
\vec{b} = \vec{b}(u) = \langle b_1(u),b_2(u),\ldots,b_t(u)\rangle,
\]
where $b_i = a_i$ for all $i\ge 2$ and
\[
b_1((p)) = p_1p_2 + p_3 + p_4^2, \quad
b_1([l]) = l_1l_2+l_3^2+l_4.
\]

The following three theorems  are analogs to
Theorems \ref{S:invariantForC},
\ref{S:invariantInverseNot4} and \ref{S:numCompNot4} for $q=4$.
Their  proofs are similar to those,  and can be found in
\cite{LV04}, or, using the presentation of $D(k,q)$   adopted   in this survey,
 in \cite{LSW17}.
\begin{theorem}\label{S:invariantCq4}{\rm{(\cite{LV04})}}
 Let $u$ be in the component of $D(k,4)$ containing $(0)$. Then $$\vec{b}(u) = \vec{0}.$$
\end{theorem}

\begin{theorem}\label{S:invInverse4}{\rm{(\cite{LV04})}}
  Let $k\ge 4$. If $v\in V(D(k,4))$ satisfies $\vec{b}(v) = \vec{0}$, then $v\in V(C)$.
\end{theorem}

Similarly to the proof of Theorem~\ref{S:numCompNot4}, one can show the following result.

\begin{theorem}\label{S:numComp4}{\rm{(\cite{LV04})}}
Let $k\ge 4$ and $t=\lfloor \frac{k+2}{4} \rfloor$. Then $c(D(k,4)) = 4^t$ and  $c(D(2,4)) = c(D(3,4)) = 1$.
 \end{theorem}

 \begin{remark} The analog of Corollary~\ref{S:character} does not hold for $q=4$. The reason for this is
 the special first coordinate of the invariant vector.
 Indeed, let $\omega$ be a primitive element of $\mathbb{F}_4$. Then $$(p) = (0,0,\omega,0,\ldots,0)\sim [0,0,\omega,
 0,\ldots,0] = [l],$$ in $D(k,4),$ but
 \[
 \vec{b}((p)) = \langle \omega, 0, \ldots, 0\rangle \neq \langle \omega^2,0,\ldots,0\rangle = \vec{b}([l]).
 \]
 \end{remark}
\medskip

\subsection{Diameter of $C\!D(k,q)$}
\bigskip

Let $\d(C\!D(k,q))$ denote the diameter of the graph $C\!D(k,q)$.
For small values of $k$ and $q$, we have the following computational results (\cite{Sch94,Tho97,Tho02}).

\noindent For $k=2$, the diameter of $C\!D(2,q)$ is 4 for $3\le q\le  49$.\\
For $k=3$, the diameter of $C\!D(3,q)$ is 6 for $3\le q \le 47$.\\
For $k\ge 4$ and the following pairs $(k,q)$, we have:\\
\vskip 3mm
\begin{table}[H]
\centering
\begin{tabular}{|c|c|c|c|c|c|c|c|}
\toprule
$k$ & 4 & 5,6 & 5, 6 & 7 & 8 & 9,10 & 11,12 \\
\hline
$q$ & 3,[5,23] & 5 & [7,13] & 5,[7,9] & 5,7 &  5 & 5\\
 \hline
diameter & 8 & 12& 10 & 12 & 12 & 14  & 16 \\
\bottomrule
\end{tabular}
\end{table}

\noindent
As mentioned before, $[m,n]$ denotes the set of all  integers $k$ such that $m\le k\le n$.
For $q=3$ and $q=4$, the diameter exhibits different behavior.

\begin{table}[H]
\centering
\begin{tabular}{ |c|cccccccccccccc|}
\toprule
$k$ &2   & 3 & 4   & 5   & 6 & 7     &8    &9     & 10 & 11 & 12 &13  & 14 & 15\\
\hline
diameter & 4 & 6 & 8 & 12 & 12 & 12 & 14 & 17 & 17 & 22 & 22 & 24 & 24 & 26\\
\bottomrule
\end{tabular}
\caption{Diameter of $C\!D(k,3)$ for small $k$ }
\end{table}

\begin{table}[H]
\centering
\begin{tabular}{ |c|ccccccccccc|}
\toprule
$k$ &2   & 3 & 4   & 5   & 6 & 7     &8    &9     & 10 & 11 & 12 \\
\hline
diameter & 4 & 6 & 6 &8 &8 &10 & 12 &16 &16 &16 &18\\
\bottomrule
\end{tabular}
\caption{Diameter of $C\!D(k,4)$ for small $k$ }
\end{table}

\begin{conjecture}{\rm{(\cite{LUW97})}}
There exists a positive constant $C$ such that for all $k\ge 2$, and all prime powers $q$,
\[
\d(C\!D(k,q)) \le (\log_{q-1}q)k+C.
\]
\end{conjecture}

The following conjecture was stated in \cite{Sch94}.
\begin{conjecture}\label{S:diaExact}{\rm{(\cite{Sch94})}}
The diameter of $C\!D(3, q)$ is 6 for all prime powers q. The diameter of $C\!D(k, q)$ is $k +5$,
if $k > 3$ is odd, and $k + 4$, if $k$ is even, provided that q is a large enough prime power.
\end{conjecture}

Some parts  of Conjecture~\ref{S:diaExact} were proved in~\cite{Sch94}, namely for $k=3$ and all odd prime powers $q$, and for $k=4$ and prime power $q$ satisfying  the following three conditions:
 $q$ is odd,   $(q -1, 3) = 1$,  and either 5
is a square in $\mathbb{F}_q$ or $z^4 - 4z^2 -z +1 = 0$ has a solution in $\mathbb{F}_q$.
For the lower bound of the diameter, it is shown in \cite{Sch94} that
for all odd $k \ge5$ and all prime powers $q$, $\d(C\!D(k, q)) \ge k+3$.
Then this bound was improved by \cite{Sun17}:  for all prime powers $q\ne 4$,
$\d(C\!D(k,q)) \ge k+5$ for odd $k\ge 5$, and $\d(C\!D(k,q))\ge k+4$ for even $k\ge 4$.
\medskip

\subsection{Spectrum of $D(k,q)$}
\bigskip

We would like to end this section with a problem about the spectra of the graphs $D(k,q)$,  which have the same eigenvalues
as the graphs $C\!D(k,q)$,  but with higher multiplicities.   In particular,  we wish to find the second largest eigenvalue $\lambda_2$ for these graphs,
defined as  the largest eigenvalue smaller than $q$.  Though it is known to have a relation to the  diameter of
$C\!D(k,q)$,  $\lambda_2$  is also related to other properties of these graphs, including
the expansion properties (see Hoory, Linial and Wigderson \cite{HLW06} on such relations).  It is known that for some $(k,q)$, the graphs
$D(k,q)$ are  Ramanujan,  i.e., they have   $\lambda_2\le  2\sqrt{q-1}$. In particular, it is the case for  $k=2,3$, see    Li, Lu and Wang \cite{LLW09}   or \cite{CLL14}, where the whole spectrum was determined.
However,  Reichard \cite{Rei01} and Thomason  \cite{Tho02} independently showed by computer that the graphs $D(4,q)$ are not Ramanujan for certain $q$. This implies  that for the same
$q$, the graphs $D(k,q)$ are not Ramanujan for $k\ge 4$, since the spectrum of $D(4,q)$   is embedded in that of $D(k,q)$, for $k\ge 4$.    Later these computations were extended and confirmed by other researchers.  At the same time, we are not aware of any example of
$D(k,q)$ with $\lambda_2 >  2\sqrt{q-1} +1$.
The following bound  appears in  Ustimenko \cite{Ust03} (see  Moorhouse, Sun and Williford \cite{MSW17} for discussion and details).
\begin{conjecture}\label{Ustconj} \rm{(Ustimenko)}
 For all $(k,q)$,  $C\!D(k,q)$ has second largest eigenvalue less than or equal to $2\sqrt{q}$.
\end{conjecture}
 For $k=4$, this conjecture was proven in \cite{MSW17}, where the whole spectrum  of the graph $D(4,q)$ was determined.    For $k=5$ and odd $q$,  the conjecture was proven by Gupta and Taranchuk \cite{GupTar24}.  Both papers \cite{MSW17} and \cite{GupTar24} utilized group representations.
\begin{problem}
Determine a good upper bound on $\lambda_2 (D(k,q))$ for $k\ge 7$,  or find the spectrum of  $D(k,q)$ for $k\ge 5$.
\end{problem}
\medskip

\subsection{Graphs  $B\Gamma_n$ and $C\!D(k,q)$ as iterated voltage lifts}\label{voltagelift}
\bigskip

Here we describe how one can view the graphs $B\Gamma_n = B\Gamma_n(R; f_2, \ldots , f_n)$ and $C\!D(k,q)$, $n\ge 2$,   as the result of iterated voltage lifts. The connection for  $C\!D(k,q)$ was described  in \cite{Ers17}, and our presentation here is close to that in \cite{Ers17}.

Let $\Gamma$  be an undirected graph,  with  loops and multiple edges  allowed.
We replace each edge of $\Gamma$ by a pair of oppositely directed arcs, and denote the resulting digraph  by $ \overrightarrow{\Gamma} $.   If $e$ is an arc,  then $e^{-1}$  will denote its reverse. If $D( \overrightarrow{\Gamma})$   is the  set of all arcs of $ \overrightarrow{\Gamma}$, then $|D( \overrightarrow{\Gamma})| = 2 |E(\Gamma)|$.

For a finite  group $G$, a mapping $\alpha:  D( \overrightarrow{\Gamma})\to G$ is called a {\it voltage assignment}  on $ \overrightarrow{\Gamma}$  if
$\alpha(e^{-1})  = (\alpha (e))^{-1}$ for all $e\in D( \overrightarrow{\Gamma})$. Given a voltage assignment $\alpha$,  we denote  the
{\it voltage lift of $ \overrightarrow{\Gamma}$ corresponding to $\alpha$} as a digraph  $ \overrightarrow{\Gamma}^{\alpha}$ defined as  follows.

The vertex set $V$ of  $\overrightarrow{\Gamma}^{\alpha}$ is the Cartesian product  $V(\Gamma) \times G$, and the arc set  is
$D(\overrightarrow{\Gamma})\times G$.
Let $e$ be  an arc in $\overrightarrow{\Gamma}$  from vertex $u$ to vertex $v$. We define the arc $(e, g)$ in $\overrightarrow{\Gamma}^{\alpha}$ to have
initial vertex $(u,g)$ and terminal vertex $(v, g\alpha(e))$.      Note that by the definition  of
voltage assignments, arc $(e^{-1},g^{-1})$ in  $\overrightarrow{\Gamma}^{\alpha}$ has  initial vertex $(v, g\alpha(e))$  and terminal vertex $(u,g)$.
Replacing  each of  these pairs of arcs with an undirected edge joining $u$ and $v$,   we obtain an undirected  graph that we denote by ${\Gamma}^{\alpha}$ and call the {\it  voltage lift of $\Gamma$ corresponding to $\alpha$}.

Now we explain how the graph $B\Gamma_n = B\Gamma_n(R; f_2, \ldots , f_n)$  can be viewed as a voltage lift of $B\Gamma_{n-1} = B\Gamma_{n-1}(R; f_2, \ldots , f_{n-1})$.  For this, we  define a voltage assignment $\alpha$  so that $B\Gamma_n  = B\Gamma_{n-1}^\alpha$.   Let
$G = (R, +)$  be the additive group of the ring $R$,  and let us define $B\Gamma_1$ to be the complete bipartite graph $K_{q,q}$.
Vertices in one partition can be denoted by $(p_1)$ (points)  and in the other one by  $[l_1]$ (lines),  where all distinct  $p_1$  and all distinct  $ l_1$  range over $G$.
We can  view  vectors $(a_1,a_2, ..., a_{k-1},a_k)\in R^k$ as ordered pairs  $((a_1,a_2, ..., a_{n-1}), a_n) \in R^{n-1}\times G$.  If   $ (p_1, p_2, ..., p_{n-1})$ and  $[l_1,l_2, ..., l_{n-1}]$ form  an edge $e$  in $B\Gamma_{n-1}$,  and $\alpha (e)$ is defined as $ f_k(p_1,\ldots,p_{n-1}, l_1, \ldots l_{n-1})$,  then
 $(p_1, p_2, ..., p_{n-1},p_n)$  and  $[l_1,l_2, ..., l_{n-1} ,l_n]$ form an edge in $B\Gamma_n$.

Now we explain how the graph $C\!D(k,q)$ can be viewed as a voltage lift of $C\!D(k-1,q)$.   Let   $ G= (\F_q, +)$  be the additive group of the field $\F_q$,  and let us define $C\!D(1,q)$ to be the complete bipartite graph $K_{q,q}$. By  the argument above, we conclude that
if  $2\le k\not=  4n +2$, $n\ge 1$, then  $C\!D(k,q)$ is a voltage lift of $C\!D(k-1,q)$.
If $k= 4n+2 \ge 6$, then we know that $D(k,q)$ has $q$ times more components than $D(k-1,q)$.
Each is denoted by $C\!D(k,q)$ and is isomorphic to $C\!D(k-1,q)$.
\begin{theorem}\label{voltage} \rm(\cite{Ers17})  Let $2\le k\not = 4n+2, n\ge 2$, and $q$ be a prime power.
For $\Gamma = C\!D(k-1,q)$, there exists a voltage assignment $\alpha:  D(\overrightarrow{\Gamma}) \to ( \F_q, +) $  such that the lifted graph $\Gamma^\alpha$ is isomorphic to $C\!D(k, q$).
\end{theorem}
\medskip

\section{Applications of graphs $D(k,q)$ and $C\!D(k,q)$}\label{appldkq}

\subsection{Bipartite graphs of given bi-degree
and girth}\label{SS:rsg}
\bigskip

A bipartite graph $\Gamma$
with bipartition $V_1 \cup V_2$
is said to
be {\it biregular}
if there exist integers $r$, $s$ such that
deg($x$)=$r$ for all $x\in V_1$
and deg($y$)=$s$ for all $y\in V_2$.
In this case, the pair  $r,s$ is called
the {\it bi-degree} of $\Gamma$.
By an {\it $(r,s,t)$-graph} we shall mean
any biregular graph with bi-degree $r,s$ and girth
exactly $2t$.

For which $r,s,t \ge 2$ do
$(r,s,t)$-graphs exist?
Trivially, $(r,s,2)$-graphs exist
for all $r,s\ge 2$; indeed, these
are the complete bipartite graphs.
For all $r,t \ge 2$,
Sachs~\cite{Sac63}, and Erd\H os and
Sachs~\cite{ES63}, constructed
$r$-regular graphs
with girth $2t$.
From such graphs,
$(r,2,t)$-graphs can be
trivially obtained by
subdividing (i.e.~inserting a new
vertex on) each edge of the original
graph.

By explicit construction, $(r,s,t)$-graphs exist for all $r,s,t\ge 2$, see \cite{FLSUW95}.
The results can be viewed as
biregular versions of the results from
\cite{Sac63} and \cite{ES63}.
The paper \cite{FLSUW95} contains
two constructions: a {\it recursive} one
and an {\it algebraic} one.
The recursive construction
establishes existence for all
$r,s,t\ge 2$, but the
algebraic method works only for
$r,s \ge t$.
However, the graphs obtained by
the algebraic method
are much denser and exhibit
the following nice
property:
one can construct
an $(r,s,t)$-graph $\Gamma$
such that
for all $r\ge r'\ge t\ge 3$ and $s\ge s'\ge t\ge 3$,
$\Gamma$ contains an $(r',s',t)$-graph
$\Gamma'$ as an induced subgraph.
\medskip

\subsection{Cages}\label{SS:cages}
\bigskip

Let $k\ge 2$ and $g\ge3$ be integers.
A $(k,g)$-graph is a $k$-regular graph
with girth
$g$. A $(k,g)$-{\it cage} is a
$(k,g)$-graph of minimum order.
The problem of determining the order
$\nu(k,g)$ of a $(k,g)$-cage is unsolved for
most pairs $(k,g)$ and is extremely hard in the general
case. For the state of the survey on cages,  we refer the reader to
Exoo and  Jajcay \cite{EJ13}.

In \cite{LUW97}, Lazebnik, Ustimenko and Woldar
established general upper bounds for $\nu(k,g)$ which are
roughly the 3/2 power of the lower bounds (the previous results had  upper bounds roughly equal to the squares
of the lower bounds),
and provided explicit constructions for such
$(k,g)$-graphs. The main ingredients of their construction were
the graphs $C\!D(n,q)$ and
certain induced subgraphs of these,
manufactured
by the method described in Section~\ref{SS:indsub}.
The precise result follows.

\begin{theorem}\label{T:cages}\rm{(\cite{LUW97})}
Let  $k\ge 2$ and $g\ge 5$ be
integers, and let
$q$ denote the smallest odd prime power
for which $k\le q$.
Then
   \begin{equation}
       \nu(k,g)\le 2kq^{\frac{3}{4}g - a},\notag
   \end{equation}
where $a = 4$, $\frac{11}{4}$, $\frac{7}{2}$, $\frac{13}{4}$ for
$g\equiv  0, 1, 2, 3 \pmod{4}$,  respectively.
\end{theorem}
\medskip

\subsection{Structure of extremal graphs of large girth}
\label{SS:struc}
\bigskip

Let $n\ge 3$,
and let $\Gamma$ be a graph of order $\nu$ and
girth at least $n+1$
which has the greatest number of edges
possible subject to these requirements
(i.e.~an extremal graph).
Must $\Gamma$ contain an
$(n+1)$-cycle?
In~\cite{LW97}, Lazebnik and Wang
present several results
where  this question is answered affirmatively,
see also Garnick and Nieuwejaar \cite{GN92}.
In particular,  this is always
the case when
$\nu$ is large compared to $n$:
$\nu \ge 2^{a^2 + a +1} n^a$, where $a = n-3 -
\lfloor{\frac{n-2}{4}\rfloor}$, $n\ge 12$.
To obtain this result they used certain
generic properties of
extremal graphs, as well as of
the graphs $C\!D(k,q)$.
\medskip

\section{Applications to coding theory and cryptography}\label{S:codecryp}
\bigskip

Graphs with many edges and without short cycles have been used in coding theory in the construction and analysis of Low-Density Parity-Check (LDPC) codes,
see, e.g.,  Kim, Peled, Pless and Perepelitsa \cite{KPPP02},  Kim, Peled, Pless, Perepelitsa and  Friedland \cite{KPPPF04},  Kim, Mellinger and Storme \cite{KMS07},  Sin and Xiang \cite{SX06},  Pradhan, Thangaraj and Subramanian \cite{PTS16}.  For the last sixteen years,  V. A. Ustimenko and his numerous collaborators and students have been applying algebraically defined graphs and digraphs to coding theory and cryptography.  We mention just a few recent papers,  and many additional references can be found therein:  Klisowski and Ustimenko \cite{KU12},
Wr\'{o}blewska and Ustimenko \cite{WU14} and Ustimenko  \cite{Ust15}.

A preprint by  Chojecki, Erskine, Tuite and Ustimenko \cite{CETU24}  contains
graph constructions similar to $D(k,q)$. The authors present many computational results concerning their parameters, state several  conjectures, and
mention possible applications of
these graphs and related groups to LDPC codes, noncommutative cryptography
and stream ciphers design.

\medskip

\section*{Acknowledgments}
\bigskip

The authors are grateful to Shuying Sun for helpful discussions,
and to Grahame Erskine, Alex Kodess, Brian Kronenthal, Meng Liu, Ben Nassau, Vladislav Taranchuk, James Tuite and Jason Williford for valuable suggestions. This work was partially supported by NSFC ($\#$12471323, Ye Wang).

\end{document}